\documentclass[11pt]{article}
\usepackage{mthesis}
\usepackage[utf8]{inputenc}
\usepackage{amsmath,amssymb,amsthm}
\usepackage{mathtools}
\usepackage{enumerate}
\usepackage{tikz}
\usepackage{lipsum}
\newcommand{\q}[1]{``#1''}

\newtheorem{theorem}{Theorem}[section]
\newtheorem{lemma}[theorem]{Lemma}
\newtheorem{definition}[theorem]{Definition}
\newtheorem{example}[theorem]{Beispiel}

\newtheorem{notation}[theorem]{Notation}
\newtheorem{remark}[theorem]{Remark}
\newtheorem{corollary}[theorem]{Corollary}
\newtheorem{theorem/definition}[theorem]{Satz/Definition}
\newtheorem{proposition}[theorem]{Proposition}
\newtheorem*{setting}{Setting}
\newtheorem*{convention}{Convention}

\newcommand{\R}{\mathbb R}

\newcommand{\N}{\mathbb{N}}

\renewcommand{\L}{{\cal L}}
\newcommand{\C}{{\cal C}}

\newcommand{\op}{\operatorname}
\newcommand{\supp}{\operatorname{supp}}

\newcommand{\norm}[1]{\left\lVert#1\right\rVert}
\newcommand{\ep}{\varepsilon}
\renewcommand{\(}{\left(}
\renewcommand{\)}{\right)}

\newcommand{\mres}{\mathbin{\vrule height 1.6ex depth 0pt width
0.13ex\vrule height 0.13ex depth 0pt width 1.3ex}}
\newcommand{\Id}{\operatorname{Id}}

\numberwithin{equation}{section}

\begin{document}
\selectlanguage{english}

\title{Optimal transport and regularity of weak Kantorovich potentials on a globally hyperbolic spacetime}
\author{{\sc Alec Metsch$^{1}$} \\[2ex]
      $^{1}$ Universit\"at zu K\"oln, Institut f\"ur Mathematik, Weyertal 86-90, \\
      D\,-\,50931 K\"oln, Germany \\
      email: ametsch@math.uni-koeln.de \\[1ex]
      }
%\date{}

\maketitle

\begin{abstract}
\noindent
We consider the optimal transportation problem on a globally hyperbolic spacetime for some cost function $c_2$, which corresponds to the optimal transportation problem on a complete Riemannian manifold where the cost function is the Riemannian distance squared. Building on insights from previous studies on the Riemannian and Lorentzian case, our main goal is to investigate the regularity of $\pi$-solutions (weak versions of Kantorovich potentials), from which we can conclude, in a classical way, the existence, uniqueness and structure of an optimal transport map between given Borel probability measures $\mu$ and $\nu$, under suitable assumptions.
\end{abstract}

\section{Introduction}

The optimal transportation problem, originally due to Monge \cite{Monge}, is the problem of minimizing the transport cost of a transport map between two given mass distributions.
That is, given two measurable spaces $X$ and $Y$, let $c:X\times Y\to [0,\infty]$ be a measurable function, and let $\mu$ and $\nu$ be probability measures on $X$ and $Y$. Then one is interested in minimizers for the cost
\begin{align}
    \inf\bigg\{\int_{X} c(x,T(x))\ d\mu(x)\mid T:X\to Y \text{ measurable, } T_\#\mu=\nu\bigg\}, \label{Mongee}
\end{align}
where $T_\#\mu$ denotes the push-forward measure of $\mu$, defined by $T_\#\mu(B)=\mu(T^{-1}(B))$ for all measurable sets $B\subseteq Y$. This problem can be ill-posed in the sense that there does not exist any measurable map $T$ with $T_\#\mu=\nu$. A simple example is when $\mu$ is a Dirac-measure, but $\nu$ is not. For this reason, one often studies the following relaxed problem, proposed by Kantorovich in \cite{Kantorovich1}, \cite{Kantorovich2}:
\begin{align}
\inf\bigg\{\int_{X\times Y} c(x,y)\, d\pi(x,y)\mid \pi\in \Gamma(\mu,\nu)\bigg\}, \label{total}
\end{align}
where $\Gamma(\mu,\nu)$ stands for the set of all couplings of $\mu$ and $\nu$, i.e.\ all probability measures $\pi$ on $X\times Y$ such that the first (resp.\ second) marginal of $\pi$ is $\mu$ (resp.\ $\nu$). This formulation is a generalization in the sense that any map $T$ as above gives rise to a coupling. In addition, $\Gamma(\mu,\nu)$ is never empty, as it always contains the product measure. Moreover, assuming that $X,Y$ are complete, separable metric spaces equipped with the corresponding Borel $\sigma$-algebras, there always exists a minimizer under very mild conditions on $c$ (see \cite{Ambrosio} or Theorem \ref{1}).

Let us return to Monge's formulation. In the case that the problem is not ill-posed, one is interested in the existence (and uniqueness) of an optimal (transport) map, that is, a measurable map $T$ that minimizes \eqref{Mongee}.
In general, this is a complex problem, and it is not always true that such a map exists.

A result, known as Brenier's Theorem \cite{Brenier}, states that in the case where $X=Y=\R^n$ (with the Borel $\sigma$-algebra), $c(x,y)=|x-y|^2$ and $\mu$ does not give mass to $(n-1)$-rectifiable sets, an optimal transport map exists and is unique, assuming that $\mu$ and $\nu$ have finite second moments. In particular, there exists an optimal transport map when $\mu$ is absolutely continuous w.r.t.\ the Lebesgue measure. More results are known in $\R^n$ (under suitable assumptions on the measures) when $c(x,y)=h(x-y)$ with $h$ strictly convex (\cite{Ambrosio}, Theorem 6.2.4), and also in the more difficult case when $c(x,y)=|x-y|$ there are positive results \cite{Ambrosio/Pratelli}. In all these cases, the proofs rely on Kantorovich's formulation, first showing the existence of a minimizer - enjoying additional properties in \cite{Ambrosio/Pratelli} - and then proving that this minimal coupling is actually induced by a map and possibly unique.

Since connected Riemannian manifolds $M$ are equipped with a distance function, it is natural to consider the case $X=Y=M$, where $M$ is a connected Riemannian manifold, and the cost function is given by $c=d^2$, $d$ denoting the Riemannian distance. It was first proved by McCann \cite{McCann} in the case when $M$ is compact that there exists a unique optimal map, provided $\mu$ is absolutely continuous w.r.t.\ the Lebesgue measure (or volume measure) on $M$. The case where $M$ is non-compact was (for example) treated by Fathi and Figalli \cite{Fathi/Figalli}, who deal with a complete and connected manifold.

In this paper we are studying the Monge-Kantorovich problem and Kantorovich potentials on a globally hyperbolic $(n+1)$-dimensional spacetime $(M,g)$ for the Lorentzian cost function 
\begin{align}
c_2:M\times M\to [0,\infty],\  
c_2(x,y):=
\begin{cases}
(\tau(y)-\tau(x)-d(x,y))^2,\ &(x,y)\in J^+,
\\\\
\infty,\ &(x,y)\notin J^+. \label{haihduaisoafafdfsfdfsfdfdsfsfdsfdsfs}
\end{cases}
\end{align}
Here, $d$ denotes the Lorentzian distance function (or time separation), $\tau$ is a splitting (or time function) satisfying the growth condition \eqref{splitting} and $J^+$ denotes the set of all causally related points in $M\times M$.

Observe that in Lorentzian geometry one is interested in maximizing the length functional. Since the optimal transportation problem is usually stated as a minimization problem, the minus sign appears in the cost functions.

At this point, let us reference the existing literature on this topic. The optimal transportation problem for a Lorentzian cost function and measures concentrated on spacelike hyperplanes in Minkowski space was initially proposed by Brenier \cite{Brenierrel}. Based on this idea, Bertrand and Puel \cite{Bertrand/Puel} studied this problem for certain cost functions on $\R^n$, which include the relativistic heat cost. They proved, under appropriate regularity assumptions on the measures, that there exists a unique optimal coupling (i.e.\ a coupling that minimizes Kantorovich's formulation) and that this coupling is induced by a map, provided the total cost is finite. This study was extended by the same authors, along with Pratelli in \cite{Bertrand/Puel/Pratelli}, and by Louet, Pratelli and Zeisler in \cite{Louet/Pratelli/Zeisler}, where they proved the existence of a Kantorovich potential for \q{superciritical speeds}, and that optimal couplings (up to negligibe sets) transport mass along timelike geodesics. The notion of $s$-Lorentz Wasserstein distance on a general spacetime $M$ was first proposed by Eckstein and Miller in \cite{EckStein/Miller}. For $s=1$, up to a sign, it is defined as the optimal transportation cost with the negative Lorentzian distance serving as cost function. With a slightly different cost function, namely
\begin{align}
M\times M\to \R\cup \{\infty\},\ (x,y)\mapsto
\begin{cases}
-d(x,y), &(x,y)\in J^+,
\\\\
\infty, &(x,y) \notin J^+. \label{ioaufhdoahfoaidai}
\end{cases}
\end{align}
 (observe that the transport cost agrees for causally related (see Definition \ref{x11}) measures), the results obtained in \cite{Bertrand/Puel}, \cite{Bertrand/Puel/Pratelli}, \cite{Louet/Pratelli/Zeisler}  were generalized by Suhr \cite{Suhr} and Suhr and Kell \cite{Kell} to the setting of a globally hyperbolic Lorentz-Finsler spacetime. The authors proved that, under suitable assumptions on the measures, there exists a unique optimal coupling and it is induced by a transport map. In \cite{Kell}, conditions were also provided under which a (weak) dual solution for the optimal transportation problem (in the sense of  Kantorovich potentials) exists. In this spirit, they generalized the condition of superciritical speed to the condition of strict timelikeness.
McCann \cite{McCann2} considered, for $q\in (0,1)$, the cost function
\begin{align}
    M\times M\to \R\cup\{\infty\},\ (x,y)\mapsto
\begin{cases}
-\frac{1}{q}d(x,y)^q, &(x,y)\in J^+, \label{q}
\\\\
\infty, &(x,y) \notin J^+,
\end{cases}
\end{align}
on a globally hyperbolic spacetime.
He characterized, with the help of optimal transport theory, lower Ricci curvature bounds in timelike directions through displacement convexity properties of the Boltzmann-Shannon entropy functional along $q$-geodesics. In this way, McCann extended the results on the characterisation of lower  Ricci curvature bounds on Riemannian manifolds to the Lorentzian case. Independently and around the same time, Mondino and Suhr \cite{Mondino/Suhr} studied the same cost function and provided a formulation of the Einstein equation in terms of convexity properties of the Boltzmann-Shannon entropy functional along regular displacement interpolations, i.e.\ certain interpolations that arise from exponentiating the gradient of a smooth Kantorovich potential. In the paper, the authors also provide a characterisation of upper (and lower) timelike Ricci curvature bounds through concavity (convexity) properties of the Entropy functional along some (any) regular displacement interpolation, leading to the notion of sysnthetic timelike Ricci curvature bounds on measured Lorentzian pre-length spaces.

Let us also mention the work by Cavalletti and Mondino \cite{Mondino/Cavalletti}, who studied the cost function \eqref{q} on measured Lorentzian pre-length spaces. In the first part of this work, the authors give a detailed overview of the general theory of optimal transportation in their setting, generalizing results known for better behaved (in particular, real-valued) cost functions to the Lorentzian case. For example, the authors prove (under suitable assumptions) that cyclical monotonicity implies optimality, optimal couplings are stable under narrow convergence and that cyclical monotonicity w.r.t.\ the Lorentzian distance implies strong Kantorovich duality, i.e.\ the existence of a maximizing pair in the dual optimal transport formulation. This study bears similarities to our work, as explained in the paragraph after Proposition \ref{abc}.  In the second and third part of the paper, which do not have parallels with our work, building up on the results of McCann about the characterisation of lower Ricci curvature bounds in terms of convexity properties along $q$-geodesics, Cavalletti and Mondino define the notion of lower Ricci curvature bounds on a measured Lorentzian pre-length space, extending the theory for measured metric spaces having Ricci curvature bounded from below, pioneered by Sturm \cite{Sturm1}, \cite{Sturm2} and Lott, Villani \cite{Lott/Villani}.

We focus on Borel probability measures $\mu,\nu$ on $M$ such that $\tau\in L^2(\mu)\cap L^2(\nu)$, and which are causally related, meaning that there exists a coupling $\pi\in \Gamma(\mu,\nu)$ on $M\times M$ supported on $J^+$ (which amounts to the idea that the mass from $\mu$ to $\nu$ can be transported along non-spacelike curves, see also \cite{Kell}, \cite{Mondino/Cavalletti}). These two assumptions imply that the total cost of the Kantorovich minimization problem,
\begin{align}
\inf\bigg\{\int_{M\times M} c_2(x,y)\, d\pi(x,y)\mid \pi\in \Gamma(\mu,\nu)\bigg\}, \label{adjgjzgjjgj}
\end{align}
is finite.
A stronger condition than being \q{causally related} is the requirement that $\mu$ and $\nu$ are strictly timelike \cite{Kell}, meaning that there exists a causal coupling supported on $I^+$ (the set of all chronologically related points).
Roughly speaking, this means that there is a way to transport all the mass from $\mu$ to $\nu$ along timelike curves (but this does not have to be the optimal coupling).

As already mentioned, one of the main questions in the theory of optimal transport concerns the existence and uniqueness of an optimal transport map under suitable assumptions on the measures. For this task, Kantorovich potentials proved to be a powerful tool.
In the case of a real-valued and lower semi-continuous cost function, there are general existence results for Kantorovich potentials (see for example \cite{Ambrosio}, Theorem 6.1.4, for the Rockafellar construction). However, the proofs make crucial use of the fact that $c$ is finite and it is not clear why these results extend to the case where the cost function also attains the value $\infty$, as in the case we are interested in.  Our first result deals with the existence of a weaker version of Kantorovich potentials (called $\pi$-solution, see Definition \ref{daui9dhoafzhiofzhaio}) for the cost function $c_2$:

\begin{proposition}\label{abc}
Consider the problem \eqref{adjgjzgjjgj}. Let $\mu,\nu$ be Borel probability measures that are strictly timelike and such that $\tau\in L^2(\mu)\cap L^2(\nu)$. Assume that $\supp(\mu)$ is connected and $\supp(\mu)$, $\supp(\nu)$ are causally compact.

Then for any $\pi\in \Gamma_o(\mu,\nu)$ there exists a $\pi$-solution.
\end{proposition}

For the notion of \q{causal compactness}, see Definition \ref{causally}. As usual, $\Gamma_o(\mu,\nu)$ denotes the set of all couplings which minimize Kantorovich's formulation \eqref{adjgjzgjjgj}. This proposition already appeared in \cite{Kell}, but for the cost function \eqref{ioaufhdoahfoaidai}.
In that paper, the authors showed that, under suitable assumptions on the measures, the \q{standard} Rockafellar construction also works and provides a $\pi$-solution (for this cost function), $\pi$ being an optimal coupling. The proof of Proposition \ref{abc} consists in showing that the arguments in \cite{Kell} can also be applied to our case for the cost function $c_2$. At this point, let us refer to section 2.2 in \cite{Bertrand/Puel/Pratelli}, where the authors explain the difficulty in constructing Kantorovich potentials when the cost functions is not finite. The approach in \cite{Bertrand/Puel/Pratelli} uses the so called \q{finite chain Lemma}, which is somehow an equivalent approach to the one in \cite{Kell}, and hence to ours.
Let us also mention the results on Kantorovich duality in \cite{Mondino/Cavalletti}. There the authors prove that the cyclical monotonicity (w.r.t.\ the Lorentzian distance, not the cost function) of an optimal coupling implies strong Kantorovich duality, i.e.\ the existence of a maximizing pair in the dual optimal transportation problem. The result is stronger than ours, since a $\pi$-solution does not need to be integrable. However, the assumption of the cyclical monotonicity w.r.t.\ the Lorentzian distance is not guaranteed for the measures in Proposition \ref{abc} and we will work only with $c_2$-cyclical monotonicity, which always holds for an optimal coupling (see Theorem \ref{2}).

In the theory of optimal transport on Riemannian manifolds and well-behaved cost functions, it is possible to derive a formula for the (unique) optimal transport map, which involves the cost function and the gradient of the Kantorovich potential. To follow this strategy (using a $\pi$-solution instead of the Kantorovich potential), we need to prove some regularity results for our $\pi$-solution. The first step in this direction is the following theorem, which allows us to prove the main Theorem \ref{unghtrdsxcfghjukil} below and which extends a result already known in the Riemannian case \cite{Figalli/Gigli}.
\begin{theorem} \label{mfdghufsa}
 Let $\varphi:M\to \R \cup \{\pm\infty\}$ be a $c_2$-convex function, and define
\begin{align*}
D:=\{x\in M\mid \varphi(x)\in \R\} \text{ and }
\Omega:=\op{int}(D).
\end{align*}
Then the following assertions hold:
\begin{enumerate}[(i)]
\item
$\varphi_{|\Omega}$ is locally bounded.
\item
$D\backslash \Omega$ is countably $n$-rectifiable.
\item
For each compact $K \subseteq \Omega$, the set of all $y\in M$ such that
\begin{align*}
\psi(y)-c_2(x,y)\geq \varphi(x)-1 \text{ for some } x\in K
\end{align*}
is relatively compact. Here, $\psi:=\varphi^{c_2}$ denotes the $c_2$-transform of $\varphi$.
\end{enumerate}
\end{theorem}
 
As mentioned above, this theorem is known in the Riemannian case, see \cite{Figalli/Gigli}, which addresses the corresponding optimal transportation problem on a complete and connected Riemannian manifold, where the cost function is given by $d_R^2$, $d_R$ denoting the Riemannian distance. Except for some modifications, the proof of \cite{Figalli/Gigli} also applies to our setting. In particular, we do not claim that there are any new major arguments in our proof. Nevertheless we provide a detailed proof here because of the differences. It is worth mentioning that the proof relies on the superlinearity of the Lagrangian associated with our cost function (Section \ref{sec3}), and that the same proof fails for the cost functions \eqref{ioaufhdoahfoaidai} or \eqref{q}. Actually, this is the reason why we chose to work with this cost function.

In the Riemannian case one can easily deduce from the above theorem that $\varphi_R$ (denoting a Kantorovich potential for the cost function $d_R^2$) is locally semiconvex on $\Omega$ (see \cite{Figalli/Gigli}). This follows from the fact that $d_R^2$ is locally semiconcave and that the finite supremum of uniformly locally semiconvex functions is again locally semiconvex. From this one finally concludes that, if $\mu$ is absolutely continuous w.r.t.\ the Lebesgue measure and the total cost \eqref{total} is finite, there is a unique optimal coupling and it is induced by a transport map. Moreover, it is possible to prove a formula for this transport map in terms of the derivative of $\varphi_R$ (which exists $\mu$-a.e.\ by the local semiconvexity). Observe that the delicate part is the local semiconvexity of $\varphi_R$. Indeed, one can prove the existence and uniqueness of an optimal map (or coupling) without invoking local semiconvexity, relying on the fact that $\varphi_R$ is approximately differentiable $\mu$-a.e., which is easier to prove (see \cite{Fathi/Figalli}).

The following theorem is our main result of this paper and deals with a corresponding result for the Lorentzian case. The difficulty in our case lies in the fact that the cost function $c_2$ is not locally semiconcave. But it is, when restricted to $I^+$, as we will see. Thus, roughly speaking, we need that the $c_2$-subdifferential of $\varphi$ is locally bounded away from $\partial J^+$. We were only able to prove a weaker version of the result in \cite{Figalli/Gigli}. However, as an easy corollary (see below), it still enough to prove uniqueness of the optimal coupling and existence of an optimal transport map under suitable assumptions.

\begin{theorem} \label{unghtrdsxcfghjukil}
Consider the problem \eqref{adjgjzgjjgj}. Let $\mu,\nu$ be Borel probability measures that are causally related and such that $\tau\in L^2(\mu)\cap L^2(\nu)$. Assume that $\supp(\mu)\cap \supp(\nu)=\emptyset$ and that $\mu$ is absolutely continuous w.r.t.\ the Lebesgue measure on $M$. Let $\pi\in \Gamma_o(\mu,\nu)$ and assume that $\varphi:M\to \R \cup \{\pm\infty\}$ is a $\pi$-solution.

Then there exists an open set $\Omega_1\subseteq \Omega$ of full $\mu$-measure such that $\varphi$ is locally semiconvex on $\Omega_1$.
\end{theorem}

The assumption $\supp(\mu)\cap \supp(\nu)=\emptyset$ counts for the fact that there is no trivial transport. 
The regularity assumption on $\mu$ is quite natural. Indeed, since we only consider the set $\Omega$, we require that $D\backslash \Omega$ is irrelevant for $\mu$, i.e. that $\mu$ does not give mass to countably $n$-rectifiable sets. In fact, we expect that our theorem also works in the case in which $\mu$ only satisfies this regularity assumption.

This result is new in the Lorentzian context and since the proof relies on Theorem \ref{mfdghufsa}, it makes crucial use of the particular cost function. It is not clear if this result extends to the other Lorentzian cost functions. Theorem 4.3 by McCann \cite{McCann2} provides a variant of Theorem \ref{unghtrdsxcfghjukil}. However, in his characterisation of Ricci curvature bounds, McCann did not require  general duality results. Instead, it was sufficient to establish strong duality for measures that are $q$-seperated. Consequently, the semiconvexity of a Kantorovich potential (relative to $\supp(\mu)$) was shown only under the strong assumption of $q$-separation. In contrast, our result holds under the much more general assumption that a $\pi$-solution exists.

As explained above, there already exists several results concerning the existence and uniqueness of an optimal transport map in the Lorentzian setting (\cite{Kell}, \cite{Suhr}, \cite{McCann}). It is therefore not surprising that these results extend to the cost function we deal with. However, in analogy with the Riemannian case, and to complete the picture, it is interesting to see that these results easily follow from our main theorem. Indeed, using the same arguments as in the Riemannian case \cite{Fathi/Figalli}, \cite{Figalli/Gigli}, the local semiconvexity of $\varphi$ (and, hence, its a.e.\  differentiability) on $\Omega_1$ will allow us to deduce that any optimal coupling $\pi$ that admits a $\pi$-solution must be induced by a transport map. Thus, we will finally prove the following:

\begin{corollary} \label{unghtrdsxcfghjukil3}
Consider the problem \eqref{adjgjzgjjgj}. Let $\mu,\nu$ be Borel probability measures that are causally related and such that $\tau\in L^2(\mu)\cap L^2(\nu)$. Assume that $\supp(\mu)\cap \supp(\nu)=\emptyset$ and that $\mu$ is absolutely continuous w.r.t.\ the Lebesgue measure on $M$. Let $\pi \in \Gamma_o(\mu,\nu)$ be an optimal coupling which admits a $\pi$-solution $\varphi$.
Then:
\begin{enumerate}[(i)]
\item  $\pi$ is induced by a transport map $T$.
More precisely, $\mu$-a.e., $T(x)$ is uniquely defined by the equation
\begin{align}
    \frac{\partial c_2}{\partial x}(x,T(x))=-d_x\varphi. \label{sdqwertgfghjk}
\end{align}
\item 
If there exists an optimal coupling different from $\pi$, then there also exists an optimal coupling $\pi'$ that does not admit a $\pi'$-solution.
\end{enumerate}
\end{corollary}

To my knowledge, our procedure (following very closely the Riemannian setting) is new in the Lorentzian context. However, as mentioned above, the results of the corollary are not. A very general result for the cost function \eqref{q} has been obtained by McCann \cite{McCann}. He proved (under the assumptions that the total cost is finite and that $\mu$ is absolutely continuous) that there exists at most one optimal coupling concentrated on $I^+$ (see Theorem 7.1). We assume that this result extends to our case. Then, since the assumptions of the above corollary imply that $\pi$ is concentrated on $I^+$, see Proposition \ref{hiuajdoaildkao} and also \cite{Kell}, this shows that there actually exists no further optimal coupling that admits a $\pi$-solution (provided the result of McCann holds). Since the main goal of this work was the local semiconvexity of a $\pi$-solution, we will not pursue this further. However, let us remark that in this work we are able to prove a formula for the optimal transport map in terms of the $\pi$-solution.

Proposition \ref{abc} and the above corollary immediately yield:

\begin{corollary} \label{unghtrdsxcfghjukil4}
    Consider the problem \eqref{adjgjzgjjgj}. Let $\mu,\nu$ be Borel probability measures that are strictly timelike and such that $\tau\in L^2(\mu)\cap L^2(\nu)$. Assume that $\supp(\mu)$ is connected, $\supp(\mu)$, $\supp(\nu)$ are causally compact and $\supp(\mu)\cap \supp(\nu)=\emptyset$. Furthermore, assume that $\mu$ is absolutely continuous w.r.t.\ the Lebesgue measure on $M$. 

    Then there exists a unique optimal coupling and it is induced by a transport map.
\end{corollary}

Note that if one is only interested in the existence of an optimal map, but not in its structure (resp. the structure of the $\pi$-solution), one can either argue as in \cite{McCann2} as explained above (at least we expect that) or, using the existence of $\pi$-solutions but not their local semiconvexity, one can use similar arguments as in \cite{Fathi/Figalli} together with Proposition \ref{hiuajdoaildkao} to show that $\varphi$ is approximately differentiable $\mu$-a.e. and that \eqref{sdqwertgfghjk} holds $\mu$-a.e.\ with $d_x\varphi$ replaced by its approximate differential. 
\bigskip

This paper is organized as follows: In Chapter 2 we recall the most important definitions and some well-known results about Lorentzian geometry and the theory of optimal transport. In Chapter 3 we show that the cost function $c_2$ arises as minimal action of some Lagrangian $L_2$ defined on the tangent bundle $TM$ and we investigate the existence and properties of minimizing curves for $L_2$. We conclude Chapter 3 with the definition of an $L_2$-exponential function. In Chapter 4 we start investigating the optimal transportation problem for the cost function $c_2$. In this chapter we will prove Proposition \ref{abc}. Chapter 5 is devoted to the proof of Theorem \ref{mfdghufsa} and in Chapter 6 we will prove Theorem \ref{unghtrdsxcfghjukil} and Corollary \ref{unghtrdsxcfghjukil3}.

\section{Preliminaries}

In this brief chapter, we recall the fundamental concepts of Lorentzian geometry and the theory of optimal transport. 
{\center {\large{\textbf{Spacetimes}}}}
\bigskip

We consider a \emph{spacetime} $(M,g)$, that is, $M$ is a smooth and connected manifold (i.e. Hausdorff and second-countable), $g$ is a symmetric $(0,2)$-tensor field of constant signature $(1,n)=(-,+,...,+)$, where $n+1:=\dim(M)$, and $M$ is time-oriented. \emph{Time-orientability} means that there exists a smooth global vector field $X:M\to TM$ such that, for each $x\in M$, $g_x(X(x),X(x))<0$. 

A vector $v\in T_xM$ is called \emph{timelike, spacelike, or lightlike} if 
\begin{align*}
g_x(v,v)<0,\ g_x(v,v)>0 \text{ or } g_x(v,v)=0 \text{ and } v\neq 0.
\end{align*}
Timelike and lightlike vectors are referred to as \emph{causal}. A causal vector $v$ is said to be \emph{future-directed} (resp.\ \emph{past-directed}) if $g_x(v,X(x))<0$ (resp.\ $>0$).
For $x\in M$, we denote by $\C_x\subseteq T_xM$ the set of all future-directed causal vectors. Then, $\operatorname{int}(\C_x)$ consists of all future-directed timelike vectors, while $\partial \C_x$ consists of the future-directed lightlike vectors and $0$. We also set $\C:=\{(x,v)\in TM\mid v\in \C_x\}$.

According to these definitions, we say that a locally absolutely continuous curve $\gamma:I\to M$, with $I\subseteq \R$ an interval, is \emph{future-directed causal/timelike} (or \emph{future pointing causal/timelike}) if $\dot \gamma(t)$ is future-directed causal/timelike for almost every $t$.

Two points $x,y\in M$ are said to be \emph{causally related} (resp.\ \emph{chronologically related}) if there exists a future-directed (absolutely continuous) causal (resp. timelike) curve connecting them. In this case, we write $x<y$ (resp. $x\ll y$, respectively). We write $x\leq y$ if $x=y$ or $x<y$.
The relations $>,\gg$ and $\geq$ are defined analogously.

 We define the \emph{chronological future/past} and \emph{causal future/past} of a point $x\in M$ as follows:
\begin{align*}
&I^+(x):=\{y\in M\mid y\gg x\},
\\
&I^-(x):=\{y\in M\mid y\ll x\},
\\
&J^+(x):=\{y\in M\mid y\geq x\},
\\
&J^-(x):=\{y\in M\mid y \leq x\}.
\end{align*}
We also set 
\begin{align*}
J^+:=\{(x,y)\in M\mid y\in J^+(x)\} \text{ and } I^+:=\{(x,y)\in M\mid y\in I^+(x)\}.
\end{align*}

In this paper we focus on globally hyperbolic spacetimes, which are defined as follows:

\begin{definition}\rm
A spacetime $M$ is said to be \emph{globally hyperbolic} if there is no causal loop (i.e. no closed future pointing causal curve), and if for any $x,y\in M$, the intersection $J^+(x)\cap J^-(y)$ is compact.
\end{definition}
For the remainder of this chapter, we assume that $M$ is globally hyperbolic. We now recall the definition of the Lorentzian length functional and its key properties. 

\begin{definition}\rm
The \emph{length} of an absolutely continuous, future pointing causal curve $\gamma:[a,b]\to M$ is defined by
\begin{align*}
L(\gamma):=\int_a^b |\dot \gamma(t)|_g \in [0,\infty).
\end{align*}
The \emph{(Lorentzian) distance function} or \emph{time seperation} is the function $d:M\times M\to \R$, defined such that for $x<y$, the distance $d(x,y)$ is the supremum of $L(\gamma)$ over all (a.c. future pointing causal) curves connecting $x$ with $y$, and such that $d(x,y)=0$ if $x\nless y$.

A curve $\gamma$ connecting $x$ with $y>x$ is said to be \emph{(length) maximizing} if $L(\gamma)=d(x,y)$.
\end{definition} 

Alternative (and more common) definitions of the causal future/past/... and the distance function are based on considering only piecewise smooth curves.
The equivalence of these definitions is established in \cite{Minguzzi}, Theorem 2.9.

\begin{proposition}
    The distance function is continuous. Moreover, for any points $x<y$, there exists a length maximizing geodesic connecting $x$ and $y$.
Additionally, the set $J^+$ is closed.
\end{proposition}
\begin{proof}
    See \cite{ONeill}, Chapter 14, Proposition 19 and Lemma 22. Observe that the global hyperbolicity is crucial.
\end{proof}

 It is often helpful to fix an arbitrary complete Riemannian metric on $M$, which we will denote by $h$.
Recall that, according to the Whitney embedding theorem, every manifold can be embedded in some $\R^d$ as a closed submanifold. Since closed submanifolds in $\R^d$ are complete w.r.t.\ the induced metric, we can define a complete Riemannian metric on $M$ by pulling back this metric.

We will denote the norm of a vector $v$ (w.r.t.\ $h$) by $|v|_h$. We also denote $|v|_g:=\sqrt{|g(v,v)|}$.

Since $M$ is globally hyperbolic, the proof of Theorem 3 in \cite{Bernard/Suhr} and the subsequent discussion, along with Corollary 1.8, show that there exists a smooth manifold $N$ and a diffeomorphism $M\cong \R\times N$, such that the projection $\tau:\R\times N\to \R, \ x\cong(t,z)\mapsto t,$ satisfies the following inequality:
\begin{align}
d_x\tau(v)\geq \max\bigg\{2|v|_g,|v|_h\bigg\} \label{splitting}
\end{align}
for all causal vectors $v\in \C_x$. This function is called a \emph{splitting} or \emph{time function}.

Throughout this paper, when we refer to a Riemannian metric on the tangent bundle $TM$, we always mean the natural Sasaki metric. For its definition, we refer the reader to subsection \ref{sasaki} in the appendix.

For further references on Lorentzian geometry, we refer the reader to \cite{ONeill}. Observe that our notion of global hyperbolicity may seem weaker than the one presented in \cite{ONeill}. However, the definitions are actually equivalent, as shown by a theorem due to Bernal and S\'anchez \cite{Bernal}.
\bigskip

{\center {\large{\textbf{Optimal transport}}}}
\bigskip

Let $(X,d)$ be a Polish space (i.e.\ a complete, separable metric space), and let ${\cal P}$ denote the set of all Borel probability measures on $X$. Given two measures $\mu,\nu\in {\cal P}$, the Monge problem consists of finding a minimizer for 
\begin{align*}
\inf\bigg\{\int_X c(x,T(x))\, d\mu(x)\mid T:X\to X \text{ Borel},\,  T_{\#}\mu=\nu\bigg\},
\end{align*}
where $c:X\times X\to [0,\infty]$ is a Borel cost function and $T_\#\mu$ is the push-forward measure, defined by $T_\#\mu(B):=\mu(T^{-1}(B))$ for all Borel sets $B\subseteq Y$ ($T$ is also called transport map).
This problem may not always be well-defined. For instance, if $\mu$ is a Dirac measure and $\nu$ is not, there will be no transport map between $\mu$ and $\nu$. However, due to Kantorovich's formulation, one can instead search for minimizers of
\begin{align}
C(\mu,\nu):=\inf\bigg\{\int_{X\times X} c(x,y)\, d\pi(x,y)\mid \pi\in \Gamma(\mu,\nu)\bigg\}, \label{Kantorovich}
\end{align}
where $\Gamma(\mu,\nu)$ denotes the set of all \emph{couplings} between $\mu$ and $\nu$, i.e. all Borel probability measures $\pi\in {\cal P}(X\times X)$ such that the first (resp. second) marginal of $\pi$ is $\mu$ (resp. $\nu$). A coupling is said to be \emph{optimal} if it minimizes \eqref{Kantorovich}, and the value $C(\mu,\nu)$ is referred to as the \emph{total cost}.

The advantage of Kantorovich's formulation over Monge's is that the set $\Gamma(\mu,\nu)$ is never empty since it contains the product measure. Moreover, a minimizer exists under mild conditions on $c$. Kantorovich's approach also generalizes Monge's formulation in the sense that any transport map $T$ gives rise to a coupling by defining $\pi:=(\Id\times T)_\#\mu$.
\\

Let us now recall some well-known results in the theory of optimal transport regarding the existence and structure of optimal couplings.
These results can be found in the books by Ambrosio/Gigli/Savaré  \cite{Ambrosio}, or by Villani \cite{Villani}. From now on, until the rest of this chapter, let 
\begin{align*}
    c:X\times X\to [0,\infty]
\end{align*}
be a proper (i.e. $c\not \equiv \infty$) and lower semi-continuous function.

\begin{theorem} \label{1}
Let $\mu,\nu \in {\cal P}$.
\begin{enumerate}[(i)]
\item
There is duality:
\begin{align*}
C(\mu,\nu)
=
\sup\bigg\{\int_X \varphi(x)\, d\mu(x)+\int_X \psi(y)\, d\nu(y)\bigg\},
\end{align*}
where the supremum is taken over all $\varphi\in L^1(\mu),\psi\in L^1(\nu)$ such that $\varphi(x)+\psi(y)\leq c(x,y)$ for all $x,y$.

\item
There exists an optimal coupling $\pi\in \Gamma(\mu,\nu)$, i.e.
\begin{align*}
\int_{X\times X} c(x,y)\, d\pi(x,y)= C(\mu,\nu).
\end{align*}
\end{enumerate}
\end{theorem}
\begin{proof}
\cite{Ambrosio}, Theorem 6.1.1 and the first page of Chapter 6.
\end{proof}

Closely related to optimal couplings is the concept of so called Kantorovich potentials. To introduce these, we need to define the $c$-transform of a function.

\begin{definition}\rm
\begin{enumerate}[(i)]
\item
A function $\varphi:X\to \overline \R:=\R\cup \{\pm \infty\}$ is said to be \emph{$c$-convex} if there exists a function $\Psi:X\to \overline \R$ such that
\begin{align*}
\varphi(x)=\sup_{y\in X}\( \Psi(y)-c(x,y)\).
\end{align*}
Here, the convention $\infty-\infty=-\infty$ is used.

The \emph{$c$-transform} of $\varphi$ is then defined as
\begin{align}
\varphi^c(y):=\inf_{x\in X}\(c(x,y)+\varphi(x)\), \label{sjahaudhsadmm}
\end{align}
where the convention $\infty-\infty=\infty$ is used.

A function $\psi:X\to \overline \R$ is said to be \emph{$c$-concave} if $\psi=\varphi^c$ for some $c$-convex function $\varphi$.
Observe that, in this case, 
\begin{align}
    \varphi(x)=\sup_{y\in X} (\psi(y)-c(x,y)), \label{sjahaudhsadmm1}
\end{align}
where we again use the convention $\infty-\infty=-\infty.$
\item A subset $\Gamma\subseteq M\times M$ is said to be \emph{$c$-monotone} if:
\begin{align*}
 k\in \N,\ (x_i,y_i)_{1\leq i\leq k}\subseteq \Gamma\Rightarrow \sum_{i=1}^k c(x_i,y_i)\leq \sum_{i=1}^k c(x_{i+1},y_i),
\end{align*}
where $x_{k+1}:=x_1$.
\end{enumerate}
\end{definition}

\begin{convention}\rm
    Since $c$ may take the value $\infty$, we must be careful about the the conventions for sums. This issue does not arise when $c$ is real-valued (see Remark \ref{remark}). Therefore, in this paper, we will use the following conventions for sums (which are consistent with the conventions in the above definition): Given that $\varphi$ is $c$-convex:
        \begin{enumerate}[(i)]
       \item  If $\varphi^c(y)=\pm \infty$ and $\varphi(x)=\pm \infty$ then $\varphi^c(y)-\varphi(x)=-\infty$
       \item If $\varphi^c(y)=\infty$ and $c(x,y)=\infty$ then $\varphi^c(y)-c(x,y)=-\infty$
       \item If $\varphi(x)=\infty$ and $c(x,y)=\infty$ then $c(x,y)-\varphi(x)=\infty$
       \end{enumerate}
       One easily checks that 
       \begin{align*}
\varphi^c(y)-\varphi(x)\leq c(x,y) \text{ for all } x,y.
\end{align*}
However, note that usual operations do not hold when $\pm \infty$ occurs: For example,
\begin{align*}
   \varphi^c- c(x,y)=\varphi(x) \not \Rightarrow \varphi^c(y)-\varphi(x)=c(x,y).
\end{align*}
\end{convention}

\begin{definition}\rm
Let $\mu,\nu \in {\cal P}$. We say that a $c$-convex function $\varphi \in L^1(\mu)$ is a \emph{Kantorovich potential}
if $\varphi^c\in L^1(\nu)$ and 
\begin{align*}
\int_X \varphi^c(y)\, d\nu(y)-\int_X \varphi(x)\, d\mu(x)
=
\int_{X\times X} c(x,y)\, d\pi(x,y)
\end{align*}
for one (or all) optimal couplings $\pi$.

The \emph{$c$-subdifferential} of a $c$-convex function $\varphi$ at $x$ is defined as
\begin{align*}
\partial_c \varphi(x):=\{y\in X\mid \varphi^c(y)-\varphi(x)=c(x,y)\}.
\end{align*}
We also set
\begin{align*}
\partial_c \varphi=\bigcup_{x\in X} \{x\}\times\partial_{c}\varphi(x).
\end{align*}
\end{definition}

\begin{lemma}\label{wewewewwewew}
Let $\mu,\nu\in {\cal P}$ and $\pi$ be any optimal coupling. If $\varphi$ is any Kantorovich potential, then it holds that $\varphi^{c}(y)-\varphi(x)=c(x,y)$ $\pi$-a.e. 
\end{lemma}
\begin{proof}
Let $\pi\in \Gamma(\mu,\nu)$ be an optimal coupling. By the definition of Kantorovich potentials, $c\in L^1(\pi)$. Hence, we have
\begin{align*}
0=\int_{X\times X} c(x,y)\, d\pi(x,y)+\int_X \varphi(x)\, d\mu(x)-\int_X \varphi^c(y)\, d\nu(y)=
\int_{X\times X} c(x,y)+\varphi(x)-\varphi^c(y)\, d\pi(x,y).
\end{align*}
By the integrability assumptions, $\varphi(x),\varphi^c(y),c(x,y)\in \R$ $\pi$-almost surely. Since $c+\varphi-\varphi^c\geq 0$ $\pi$-a.e., it follows that, $\pi$.a.e., $c(x,y)=\varphi^c(y)-\varphi(x)$. 
\end{proof}

\begin{remark}\rm \label{remark}
\begin{enumerate}[(a)]
\item
If $c$ is real valued and $\varphi$ is a $c$-convex function that attains the value $-\infty$, then one can easily verify using the definition that $\varphi\equiv -\infty$. In particular, $\varphi \notin L^1(\mu)$. This illustrates that when $c$ is finite, any Kantorovich potential maps to $\R\cup\{\infty\}$. Hence, the definition of a $c$-convex function typically requires that $\varphi$ maps to $\R\cup\{\infty\}$.
\item In the case where $c$ is real-valued, there is an explicit construction for a $c$-convex function, namely the Rockafellar construction, and it can be shown that this function is a Kantorovich potential (under very mild conditions). For a precise statement, see \cite{Ambrosio}, Theorem 6.1.4.

However, this proof does not extend to the case where $c$ attains the value $\infty$. Therefore, we will work with a weaker version of the Kantorovich potential, defined by the properties of Lemma \ref{wewewewwewew}. See also Definition \ref{daui9dhoafzhiofzhaio}. Let us note that our definition closely resembles the definition of a $(c,\pi)$-calibrated pair of functions as described \cite{Fathi/Figalli}, Definition 2.2.
\item We conclude this chapter with a well-known result regarding the structure of an optimal coupling.
\end{enumerate}
\end{remark}

\begin{theorem}\label{2}
 Let $\mu,\nu \in {\cal P}$. Let $\pi$ be an optimal coupling, and assume that $\int_{X\times X} c(x,y)\, d\pi(x,y)<\infty$. Then $\pi$ has to be concentrated on a $c$-monotone Borel set, i.e.\ there exists a $c$-monotone Borel set $\Gamma\subseteq X\times X$ such that $\pi(\Gamma)=1$.
\end{theorem}
\begin{proof}
    \cite{Ambrosio}, Theorem 6.1.4.
\end{proof}

\section{Lagrangian action and exponential map}\label{sec3}

In this section, we define our cost function which arises as the minimal action of a particular Lagrangian on the tangent bundle $TM$. We will also investigate the existence and uniqueness of minimizing curves, along with their properties.This chapter concludes with the proof of the existence of an Euler-Lagrange flow for our Lagrangian (observe that the Lagrangian is not differentiable).

\textbf{From this point until the appendix,} we assume that our spacetime $M$ is globally hyperbolic and has dimension $\dim(M)=n+1$. It is equipped with a Lorentzian metric $g$ and the corresponding Levi-Civita connection. 
We use coordinates on $M$ indexed by $0,...,n$, and we fix some complete Riemannian metric $h$. Recall that, for $v\in T_xM$, we denote its $h$-norm by $|v|_h$, and we write $|v|_g:=\sqrt{|g(v,v)|}$. We also fix an arbitrary time function $\tau$ satisfying \eqref{splitting}.
When we refer to a future pointing causal curve being maximizing, we are always referring to the Lorentzian length functional.
\newline

The cost function we will study is defined as the minimal action of the Lagrangian $L_2:TM\to \R\cup\{\infty\}$,  given by
\begin{align*}
 L_2(x,v):=
\begin{cases}
\(d_x\tau(v)-|v|_g\)^2, \ &v\in \C_x\cup\{0\},
\\
\infty,\ &\text{ otherwise}.
\end{cases}
\end{align*}
With some abuse of notation, we also write $L_2(v)=L_2(x,v)$.
The corresponding Lagrangian action of an absolutely continuous curve $\gamma:[a,b]\to M$ is then given by
\begin{align*}
{\cal A}_2(\gamma):=\int_a^b L_2(\gamma(t),\dot \gamma(t))\, dt\in [0,\infty],
\end{align*}
and the cost function (or minimal action) we seek to investigate is
\begin{align*}
c_2:M\times M\to [0,\infty],\ c_2(x,y):=\inf\{{\cal A}_2(\gamma)\mid \gamma\in AC([0,1],M), \gamma(0)=x,\gamma(1)=y\}.
\end{align*}
Lemma \ref{daidaoidasl} establishes that this definition is consistent with \eqref{haihduaisoafafdfsfdfsfdfdsfsfdsfdsfs}.  To study the minimizers of $L_2$ and the cost function $c_2$, it is also useful to consider a different Lagrangian, namely 
\begin{align*}
    L_1:TM\to \R\cup\{\infty\},\ L_1(x,v):=
\begin{cases}
d_x\tau(v)-|v|_g, \ &v\in \C_x\cup\{0\},
\\
\infty,\ &\text{ otherwise}.
\end{cases}
\end{align*}
The corresponding Lagrangian action of an absolutely continuous curve $\gamma:[a,b]\to M$ is given by
\begin{align*}
{\cal A}_1(\gamma):=\int_a^b L_1(\gamma(t),\dot \gamma(t))\, dt\in [0,\infty]
\end{align*}
and the corresponding minimal action is 
\begin{align*}
c_1:M\times M\to [0,\infty],\ c_1(x,y):=\inf\{{\cal A}_1(\gamma)\mid \gamma\in AC([0,1],M), \gamma(0)=x,\gamma(1)=y\}.
\end{align*}

\begin{definition}\rm
An absolutely continuous curve $\gamma:[a,b]\to M$ is said to \emph{minimize the action} ${\cal A}_p$ ($p=1,2$) if, for any other curve $\bar \gamma:[a,b]\to M$ with the same start and end points as $\gamma$, we have
\begin{align*}
    {\cal A}_p(\gamma)\leq {\cal A}_p(\bar \gamma).
\end{align*}
\end{definition}

\begin{definition}\rm
 We say that an absolutely continuous curve $\gamma:[c,d]\to M$ is a \emph{reparametrization} of an absolutely continuous curve $\bar \gamma:[a,b]\to M$ if there exists an absolutely continuous
   and non-decreasing bijection $\psi:[c,d]\to [a,b]$ such that $\gamma(t)=\bar \gamma(\psi(t))$.
\end{definition}

\begin{remark}\rm
The action functional ${\cal A}_1$ is similar to the Lorentzian length functional. Indeed, for any a.c. future pointing causal curve $\gamma:[a,b]\to M$, we have
\begin{align*}
    {\cal A}_1(\gamma)=\tau(\gamma(b))-\tau(\gamma(a))-L(\gamma).
\end{align*}
However, while the Lorentzian length functional is only defined for future pointing causal curves, i.e.\  $\dot \gamma(t)\neq 0$ for almost all $t$, the action ${\cal A}_1(\gamma)$ is defined and finite for every a.c.\ curve such that $\dot \gamma(t)\in \overline \C_{\gamma(t)}=\C_{\gamma(t)}\cup \{0\}$, which allows for the possibility that $\dot \gamma(t)=0$ at certain points.

Obviously, we could also define $L(\gamma)$ using the same formula for this class of curves. Therefore, let us consider $\gamma$ as above and assume that $\dot \gamma(t)=0$ is possible (on a set of positive measure). By standard methods, one can prove that $\gamma(t)\in J^+(\gamma(s))$ whenever $s\leq t$ (see \cite{Minguzzi}, beginning of the proof of Theorem 2.9). It is a well-known result about curves in metric spaces that $\gamma$ is a reparametrization of a Lipschitz curve $\bar \gamma$  parametrized by $h$-arc length.

Using the fact that, easily verified, $\bar \gamma(t)\in J^+(\bar \gamma(s))$ for $s\leq t$, and the fact that $\bar \gamma$ is parametrized by $h$-arc length, it follows that $\dot {\bar \gamma}(t)\in \C_{\bar \gamma(t)}$ at each differentiability point $t$. Thus, $\bar \gamma$ is future pointing causal. Moreover, since $\gamma$ is a reparametrization of $\bar \gamma$ it is immediate that
\begin{align*}
    L(\gamma)=L(\bar \gamma).
\end{align*}
This shows that $L(\gamma)\leq d(\gamma(a),\gamma(b))$. Furthermore, if $\gamma$ is maximizing, i.e. $L(\gamma)= d(\gamma(a),\gamma(b))$, then so is $\bar \gamma$. By a well-known result for future pointing causal a.c.\ curves, it follows that $\bar \gamma$ is a reparametrization of a maximizing geodesic (see \cite{Minguzzi}, Theorem 2.9 or 2.20). In particular, $\gamma$ is a reparametrization of a maximizing geodesic (here we use the fact that the composition of two non-decreasing absolutely continuous curves is again absolutely continuous).

Let us summarize these results in the following proposition.
\end{remark}

\begin{proposition}
Let $x<y$ and let $\gamma:[a,b]\to M$ be an absolutely continuous curve with $\dot \gamma(t)\in \overline  \C_{\gamma(t)}$ connecting $x$ with $y$. Then, we have $L(\gamma)\leq d(x,y)$ and $L(\gamma)=d(x,y)$ if and only if $\gamma$ is a reparametrization of a maximizing geodesic.
\end{proposition}
With this proposition in hand, let us now investigate the existence and regularity of minimizers for our action functionals ${\cal A}_1$ and ${\cal A}_2$.

\begin{corollary}\label{havdza}
Let $x<y$. Then an absolutely continuous curve $\gamma:[a,b]\to M$ connecting $x$ with $y$ minimizes ${\cal A}_1$ if and only if $\gamma$ is a reparametrization of a maximizing geodesic.
\end{corollary}
\begin{proof}
This follows from the discussion above.
\end{proof}

\begin{proposition}\label{minimizer}
 Let $x<y$. An absolutely continuous curve $\gamma:[a,b]\to M$ connecting $x$ with $y$ minimizes ${\cal A}_{2}$ if any only if $\gamma$ is a reparametrization of a maximizing geodesic such that $L_1(\dot \gamma(t))$ is constant.

In particular, for any $x<y$ and $a<b$, there always exists an ${\cal A}_2$-minimizer $\gamma:[a,b]\to M$ connecting $x$ with $y$, and every minimizer is smooth.
\end{proposition}
\begin{proof}
Let us start with the following observation: There exists at least one maximizing geodesic $c_0:[a,b]\to M$ connecting $x$ with $y$.
Since $t\mapsto L_1(\dot c_0(t))$ is smooth (because $|\dot c(t)|_g$ is constant) and strictly positive, we can construct a smooth reparametrization $\gamma_0:[a,b]\to M$ of $c_0$ such that $L_1(\dot \gamma_0(t))=cons.$ Then $\gamma_0$ is minimizes ${\cal A}_1$. 

Now let $\gamma:[a,b]\to M$ be any future pointing causal curve with $\gamma(a)=x$ and $\gamma(b)=y$.
By H\"older's inequality, we have
\begin{align}
{\cal A}_1(\gamma)^2=\(\int_a^b L_1(\dot \gamma(t))\, dt\)^2\leq \int_a^b L_1^2(\dot \gamma(t))\, dt \cdot (b-a)= 
{\cal A}_2(\gamma)(b-a) \label{havdza1}
\end{align}
with equality if and only if $L_1(\dot \gamma(t))=cons.$ almost surely.
Since this holds true for any curve we conclude that $\gamma$ minimizes ${\cal A}_2$ if $\gamma$ minimizes ${\cal A}_1$ and $L_1(\dot \gamma(t))=cons.$ almost surely. In particular, $\gamma_0$ above minimizes ${\cal A}_2$. Conversely, let $\gamma$ minimize ${\cal A}_2$. Then, we obtain
\begin{align*}
    {\cal A}_1(\gamma)^2\leq {\cal A}_2(\gamma)(b-a)={\cal A}_2(\gamma_0)(b-a)={\cal A}_1(\gamma_0)^2\leq {\cal A}_1(\gamma)^2.
\end{align*}
Thus, $\gamma$ minimizes ${\cal A}_1$ as well, and since equality must hold in each of the above steps, it follows that $L_1(\dot \gamma(t))=cons.$ almost surely. This proves the equivalence.

We have already established that $\gamma_0$ minimizes ${\cal A}_2$, so we have proven the existence of a minimizer.

Now, let $\gamma$ be any minimizer. In particular, since $\gamma$ minimizes ${\cal A}_1$, it must be a reparametrization of a maximizing geodesic $c:[0,d]\to M$. Thus, there exists an absolutely continuous, non-decreasing bijection $\psi:[a,b]\to [0,d]$ with $\gamma=c\circ \psi$.
Since $L_1(\dot \gamma(t))=cons.$ almost surely, we have 
\begin{align*}
    \dot \psi=\frac{cons.}{L_1(\dot c(\psi))}.
\end{align*}
Since $t\mapsto L_1(\dot c(t))$ is smooth, it is immediate to see that $\psi$ is smooth. In particular, $\gamma$ must be smooth. Thus, all minimizers must be smooth.
\end{proof}

\begin{definition}\rm
We denote the set of all future pointing causal ${\cal A}_2$-minimizers $\gamma:[0,1]\to M$ by $\Gamma$, i.e.
\begin{align*}
\Gamma:=\{\gamma\in AC([0,1],M)\mid \gamma \text{ is future pointing causal and minimizes the action } {\cal A}_2\} .
\end{align*}
Moreover, for $x<y$, we will denote the subset of all minimizers connecting $x$ with $y$ by $\Gamma_{x,y}$, i.e.
\begin{align*}
\Gamma_{x,y}:=\{\gamma\in \Gamma\mid \gamma(0)=x,\ \gamma(1)=y\}.
\end{align*}
\end{definition}

\begin{lemma}\label{daidaoidasl}
The cost function $c_2$ satisfies 
\begin{align*}
c_2(x,y)=
\begin{cases}
(\tau(y)-\tau(x)-d(x,y))^2,\ &(x,y)\in J^+,
\\
\infty,\ &\text{otherwise}.
\end{cases}
\end{align*}
Furthermore, $c_2$ is lower semi-continuous.
\end{lemma}
\begin{proof}
From Lemma \ref{minimizer} and from \eqref{havdza1}, we deduce that $c_2(x,y)=c_1(x,y)^2$
 if $x<y$.
But using Lemma \ref{havdza} one easily checks that $c_1(x,y)=\tau(y)-\tau(x)-d(x,y)$ for $x<y$. Moreover, using again the fact that, for each absolutely continuous curve $\gamma$ with $\dot \gamma(t)\in \overline \C_{\gamma(t)}$, it holds $\gamma(s)\leq \gamma(t)$ whenever $s\leq t$, we conclude that $\gamma$ cannot be closed unless it is constant. Thus, $c_2(x,x)=0$ and $c_2(x,y)=\infty$ for $(x,y)\notin J^+$. This proves the first part of the lemma. 

The second part follows immediately from the first part and from the fact that the Lorentzian distance function is continuous (see \cite{ONeill}, Chapter 14, Lemma 21), that $J^+$ is closed (\cite{ONeill}, Chapter 14, Lemma 22), and that $\tau$ is smooth.
\end{proof}

\begin{proposition}\label{flow}
 There exists a relatively open set ${\cal D}_L\subseteq \overline \C\times \R$ and a continuous (local) flow
 \begin{align*}
     \phi:{\cal D}_L\to \overline \C\subseteq TM,\ (x,v,t)\mapsto \phi_t(x,v),
 \end{align*}
 such that the following properties hold:
 \begin{enumerate}[(i)]
 \item For any $(x,v)\in \overline \C$ the map $\{t\in \R\mid (x,v,t)\in {\cal D}_L\}\to TM,\ t\mapsto \phi_t(x,v),$ is smooth and has the form $(\gamma(t),\dot \gamma(t))$. 
 \item 
 If $x\leq y$ and $\gamma:[a,b]\to M$ is a minimizing curve for ${\cal A}_2$ connecting $x$ with $y$, then $\gamma$ is part of an orbit of this flow, i.e. if $t,s\in [a,b]$ we have $(\gamma(t),\dot \gamma(t))=\phi_{t-s}(\gamma(s),\dot \gamma(s))$.

 \item If $\phi_{\cdot}(x,v)$ is defined on the interval $J$ and $t>0$, then $\phi_{\cdot}(x,tv)$
 is defined on the interval $\frac{1}{t}J$ and it holds
 \begin{align*}
    \pi\circ  \phi_s(x,tv)=\pi\circ \phi_{ts}(x,v) \text{ for all } s\in \frac{1}{t}J,
 \end{align*}
 where $\pi:TM\to M$ denotes the projection onto $M$.
 \end{enumerate}
 \end{proposition}
 \begin{proof}
     The proof is not particularly complicated but lengthy and technical. Therefore we refer the interested reader to subsection 7.2 in the appendix. Of course, the idea for the proof is to take the geodesic flow on $M$ and reparametrize every orbit according to Proposition \ref{minimizer}.
 \end{proof}

\begin{remark}\rm
\begin{enumerate}[(a)]
    \item At this point, one might ask why we don't simply define the flow $\phi$ as the Euler-Lagrange flow of $L_2$. Indeed, $L_2$ is smooth on the set $\op{int}(\C)$, and the second fiber derivative of $L_2$, $\frac{\partial^2 L}{\partial v^2}(x,v)$, is positive definite at any $(x,v)\in \op{int}(\C)$, see Corollary \ref{pokjnbvcxdertzujik} in the appendix. Thus, the Euler-Lagrange flow of $L_2$ is well-defined and smooth. 
 However, this Euler-Lagrange flow is not defined for $(x,v)\in \partial \C$ and, thus, not defined on the the entire set $\C$ (or $\overline \C$). But in our cases, we will also need to consider these situations. Therefore, we state the proposition as above, giving up the smoothness of the flow and settling for continuity.
 
 From our construction in the proof it is easily verified that the sets $\op{int}(\C)$ and $\partial \C$ are invariant for the flow. Moreover, the Euler-Lagrange flow of $L_2$ on the set $\op{int}(\C)$ agrees with our flow map restricted to this invariant set. In particular, $\phi$ is smooth in the interior of its domain. Since we omit a proof here, let us remark that we will not make use of this fact in the paper.

\item 
We will apply proposition in situations where we consider minimizing curves $\gamma_n:[0,1]\to M$ with bounded initial velocities. The proposition then implies that these curves converge, along a subsequence, to another curve (also minimizing) in the topology of uniform convergence. This result could also be derived by employing the well-known Limit curve lemma (and its applications, Lemma 9.14 and Lemma 9.25 in \cite{Ehrlich}) to the maximal geodesics $c_n:[0,1]\to M$, of which the minimizing curves are reparametrizations. Then the initial velocities $\dot c_n(0)$ converge to a causal vector $v$, and the corresponding maximal geodesic $c(t):=\exp_x(tv)$ can be reparametrized to yield the minimizing limit curve. Thus, the problem can be reduced to translating between maximal geodesics and minimizing curves for ${\cal A}_2$. However, to avoid the precise argument whenever such a result is required, it is useful to have this general proposition, which immediately provides limit curve. Moreover, it is worth mentioning that part of the proof of the proposition precisely establishes the above mentioned translation.

    \item Note that we do not need all the properties from the proposition above. In fact, we even do not need the fact that $\phi$ is actually a flow. But since this fact is interesting in its own right we stated the proposition as above.

    \item The above proposition shows the concept of an exponential map makes sense in our case. Thus, we make the following definition.
\end{enumerate}
\end{remark}

\begin{definition}\rm
The \emph{exponential map of $L_2$} is the map
\begin{align*}
    \exp_L:\{(x,v)\in \overline \C\mid (x,v,1)\in {\cal D}_L\}\to M,\ \exp_L(x,v):=\pi\circ \phi_1(x,v).
\end{align*}
\end{definition}

\begin{corollary}\label{l}
Let $K\subseteq \overline \C$ be a compact set. Then there exists $l>0$ such that $\exp_L(x,tv)$ is defined for all $t\in [0,l]$ and all $(x,v)\in K$.
\end{corollary}
\begin{proof}
This follows easily from the fact that ${\cal D}_{L}$ is open in $\overline \C\times \R$ and that $K\times \{0\}\subseteq {\cal D}_{L}$.
\end{proof}

\begin{lemma}\label{lii}
   If $\gamma:[0,b]\to M$ is any future pointing causal minimizer for ${\cal A}_2$, we have the representation
   \begin{align*}
       \gamma(t)=\exp_L(\gamma(0),t\dot \gamma(0)),\ t\in [0,b].
   \end{align*}
\end{lemma}
\begin{proof}
    From Proposition \ref{flow}, we know that $\gamma(t)=\pi(\phi_t(\gamma(0),\dot \gamma(0)))$. Now part (iii) of the same proposition yields $\pi(\phi_t(\gamma(0),\dot \gamma(0)))=\pi(\phi_1(\gamma(0),t\dot \gamma(0)))=\exp_L(\gamma(0),t\dot \gamma(0)).$
\end{proof}

\section{Optimal transport for the cost function $c_2$}
In this section, we begin investigating the optimal transportation problem for the cost function $c_2$. To this end, let us denote the set of all Borel probability measures on $M$ by ${\cal P}$.
We will start with a few simple conditions on the measures to ensure finiteness of the total cost, and will then define $\pi$-solutions (which correspond to Kantorovich potentials). The main part of this chapter is dedicated to the proof of Proposition \ref{abc}.

\begin{definition}\rm \label{x11}
\begin{enumerate}[(i)]
\item
We say that two Borel probability measures are \emph{causally related} if there exists a coupling $\pi$ that is concentrated on $J^+$. In this case, the coupling $\pi$ is called \emph{causal}.
\item
We define the set
\begin{align*}
{\cal P}_\tau^+:=\{(\mu,\nu) \in {\cal P}\times {\cal P}\mid \tau \in L^2(\mu)\cap L^2(\nu),\ \mu \text{ and } \nu \text{ are causally related}\}.
\end{align*}
\item 
We say that a pair of Borel probability measures is \emph{strictly timelike} if there exists a coupling $\pi$ whose support is contained in $I^+$. In this case, the coupling $\pi$ is called \emph{timelike}.
\end{enumerate}
\end{definition}

\begin{remark}\rm
We adopted the notion of being \q{causally related} from \cite{Kell}. Also the notion of \q{strict timelikeness} first appeared (to my knowledge) in the paper of Suhr and Kell \cite{Kell} (though with a different definition) and is similar to the notion of \q{$q$-seperatedness} in \cite{McCann2}. Note, however, that the notation of strict timelikeness does not require $\pi$ to be optimal, which is a difference compared to $q$-seperatedness in \cite{McCann2}. It is worth mentioning that, in \cite{McCann2}, McCann was still able to prove (without the assumption of $q$-seperatedness) the existence and uniqueness of an optimal coupling among all couplings concentrated on $I^+$, and that this coupling is induced by a transport map (see Theorem 7.1).
\end{remark}

\begin{lemma}\rm \label{abcde}
Let $\mu,\nu\in {\cal P}$. Then the following properties hold:
\begin{enumerate}[(i)]
\item
There always exists an optimal coupling $\pi\in \Gamma(\mu,\nu)$ (for the cost function $c_2$).
\item
If $(\mu,\nu)\in {\cal P}_\tau^+$, then the total cost $C(\mu,\nu)$ is finite, and any optimal coupling must be concentrated on a $c_2$-monotone Borel set $\Gamma\subseteq M\times M$.
\end{enumerate}
\end{lemma}
\begin{proof}
Part (i) follows from Theorem \ref{1}. For part (ii), using the representation for $c_2$ of Lemma \ref{daidaoidasl}, one easily checks that the cost of the causal coupling $\pi$ is finite. Indeed, we have $\tau(y)-\tau(x)-d(x,y)\geq 0$ on $J^+$, thanks to \eqref{splitting} and also $\tau \in L^2(\mu)\cap L^2(\nu)$. These two facts easily imply that the cost of $\pi$ is finite.
 Hence, the total cost must also be finite. The second part follows from Theorem \ref{2}.
\end{proof}

\begin{definition}\rm \label{daui9dhoafzhiofzhaio}
Let $(\mu,\nu)\in {\cal P}_\tau^+$ and $\pi \in \Gamma_o(\mu,\nu)$ (the set of optimal couplings). We say that a $c_2$-convex function $\varphi:M\to \overline \R$ is a \emph{$\pi$-solution} if $\pi$ is concentrated on the set $\partial_{c_2}\varphi$, or equivalently,
\begin{align*}
    \varphi^{c_2}(y)-\varphi(x)=c_2(x,y) \ \pi\text{-a.e.}
\end{align*}
\end{definition}

\begin{remark}\rm
\begin{enumerate}[(i)]
\item 
In the above definition, we place no measurability or integrability assumptions on $\varphi$ (or on $\varphi^{c_2}$). The integrability assumptions $\varphi\in L^1(\mu)$ and $\varphi^{c_2}\in L^1(\nu)$ are the difference between our definition and the definition of Kantorovich potentials.
\item Let us note that we will frequently use the following fact: If $\mu,\nu\in {\cal P}$, and $\pi\in \Gamma(\mu,\nu)$ is concentrated on the set $A\subseteq M\times M$, then $\mu$ (resp.\ $\nu$) is concentrated on $p_1(A)$ (resp.\ $p_2(A)$), where $p_i$ denotes the projection onto the $i$-th component. Indeed, by the inner regularity of $\pi$ there is a $\sigma$-compact set $B\subseteq A$ such that $\pi(B)=1$. Then, $p_1(B)$ is a Borel set and is of full $\mu$-measure since $\mu(p_1(B))=\pi(p_1^{-1}(p_1(B)))\geq \pi(B)=1$. Thus, $p_1(A)\supseteq p_1(B)$ is of full $\mu$-measure. The argument for $\nu$ is analogous.
\end{enumerate}
\end{remark}

\begin{lemma}\label{gbaiw}
    Let $(\mu,\nu)\in {\cal P}_\tau^+$, $\pi \in \Gamma_o(\mu,\nu)$, and let $\varphi$ be a $\pi$-solution.
    Then $\varphi$ (resp. $\varphi^{c_2}$) is real-valued $\mu$-a.e. (resp. $\nu$-a.e.).
\end{lemma}
\begin{proof}
    Since $(\mu,\nu)\in {\cal P}_\tau^+$ and $\pi \in \Gamma_o(\mu,\nu)$, the cost of $\pi$ must be finite. In particular, $\pi$ is causal. Thus, $c_2(x,y)<\infty$ $\pi$-a.e. Let $A\subseteq M\times M$ be a set of full $\pi$-measure such that $c_2(x,y)<\infty$ and $\varphi^{c_2}(y)-\varphi(x)=c_2(x,y)$ for all $(x,y)\in A$. Then (see the remark) $p_1(A)$ (resp. $p_2(A)$) has full measure w.r.t.\ $\mu$ (resp $\nu$). 
    By the convention for the sum $\varphi^{c_2}-\varphi$, we immediately deduce that $\varphi$ (resp. $\varphi^{c_2}$) is real-valued on $p_1(A)$ (resp. $p_2(A)$). This proves the lemma.
\end{proof}

\begin{proposition}[see \cite{Kell}, Theorem 2.12]
\label{hiuajdoaildkao}
Let $(\mu,\nu)\in {\cal P}_\tau^+$ be a couple of probability measures such that $\mu$ is absolutely continuous w.r.t.\ the Lebesgue measure and $\supp(\mu)\cap \supp(\nu)=\emptyset$. Let $\pi\in \Gamma_o(\mu,\nu)$ and assume that $\varphi$ is a $\pi$-solution.
\begin{enumerate}[(i)]
\item  $\pi$ is concentrated on $I^+$.
\item
For $\mu$-a.e., $x$ it holds that $\partial_{c_2}\varphi(x)\cap I^+(x)\neq \emptyset$.
\end{enumerate}
\end{proposition}
\begin{proof}
\begin{enumerate}[(i)]
\item
This can be proven either as the more general Theorem \ref{gafuiahfauhfaohfoafjaoii} or as Theorem 2.12 in \cite{Kell} (where this result first appeared for the cost function \eqref{ioaufhdoahfoaidai}). We provide only a sketch of the proof for the case where $(M,g)=(\R^{1+n},\langle \langle\cdot,\cdot\rangle\rangle)$ is the Minkowski space and $\tau(t,x)=2t$.

Suppose $\pi$ is not concentrated on $I^+$. Then, necessarily, $\pi((J^+\backslash I^+)\cap \partial_{c_2}\varphi)>0$. Define the set
\begin{align*}
    B:=\{x\in \supp(\mu)\mid \exists y\in J^+(x)\backslash I^+(x),\ (x,y)\in \supp(\pi)\cap \partial_{c_2}\varphi\}.
\end{align*}
Then $B$ cannot be a null set w.r.t.\ $\mu$. Let $x\in B$ and choose a corresponding $y$ as in the definition of $B$. Since we assume $\supp(\mu)\cap \supp(\nu)=\emptyset$, we have $x<y$. Thus, there exists a maximizing geodesic $\gamma:[0,1]\to M$ connecting $x$ with $y$, and we define $v_x:=\dot \gamma(0)/|\dot \gamma(0)|$. 

Denote by $\pi:\R^{1+n}\to \{0\}\times \R^n$ the projection onto the last $n$ coordinates.

Given $\delta>0$, we find a finite set $F$ of unit vectors $v\in S^{n}\subseteq \R^{1+n}$ such that
\begin{align*}
    B\subseteq \bigcup_{v \in F} \{x\in B\mid |v_x-v|\leq \delta\}.
\end{align*}
Since $B$ is not a null set (w.r.t.\ $\mu$ and, hence, w.r.t.\ the Lebesgue measure), at least one of the sets above is not a null set. Thus, assuming measurability of these sets, we can apply Lebesgue's differentiation theorem and obtain a Lebesgue point $x_0$ of one of these sets. Theefore, there exists $v \in F$ and $r>0$ such that, for every $m\in \N$, we find a sequence of points $x_1^m,...,x_{m}^m$ (dropping the superscript $m$ for clarity) with $v_{x_i}\approx v$ and $B\ni x_i\approx x_0+r\frac{i}{m}\pi(v)$ ($i=1,...,m$). Assume that $\varphi(x_1^m),\varphi(x_m^m)$ are bounded independently of $m\in \N$. 
Since $(x_i,y_i)\in \partial_{c_2}\varphi$ and $\varphi$ is $c_2$-convex, we have
\begin{align*}
    \varphi(x_1)\leq \varphi(x_m)+\sum_{i=1}^{m-1} c_2(x_{i+1},y_i)-c_2(x_i,y_i) 
\end{align*}
where $y_i$ is as in the definition of $v_{x_i}$. Since $(x_i,y_i)\in J^+\backslash I^+$ it is easy to see (in $\R^{1+n}$) that $c_2(x_{i+1},y_i)-c_2(x_i,y_i)\leq -\frac{C}{\sqrt{m}}$ for some uniform constant $C$ and large $m$ (here we make use of the fact that $\op{dist}(x_0,\supp(\nu))>0$). Inserting this into the inequality, we arrive at the contradiction
\begin{align*}
 \inf_{m\in \N}   \varphi(x^m_1)\leq \sup_{m\in \N} \varphi(x^m_m)-\lim_{m\to \infty} C \sqrt{m}=-\infty.
\end{align*}
This contradicts the boundedness $\inf_{m\in \N}   \varphi(x^m_1)$ and $\sup_{m\in \N}\varphi(x_m^m)$.

\item Now, we prove (ii).
From (i) and Definition \ref{daui9dhoafzhiofzhaio}, we deduce that $\pi$ is concentrated on the set $I^+\cap \partial_{c_2}\varphi$.
It follows that, for $\mu$-a.e.\ $x$, we have $I^+(x)\cap \partial_{c_2}\varphi(x)\neq \emptyset$. 
\end{enumerate}
\end{proof}

\begin{remark}\rm
Note that the fact that $\pi$ is concentrated on $I^+$ does not imply $\supp(\pi)\subseteq I^+$.
\end{remark}

Next, we aim to provide a simple criterion for the existence of $\pi$-solution, i.e.\ we prove Proposition \ref{abc}, which we restate below. First, we define what we mean by a causally compact set.

\begin{definition}\rm \label{causally}
A subset $A\subseteq M$ is called \emph{causally compact} if for any compact set $K\subseteq M$ the sets $J^+(K)\cap A$ and $J^-(K)\cap A$ are compact.
\end{definition}

\begin{proposition}\label{qayswsxdcre}
Let $(\mu,\nu) \in {\cal P}_\tau^+$ be a strictly timelike couple of probability measures. Assume that $\supp(\mu)$ is connected and that $\supp(\mu)$ and $\supp(\nu)$ are causally compact.

Then, for any $\pi\in \Gamma_o(\mu,\nu)$, there exists a $\pi$-solution.
\end{proposition}

\begin{proof}
The idea behind the proof is inspired by \cite{Kell}, Theorem 2.8. The construction of the $c_2$-convex function relies on the standard Rockafellar method, see for example \cite{Ambrosio}, Theorem 6.1.4. However, since $c_2$ is not real-valued, additional care is required to ensure the validity of the argument in our case.

Let $\pi\in \Gamma_o(\mu,\nu)$ be an optimal coupling. In particular, $\pi$ is causal. By Lemma \ref{abcde}, we find a $c_2$-monotone Borel set $\Gamma\subseteq \supp(\pi)\subseteq J^+$ on which $\pi$ is concentrated. By the inner regularity of $\pi$, we may assume that $\Gamma$ is $\sigma$-compact. Fix $(x_0,y_0)\in \Gamma$. We define the function
\begin{align*}
\varphi:M\to \overline \R,\ \varphi(x):=\sup\bigg\{\sum_{i=0}^k c_2(x_i,y_i)-c_2(x_{i+1},y_i)\bigg\},
\end{align*}
where the supremum is taken over all $k\in \N$ and all chains $(x_i,y_i)_{1\leq i\leq k}\subseteq \Gamma$. Here, $x_{k+1}:=x$.
We claim that this function is a $\pi$-solution.
\medskip

Observe that $\varphi$ is well-defined since $c_2(x_i,y_i)\in \R$. By construction, $\varphi$ is $c_2$-convex, and because $\Gamma$ is $c_2$-monotone, it satisfies $\varphi(x_0)=0$. Moreover (using the convention $\infty-\infty=-\infty$), we have
\begin{align*}
\varphi(x')\geq \varphi(x)+c_2(x,y)-c_2(x',y) \text{ for all } (x,y)\in \Gamma \text{ and } x'\in M.
\end{align*}
By the definition of the $c_2$-transform (and the conventions of the sum!), this implies 
\begin{align*}
\varphi^{c_2}(y)=c_2(x,y)+\varphi(x) \text{ for all } (x,y)\in \Gamma.
\end{align*}
 Thus, if we can show that $\varphi$ is finite on $p_1(\Gamma)$, it follows that 
\begin{align*}
    \varphi^{c_2}(y)-\varphi(x)=c_2(x,y) \text{ on } \Gamma, \text{ hence } \pi \text{-a.e.},
\end{align*}
which would coplete the proof. Thus, in the rest of the proof we will show that $\varphi$ is finite on $p_1(\Gamma)$.
\\

Let $(x_*,y_*)\in \Gamma$ be arbitrary but fixed. 
For $k\in \N$, we consider chains $(x_i,y_i)_{1\leq i\leq k}\subseteq \Gamma$ satisfying $x_{i+1}\in J^-(y_i)$ ($i\geq 1$) and $x_1\in J^-(y_*)$.
Define the following sets:
\begin{align*}
A_0:=\{y_*\},\ A_k:=\{y\mid \exists (x_i,y_i)_{1\leq i\leq k} \text{ as above with } y_k=y\} \text{ and } A:=\bigcup_{k=1}^\infty A_k.
\end{align*}
From this construction, we obtain
\begin{align}
p_2(p_1^{-1}(J^-(A_k))\cap \Gamma)=A_{k+1} \text{ for all } k\in \N_0, \label{ksaiajfsioaiv}
\end{align}
where $p_1$ and $p_2$ denote the projection on the first and second component.

We first claim that, for each $k\in \N_0$, $A_k$ and $J^-(A_k)$ are $\sigma$-compact sets.  
Indeed, obviously $A_0$ is compact and thus $J^-(A_0)$ is also $\sigma$-compact, since it is closed. Indeed, since $M$ is globally hyperbolic, the set $J^+=J^-$ is closed (see \cite{ONeill}, Chapter 14, Lemma 22), so that $J^-(K)$ is closed whenever $K$ is compact. Next assume that, for some $k$, the sets $A_{k}$ and $J^-(A_{k})$ are $\sigma$-compact. It follows that the set $(J^-(A_{k})\times M)\cap \Gamma$ is also $\sigma$-compact. Hence, $A_{k+1}=p_2(p_1^{-1}(J^-(A_k))\cap \Gamma)$ is also $\sigma$-compact as the continuous image of a $\sigma$-compact set. Thus, $J^-(A_{k+1})$ is also $\sigma$-compact, again by the global hyperbolicity of $M$. 

From the claim it follows that $A$ and, hence, also $J^-(A)$, are $\sigma$-compact. In particular, both $A$ and $J^-(A)$ are Borel sets.

Next we claim that $\nu(A)=\mu(J^-(A))$. Taking the union over all $k$ we get from \eqref{ksaiajfsioaiv} 
\begin{align*}
p_2(p_1^{-1}(J^-(A))\cap \Gamma)=A.
\end{align*}
We then use the fact that the first (resp. second) marginal of $\pi$ is $\mu$ (resp. $\nu)$ to deduce
\begin{align*}
\nu(A)=\pi\Big(p_2^{-1}\big(p_2(p_1^{-1}(J^-(A))\cap \Gamma)\big)\Big)
\geq \pi(p_1^{-1}(J^-(A))\cap \Gamma)
=\mu(J^-(A)).
\end{align*}
On the other hand, since $\pi$ is concentrated on $J^+$, we have
\begin{align*}
\nu(A)=\pi((M\times A)\cap J^+)=\pi(J^-(A)\times A)\leq \mu(J^-(A)).
\end{align*}
Both inequalities combined give the claim and the claim tells us that every causal coupling $\bar \pi\in \Gamma(\mu,\nu)$ has to couple $\mu\mres J^-(A)$ with $\nu\mres A$ in the sense that 
\begin{align*}
\bar \pi\mres (J^-(A)\times A)\in \Gamma(\mu\mres J^-(A),\nu\mres A).
\end{align*}
Here, $\mres$ means a measure restricted to a set.
Indeed, let $\bar \pi$ be such a causal coupling and $B\subseteq M$ be any Borel set. Denote $\tilde \pi:=\bar \pi\mres (J^-(A)\times A)$. Then
\begin{align*}
(p_1)_\#\tilde \pi(B)=\bar \pi((B\cap J^-(A))\times A)\leq \bar \pi((B\cap J^-(A))\times M)= (\mu\mres J^-(A))(B).
\end{align*}
On the other hand, we have by the causality of $\bar \pi$ that $(p_1)_\#\tilde \pi(M)=\tilde \pi(M\times M)=\bar \pi(J^-(A)\times A)=\bar \pi(M\times A)=\nu(A)=\mu(J^-(A))=(\mu\mres J^-(A))(M)$, and therefore must have equality for all $B$.
A similar argument shows that $(p_2)_\#\tilde \pi=\nu\mres A$.
\\

\noindent \textbf{Claim:} It holds $\supp(\mu)\subseteq I^-(A)$.

\noindent \textbf{Proof of claim:}
 By the connectedness of $\supp(\mu)$, it suffices to prove that $I^-(A)\cap \supp(\mu)$ is non-empty, closed and also open in $\supp(\mu)$. However, it is trivial that the set is  open in $\supp(\mu)$ (since $I^-(A)$ is open).
We pick a causal coupling $\bar \pi\in \Gamma(\mu,\nu)$ with $\supp(\bar \pi)\subseteq I^+$.

To prove that the set is non-empty, we choose $\bar x_*\in \supp(\mu)$ with $(\bar x_*,y_*)\in \supp(\bar \pi)\subseteq I^+$. This is possible since $y_*\in \supp(\nu)$ and $\supp(\mu)$ is causally compact, as a simple compactness argument shows. Thus $\bar x_*\in \supp(\mu)\cap I^-(A)\neq \emptyset$.

To prove that $\supp(\mu)\cap I^-(A)$ is closed in $\supp(\mu)$, let $(x_k)_k\subseteq \supp(\mu)\cap I^-(A)$ with $x_k\to x$. In particular, $x_k\in \supp(\mu\mres J^-(A))$. By the causal compactness of $\supp(\nu\mres A)$, we can choose, for each $k$, some $y_k\in \supp(\nu\mres A)$ with $(x_k,y_k)\in \supp(\bar \pi\mres (J^-(A)\times A))$ (here we use the fact that $\bar \pi\mres (J^-(A)\times A)\in \Gamma(\mu\mres J^-(A),\nu\mres A)$). We can assume that $(x_k,y_k)\in J^-(A)\times A$. Again by the causal compactness of $\supp(\nu)$ we have that, along a subsequence, $(x_k,y_k)\to (x,y)\in \supp(\bar \pi\mres (J^-(A)\times A))\subseteq I^+$ for some $y\in \supp(\nu)$. Thus, $x\in I^-(y)$ and, therefore, $x\in I^-(y_k)\subseteq I^-(A)$ for big $k$.
\hfill \checkmark
\bigskip

\noindent \textbf{Claim:}
We have $\varphi(x)\in \R$ for all $x\in p_1(\Gamma)$.

\noindent \textbf{Proof of claim:}
Let $x\in \supp(\mu)$ be arbitrary. Define the set $A$ with $(x_*,y_*):=(x_0,y_0)$ and choose a chain $(x_i,y_i)_{1\leq i\leq k}$ as in the definition of $A$ with $x\in J^-(y_k)$. The existence of such a chain follows from the preceding claim. By definition of $\varphi$, we obtain
\begin{align*}
\varphi(x)
\geq
\sum_{i=0}^k c_2(x_i,y_i)-c_2(x_{i+1},y_i)
>-\infty,\ \text{ where }x_{k+1}:=x.
\end{align*}
On the other hand, if $x\in p_1(\Gamma)$, we can define $A$ with $(x_*,y_*):=(x,y)\in \Gamma$ for some $y$. We obtain a chain $(x_i,y_i)_{1\leq i\leq k}$ as in the definition of $A$ with $x_{k+1}:=x_0\in J^-(y_k)$. Applying the definition of $\varphi$ we then find
\begin{align*}
0=\varphi(x_0)\geq \varphi(x)+\(c_2(x,y)-c_2(x_1,y)+\sum_{i=1}^k \(c_2(x_i,y_i)-c_2(x_{i+1},y_i)\)\).
\end{align*}
Since the expression in the brackets is a real number due to the properties of our chain, it follows that $\varphi(x)<\infty$.
\hfill \checkmark

This claim finally concludes the proof of the proposition.
\end{proof}

\section{First properties of $c_2$-convex functions - Proof of Theorem \ref{mfdghufsa} }
In this chapter, we will prove Theorem \ref{mfdghufsa}, which we restate here for convenience:

\begin{theorem} \label{Figalli}
Let $\varphi:M\to \overline \R$ be a $c_2$-convex function, and define
\begin{align*}
D:=\{x\in M\mid \varphi(x)\in \R\} \text{ and }
\Omega:=\op{int}(D).
\end{align*}
Then the following assertions hold:
\begin{enumerate}[(i)]
\item
$\varphi_{|\Omega}$ is locally bounded.
\item
$D\backslash \Omega$ is countably $n$-rectifiable.
\item
For each compact $K \subseteq \Omega$, the set of all $y\in M$ such that
\begin{align*}
\psi(y)-c_2(x,y)\geq \varphi(x)-1 \text{ for some } x\in K
\end{align*}
is relatively compact. Here, $\psi:=\varphi^{c_2}$.
\end{enumerate}
\end{theorem}

\begin{remark}\rm
For the definition of countable $n$-rectifiability, see subsection 7.4 in the appendix. As mentioned in the introduction, this theorem is the first step towards the regularity result stated in Theorem \ref{unghtrdsxcfghjukil}. 

The proof requires the concept of a cone. Therefore, we begin with the following definition.
\end{remark}

\begin{definition}\rm
Let $(V,\langle \cdot,\cdot \rangle)$ be a Hilbert space with induced norm $\norm{\cdot}$. The \emph{(open) Cone} in the direction $v\in V\backslash \{0\}$ with angle $\alpha\in [-1,1)$ is defined as
\begin{align*}
\op{Cone}(v,\alpha):=\{w\in V\backslash \{0\}\mid \langle {v}/{\norm{v}},w/{\norm{w}}\rangle> \alpha\}.
\end{align*}
\end{definition}

\begin{proof}
We proceed as in \cite{Figalli/Gigli}. 
\begin{enumerate}[(i)]
\item
\begin{enumerate}[(1)]
    \item 
First we prove that $\varphi_{|\Omega}$ is locally bounded from below. 

Let $x_0\in \Omega$. We need to show that there exists a neighborhood $U$ of $x_0$ such that $\varphi$ is bounded below on $U$. We argue by contradiction and assume that there exists a sequence 
\begin{align*}
(x_k)\subseteq \Omega \text{ such that } x_k\to x_0  \text{ and } \varphi(x_k)\to -\infty \ (k\to \infty).
\end{align*} 
We can pick $\bar x_0\in I^+(x_0)\cap \Omega$. Since $I^-(\bar x_0)$ is open and contains $x_0$, we have $x_k\in I^-(\bar x_0)$ for sufficiently large $k$. Since $\bar x_0\in \Omega$, we have $\varphi(\bar x_0)\in \R$, so we can find some $\bar y_0\in M$ such that
\begin{align*}
\psi(\bar y_0)-c_2(\bar x_0 ,\bar y_0)
\leq
\varphi(\bar x_0)
\leq
\psi(\bar y_0)-c_2(\bar x_0,\bar y_0)+1.
\end{align*}
We have set $\psi:=\varphi^{c_2}$. In particular, since $\varphi(\bar x_0)>-\infty$, we must have $\psi(\bar y_0)>-\infty$ and $\bar y_0\in J^+(\bar x_0)$. By definition of the $c_2$-transform, we obtain
\begin{align}
\varphi(x_k)\geq \psi(\bar y_0)-c_2(x_k,\bar y_0).\label{ouhaduz8hida}
\end{align}
Since $\bar y_0\in J^+(\bar x_0)\subseteq I^+(x_k)$ for large $k$, and using the continuity of $c_2$ on $J^+$, we conclude that $c_2(x_k,\bar y_0)$ is uniformly bounded for large $k$. Combining this with \eqref{ouhaduz8hida}, we obtain a contradiction to the assumption $\varphi(x_k)\to -\infty$.
Thus, we conclude that $\varphi_{|\Omega}$ is indeed locally bounded from below.

\item Now, following \cite{Figalli/Gigli}, we prove that $\varphi_{|\Omega}$ is also locally bounded from above. Pick $x_0\in \Omega$ and assume, for contradiction, that there exists a sequence
\begin{align*}
(x_k)\subseteq \Omega \text{ such that } x_k\to x_0  \text{ and } \varphi(x_k)\to \infty \ (k\to \infty).
\end{align*} 
 For each $k\in \N$, we find some $y_k\in J^+(x_k)$ satisfying
\begin{align*}
\psi(y_k)-c_2(x_k,y_k)
\leq
\varphi(x_k)
\leq
\psi(y_k)-c_2(x_k,y_k)+1.
\end{align*}
Since $\varphi(x_k)\to \infty$ and $c_2\geq 0$, it follows that $\psi(y_k)\to \infty$. We claim that $c_2(x_k,y_k)\to \infty$ as well. To see this, let $\bar x_0\in I^-(x_0)\cap \Omega$. Clearly,
\begin{align*}
\psi(y_k)-c_2(\bar x_0,y_k)\leq \varphi(\bar x_0)<\infty.
\end{align*}
Since $\psi(y_k)\to\infty$, we obtain $c_2(\bar x_0,y_k)\to \infty$. However, for sufficiently large $k$, we also have $x_k \in J^+(\bar x_0)$, which implies that $c_2(\bar x_0,x_k)$ is bounded due to the continuity of $c_2$ on $J^+$. But then we must have $c_2(x_k,y_k)\to \infty$, as claimed.

The claim, in turn, implies
\begin{align*}
c_1(x_k,y_k)=\sqrt{c_2(x_k,y_k)}\to \infty.
\end{align*}

According to Lemma \ref{minimizer}, we can pick for each $k$ a future pointing causal curve $\gamma_k:[0,c_1(x_k,y_k)]\to M$ such that 
\begin{align}
L_1(\dot \gamma_k(t))=d\tau(\dot \gamma_k(t))-|\dot \gamma_k(t)|_g=1, \label{hifuaojiso}
\end{align}
 which connects $x_k$ with $y_k$ while minimizing the action functionals ${\cal A}_1$ and ${\cal A}_2$.

Now, let $l>0$ such that, for any $x$ sufficiently close to $x_0$ and any $v\in \overline \C_x$ with $|v|_h\leq 2$, the curve $\exp_L(x,tv)$ is defined for all $t\in [0,l]$. Such an $l$ exists by Corollary \ref{l}. 
 By considering only large $k$, we may also assume that $l\leq c_1(x_k,y_k)$. We claim that 
\begin{align}
\inf_{C_k} \varphi\to \infty, \text{ where } C_k:=\{x\in M\mid \exists t\in [0,l], c_1(x,\gamma_k(t))\leq t/2\}. \label{c-k}
\end{align}

To prove this, let $x\in C_k$ and choose $t$ as in the definition of $C_k$. Using the fact that $c_1(\gamma_k(t),y_k)=c_1(x_k,y_k)-t$ (here we use $l\leq c_1(x_k,y_k)$), we obtain
\begin{align*}
\varphi(x)&\geq \psi(y_k)-c_1(x,y_k)^2
\geq
\psi(y_k)-(c_1(x,\gamma_k(t))+c_1(\gamma_k(t),y_k))^2
\\[10pt]
&\geq
\psi(y_k)-(c_1(x_k,y_k)-t/2)^2
\\[10pt]
&\geq
\psi(y_k)-c_1(x_k,y_k)^2+tc_1(x_k,y_k)-t^2/4
\\[10pt]
&\geq \varphi(x_k)-1+tc_1(x_k,y_k)-t^2/4.
\end{align*}
By assumption, this expression diverges to $\infty$, proving the claim. From the claim, we will easily deduce a contradiction. 

Indeed, without loss of generality, assume that there exists a limit $v$ of $\dot \gamma_k(0)$ (in the topology of the tangent bundle). This is possible since $|\dot \gamma_k(0)|_h$ is bounded by $2$, thanks to \eqref{splitting} and \eqref{hifuaojiso}. Now, define
\begin{align*}
C_\infty:=\{x\in M\mid \exists t\in[0,l],\ x\in I^-(\exp_L(x_0,tv)),\ c_1(x,\exp_L(x_0,tv))<t/2\}.
\end{align*}
 Since $\gamma_k(t)=\exp_L(x_k,t\dot \gamma_k(0))$ (by Lemma \ref{lii}), and $\exp_L$ is continuous, it follows that for large $k$, any $x\in C_\infty$ also belongs to $C_k$. Thus, from \eqref{c-k}, we conclude that $\varphi_{|C_\infty}\equiv \infty$.

However, since any neighborhood of $x_0$ intersects $C_\infty$ and $\varphi\equiv \infty$ on $C_\infty$, this finally gives a contradiction.
\end{enumerate}

\item Now we prove the second part of the theorem. Denote by $A^\pm$ the set of all $x\in D\backslash \Omega$ such that there exists a sequence $x_k\to x$ with $\varphi(x_k)=\pm\infty$ for all $k\in \N$.
Then, we have $D\backslash \Omega=A^+\cup A^-$. Hence, it suffices to prove that both $A^+$ and $A^-$ are countably $n$-rectifiable.

In both cases, we will use a result from geometric measure theory, see Lemma \ref{rect} in the appendix.

\begin{enumerate}[(1)]
\item

We begin with the simpler set $A^-$. Let $x_0\in A^-$ and let $(U,\phi)$ be any chart around $x_0$ such that the pullback metric $((\phi^{-1})^*g)_{\phi(x_0)}$ is the standard Minkowski metric on $\R^{1+n}$. By continuity of $g$, there exists a neighborhood $V\subseteq U$ of $x_0$ such that, for all $x\in V$,
\begin{align}
((\phi^{-1})^* g)_{\phi(x)}\leq g_0:=-\frac{1}{2} (dx_0)^2+\sum_{i=1}^n 2 (dx_i)^2 \text{ as quadratic forms.} \label{daoldjail}
\end{align}
By possibly shrinking $U$, we may assume $U=V$.
Now, take any $x\in U\cap A^-$ and choose a sequence $x_k\to x$ with $\varphi(x_k)=-\infty$ for all $k$. Let $\bar x\in I^+(x)$ be arbitrary. Then $\bar x\in I^+(x_k)$ for sufficiently large $k$. In particular, $c_2(\bar x,y)=\infty$ if $c_2(x_k,y)=\infty$. From this, it easily follows that $\varphi(\bar x)=-\infty$. Hence,
\begin{align*}
I^+(x)\cap A^-=\emptyset.
\end{align*}

Using $g_0$ and the relation \eqref{daoldjail}, it is starightforward to verify that there exists a constant $1\approx \alpha<1$, independent of $x$, such that
\begin{align*}
\(\phi(x)+\op{Cone}(e_0,\alpha)\)\cap \phi(U)\subseteq \phi((A^-)^c\cap U).
\end{align*}
This yields
\begin{align*}
T_{\phi(x)}(\phi(A^-\cap U))\subseteq \R^{1+n}\backslash \op{Cone}(e_0,\alpha),
\end{align*}
where $T_{\phi(x)}(\phi(A^-\cap U))$ denotes the tangent cone of $\phi(A^-\cap U)$ at $\phi(x)$, see Definition \ref{tangentcone}. Now, applying the result from geometric measure theory (see Lemma \ref{rect}), we conclude that $\phi(A^-\cap U)$ is countably $n$-rectifiable. Hence, $A^-$ is countably $n$-rectifiable by Lemma \ref{mnhas}.

\item Now we turn to the set $A^+$. We devide the set again into
\begin{align*}
    A^+=A^+_1\cup A^+_2:=\{x\in A^+\mid I^-(x)\subseteq D^c\} \cup \{x\in A^+\mid I^-(x)\cap D\neq \emptyset\}.
\end{align*}
Using the same arguments as at the end of (1), we deduce that the set $A^+_1$ is countably $n$-rectifiable. Next, we consider $A^+_2$ and follow the strategy from \cite{Figalli/Gigli}. Let $x_0\in A^+$ and let $l>0$ and $(U,\phi)$ be any chart around $x_0$ such that
the exponential map $\exp_{L}(x,tv)$ is defined for all $x\in U$, $v\in \C_x$ with $|v|_h\leq 2$ and $t\leq l$.
We can repeat the proof of (i) to obtain that, for any $x\in U\cap A^+_2$,
\begin{align}
\varphi(\bar x)=\infty \text{ for all } \bar x\in C_\infty(x), \label{fdafafafwefwfew}
\end{align}
where
\begin{align}
C_\infty(x)=\{\bar x\in M\mid \exists t\in[0,l],\ \bar x\in I^-(\exp_L(x,tv)),\ c_1(\bar x,\exp_L(x,tv))<t/2\} \label{fdafafafwefwfew1}
\end{align}
for some $v\in \C_{x}$ with $|v|_h\leq 2$
(here we use the fact that we can find $\bar x\in I^-(x)$ with $\varphi(\bar x)\in \R$, which plays the role of $\bar x_0$ in the proof of (i)).

Without loss of generality, we can assume that \eqref{daoldjail} holds and that $\tau\circ \phi^{-1}$ is Lipschitz with constant $L$.
Using Lemma \ref{mnhas}, we need to prove that $\phi(U\cap A_2^+)\subseteq \R^{1+n}$ is countably $n$-rectifiable.

We choose $a>0$ with $-\frac{1}{2}(1-2a)^2+8na^2<0$ and $\ep>0$ with $La\ep+L\ep<\frac{1}{2}$.
Using \eqref{fdafafafwefwfew}, \eqref{fdafafafwefwfew1} and Corollary \ref{fdafafafwefwfew2}, it suffices to prove that for each $x\in A_2^+\cap U$ and $v\in \C_x$ as above, there exists $w\in \R^{1+n}$ such that
\begin{align}
 \phi(x)+B_{a\ep t}(tw) \subseteq \phi(C_\infty(x)\cap U) \text{ for small } t. \label{sasaasoaoaoaoa}
\end{align}

 Now, let $x\in A_2^+\cap U$ be given and $v\in \C_x$ as above. Let $\R^{1+n}\ni u:=d_x\phi(v)$, and set $w:=(u_0-\ep,u_1,...,u_n)$. We first claim that for small $t$, the following holds:
\begin{align}
\phi(x)+B_{at\ep}(tw)\subseteq \phi(I^-(\exp_L(x,tv))\cap U). \label{djiaoda}
\end{align}
Indeed, let $c(t):=\phi(\exp_L(x,tv))$ and $y\in \R^{1+n}$ with
\begin{align}
    |y-(\phi(x)+tw)|<at\ep, \label{ofisifaztdgaui}
\end{align}
and observe that for small $t$, the line $\gamma(s)$, defined by $\gamma:[0,1]\to \R^{1+n},\  s\mapsto y+s(c(t)-y)$, lies in $\phi(U)$. It suffices to prove that this curve is future pointing causal w.r.t.\ $g_0$. The derivative is $\gamma'(s)=\gamma'(0)=c(t)-y=\phi(x)+t\dot c(0)+o(t)-y=(tu-tw)+(\phi(x)+tw-y)+o(t)$. By \eqref{ofisifaztdgaui} and the choice of $a$, we obtain for small values of $t$
\begin{align*}
g_0(\gamma'(s),\gamma'(s))
&=-\frac{1}{2}(t\ep+(\phi(x)+tw-y)_0+o(t))^2+\sum_{i=1}^n 2 ((\phi(x)+tw-y)_i+o(t))^2
\\[10pt]
&\leq -\frac{1}{2} (t\ep(1-2a))^2+2n(2at\ep)^2
\\[10pt]
&=(t\ep)^2(-\frac{1}{2}(1-2a)^2+8na^2)<0.
\end{align*}
This proves \eqref{djiaoda}. To prove \eqref{sasaasoaoaoaoa} we observe that since $(\phi^{-1}(y),\exp_{L}(x,tv))\in I^+$ and $\tau \circ \phi^{-1}$ is Lipschitz with constant $L$, we have
\begin{align*}
c_1(\phi^{-1}(y),\exp_{L}(x,tv))&\leq \tau(\exp_{L}(x,tv))-\tau(\phi^{-1}(y))
\\[10pt]
&\leq L |y-c(t)|
\\[10pt]
&\leq L|y-(\phi(x)+tw)|+L|\phi(x)+tw-c(t)|
\\[10pt]
&\leq Lat\ep +Lt|w-u|+Lo(t)
\\[10pt]
&\leq Lat\ep +Lt\ep+ Lo(t)
< t/2
\end{align*}
for small $t$, where we have used \eqref{ofisifaztdgaui} in the fourth step and the definition of $w$ and $\ep$ in the fifth and last step. This finally proves \eqref{sasaasoaoaoaoa}.

\end{enumerate}
\item
Let $K\subseteq \subseteq \Omega$ be compact, and choose an open set $U\subseteq M$ such that $K\subseteq U\subseteq \subseteq \Omega$. We aim to prove that $d_h(x,y)$ is uniformly bounded for $(x,y)\in K\times M$ satisfying the condition
\begin{align}
\varphi(x)\leq \psi(y)-c_2(x,y)+1. \label{sasahdaoudiaj}
\end{align}
Since $K$ is compact and the exponential map $\exp_L$ is continuous, we can choose $l>0$ such that, for all $x\in K$, all $v\in \C_x$ with $|v|_h\leq 2$, and all $t\leq l$, $\exp_L(x,tv)$ is defined and lies within the open set $U$.

Next, assume by contradiction that we can find a sequence $(x_k,y_k)\subseteq K\times M$ satisfying condition \eqref{sasahdaoudiaj} such that $d_h(x_k,y_k)\to \infty$. We argue as in part (i) of the proof. Since $y_k\in J^+(x_k)$, we have $c_1(x_k,y_k)\geq d_h(x_k,y_k)/2\geq l$ for large $k$, thanks to \eqref{splitting}. We then find a future pointing causal curve $\gamma_k:[0,c_1(x_k,y_k)]\to M$, connecting $x_k$ to $y_k$, such that $d\tau(\dot \gamma_k(t))-|\dot \gamma_k(t)|_g=1$, and which minimizes the action functionals ${\cal A}_1$ and ${\cal A}_2$. Moreover, $\gamma_k(t)=\exp_L(x_k,tv_k)$ for some $v_k\in \C_{x_k}$ with $|v_k|_h\leq 2$. As in part (i), we can deduce that
\begin{align*}
\varphi(\gamma_k(l))
\geq \varphi(x_k)-1+lc_1(x_k,y_k)-l^2/4.
\end{align*}
Since $\gamma_k(l)\in U$ and $\varphi$ is bounded on $U$ by (i), it follows that $c_1(x_k,y_k)$ is bounded. 
Therefore, using \eqref{splitting}, we deduce that $d_h(x_k,y_k)\leq 2c_1(x_k,y_k)$ is also bounded. This is a contradiction.

Thus, the set of all $y\in M$ satisfying \eqref{sasahdaoudiaj} for some $x\in K$ is bounded w.r.t.\ $h$. By the Hopf-Rinow theorem, this means that this set is relatively compact due to the completeness of $h$.
\end{enumerate}
\end{proof}

\section{Semiconvexity and existence/uniqueness of an optimal transport map}
In this chapter, we aim to prove  Theorem \ref{unghtrdsxcfghjukil} and Corollary \ref{unghtrdsxcfghjukil3}. For this purpose, we will always assume that the conditions of both the theorem and the corollary are satisfied. That is, we assume throughout this chapter the following:

\begin{setting} 
Let $(\mu,\nu)\in {\cal P}_\tau^+$ be a causal couple of probability measures such that $\supp(\mu)\cap \supp(\nu)=\emptyset$. Furthermore, assume that $\mu$ is absolutely continuous w.r.t.\ the Lebesgue measure on $M$. Let $\pi\in \Gamma_o(\mu,\nu)$, and assume that $\varphi:M\to \R \cup \{\pm\infty\}$ is a $\pi$-solution. We also denote $\psi:=\varphi^{c_2}$.
\end{setting}

We now restate Theorem \ref{unghtrdsxcfghjukil} and Corollary \ref{unghtrdsxcfghjukil3}:

\begin{theorem} \label{hudaodiaoiii}
There exists an open set $\Omega_1\subseteq \Omega$ of full $\mu$-measure such that $\varphi$ is locally semiconvex on $\Omega_1$.
\end{theorem}

\begin{corollary}\label{cd}
The following properties hold:
\begin{enumerate}[(i)]
\item  $\pi$ is induced by a transport map $T$.
More precisely, $\mu$-a.e., $T(x)$ is uniquely defined by the equation
\begin{align*}
    \frac{\partial c_2}{\partial x}(x,T(x))=-d_x\varphi.
\end{align*}
\item 
If there exists an optimal coupling different from $\pi$, then there also exists an optimal coupling $\pi'$ that does not admit a $\pi'$-solution.
\end{enumerate}
\end{corollary}

\begin{remark}\rm
\begin{enumerate}[(a)]
\item Our definition of local semiconvexity differs from the one used by many other authors and is often referred to as locally subdifferentiable. See Definition \ref{opaudowzuafdddf} in the appendix for our notion.
\item
Let us explain the strategy behind the proof of Theorem \ref{hudaodiaoiii}. 
As the result is motivated by \cite{Figalli/Gigli} let us first quickly explain the strategy behind the proof in \cite{Figalli/Gigli}, which deals with the cost function $d_R(x,y)^2$, where $d_R$ is the Riemannian distance on a complete and connected Riemannian manifold.

First one proves, as here, Theorem \ref{Figalli} for the Riemannian case. Let us use the same notation as for our case but with an indexed $R$.
By (iii) of this theorem, if $x_0\in \Omega_R$ and $U\subseteq \subseteq \Omega_R$ is some compact neighborhood of $x_0$ the set $A$ of all $y\in M$ such that
\begin{align*}
\varphi_R(x)\leq \psi_R(y)-d_R^2(x,y)+1 \text{ for some } x\in U
\end{align*}
is relatively compact. Here, $\psi_R:=\varphi_R^{d_R^2}$.
Hence, in $U$, $\varphi_R$ is given as the supremum of the family of functions $(\psi_R(y)-d_R^2(\cdot,y))_{y\in A}$.
Since $A$ is relatively compact and $d_R^2$ is locally semiconcave on $M\times M$, this family is actually uniformly locally semiconvex (see \cite{Fathi/Figalli}, proposition A.17) so that $\varphi_R$ is locally semiconvex on $U$ as the finite supremum of a uniformly locally semiconvex family of functions (see \cite{Fathi/Figalli}, proposition A.16).

Let us now return to the Lorentzian case and the cost function $c_2$.
Unfortunately, we cannot follow this approach since $c_2$ is not locally semiconcave on $M\times M$. However, we will be able to prove that there exists an open set $\Omega_1\subseteq \Omega$ of full $\mu$-measure, such that, if $x_0\in \Omega_1$ there is an open neighborhood $x_0\in U\subseteq \Omega_1$ and some $\delta>0$ such that, for all $x\in U$:
\begin{align*}
\varphi(x)=\sup\{\psi(y)-c(x,y)\mid d(x,y)\geq \delta\}.
\end{align*}
Roughly speaking, this shows that, locally in $\Omega_1$, the $c_2$-subdifferential of $\varphi$ is bounded away from $\partial J^+$. 
Since $c_2$ is locally semiconcave on $I^+$ (see Proposition \ref{injiuhgb}) this will allow us to prove the local semiconvexity of $\varphi$ on $U$ and, hence, on $\Omega_1$. The following two theorems (Theorem \ref{gafuiahfauhfaohfoafjaoi} and Theorem \ref{gafuiahfauhfaohfoafjaoii}) will be the main steps in the proof of Theorem \ref{hudaodiaoiii}. Let us mention that the idea (for the theorem and for the proof) for Theorem \ref{gafuiahfauhfaohfoafjaoii} comes from \cite{Kell}.
\end{enumerate}
\end{remark}

\begin{definition}\rm\label{deff}
We define the open set $\Omega_0$ as the set of all $x_0\in \Omega$ such that the following holds: There exists $\delta>0$ and a neighborhood $x_0\in U\subseteq \Omega$ such that, for all $x\in U$,
\begin{align*}
\sup\{\psi(y)-c_2(x,y)\mid d_h(x,y)\leq \delta\}<\varphi(x).
\end{align*}
\end{definition}

\begin{theorem} \label{gafuiahfauhfaohfoafjaoi}
It holds $\mu(\Omega_0)=1$.
\end{theorem}

\begin{theorem}\label{gafuiahfauhfaohfoafjaoii}
For ${\cal L}$-a.e. $x\in \Omega_0$ there exists $\delta=\delta(x)$ such that:
\begin{align*}
\sup\{\psi(y)-c_2(x,y)\mid  d(x,y)\leq \delta\}< \varphi(x).
\end{align*}
Here, $\L$ denotes the Lebesgue measure on $M$.
\end{theorem}

{\center {\large{\textbf{Proof of Theorem \ref{gafuiahfauhfaohfoafjaoi}}}}}
\bigskip

We have to prepare a little bit for the proof and start with a simple geometric lemma.

\begin{lemma}\label{cone}
Let $(V,\langle \cdot,\cdot \rangle)$ be a Hilbert space with induced norm $\norm{\cdot}$. Let $v\in V$ with $\norm{v}=1$. If $w\in \op{Cone}(v,\alpha)$ then
\begin{align*}
\norm{w-\norm{w}v}\leq 2\sqrt{1-\alpha}\norm{w}.
\end{align*}
\end{lemma}
\begin{proof}
If $\norm{w}=1$ we have
\begin{align*}
\norm{w-v}^2=\norm{w}^2+\norm{v}^2-2\langle v,w\rangle
\leq 2-2\alpha\leq 4(1-\alpha).
\end{align*}
This yields $\norm{w-v}\leq 2\sqrt{1-\alpha}$ and proves the lemma in the first case.
The general case follows from this special case and the positive homogeinity of $\norm{w-\norm{w}v}$ in $w$.
\end{proof}

\begin{notation}\rm
    Recall that $\exp$ always denotes the exponential function w.r.t.\ the Lorentzian metric $g$.
\end{notation}

\begin{proposition} \label{dokhmbjuf}
For $\mu$-a.e.\ $x_0\in \Omega$ there exist $r,\delta>0$, $\alpha\in [-1,1)$ and $v_0\in T_{x_0}M$ with $|v_0|_h=1$ such that for any $w\in \op{Cone}(v_0,\alpha)\cap B_r(0)$ we have for $x:=\exp_{x_0}(w)$:
\begin{align*}
\sup\{\psi(y)-c_2(x,y)\mid d_h(x,y)\leq \delta\}<\varphi(x).
\end{align*}
Here, $B_r(0)$ denotes the open ball centered at $0$ of radius $r$ (w.r.t. $|\cdot|_h$).
\end{proposition}
\begin{proof}
From Lemma \ref{gbaiw}, Proposition \ref{hiuajdoaildkao} and Theorem \ref{Figalli}(ii) we know that the set of all $x_0\in \Omega$ such that there is $y_0\in I^+(x_0)$ with $y_0\in\partial_{c_2}\varphi(x_0)$ is of full $\mu$-measure. Pick such an $x_0$ and a corresponding $y_0$.
\\

\noindent \textbf{Step 1:} Let $\ep>0$. Assume that $x\in \Omega\cap I^+(x_0)$ is close enough to $x_0$ so that $y_0\in I^+(x)$ and let $y\in J^+(x)$ be such that
\begin{align}
    \psi(y)-c_2(x,y)\geq \varphi(x)-\ep\geq \psi(y_0)-c_2(x,y_0)-\ep, \label{doadaodjuaopidadioaadpao}
\end{align}
where the second inequality holds true by the $c_2$-convexity of $\varphi$.
On the other hand, exchanging the roles of $x_0$ and $x$ and using $y_0\in \partial_{c_2}\varphi(x_0)$, we also have
\begin{align}
    \psi(y_0)-c_2(x_0,y_0)=\varphi(x_0)\geq \psi(y)-c_2(x_0,y). \label{doadaodjuaopidadioaadpao1}
\end{align}
Since $y_0\in I^+(x)$ and $y\in J^+(x)\subseteq I^+(x_0)$ and $x_0,x\in \Omega$ all the numbers appearing in the two (in)equalities above are finite. Thus, substracting \eqref{doadaodjuaopidadioaadpao} from \eqref{doadaodjuaopidadioaadpao1} gives
\begin{align}
    c_2(x,y_0)-c_2(x_0,y_0)+\ep \geq c_2(x,y)-c_2(x_0,y). \label{doadaodjuaopidadioaadpao2}
\end{align}
This condition (which basically corresponds to the $c_2$-monotonicity of $\partial_{c_2}\varphi$) is the starting point for our argument and in step 2 we want to derive a contradiction from it.
\\

\noindent \textbf{Step 2:}
Let $\gamma:[0,1]\to M$ be a minimizing curve for ${\cal A}_2$ between $x_0$ and $y_0$ and set 
\begin{align*}
v_0:=\frac{\dot \gamma(0)}{|\dot \gamma(0)|_h}.
\end{align*}
We assume that, for any $k\in \N$, there is $w_k\in \op{Cone}\big(v_0,1-\frac{1}{k}\big)\cap B_{\frac1k}(0)$ such that, for $x_k=\exp_{x_0}(w_k)\in M$ we find a sequence $y_{k,j}\in J^+(x_k)$ with
\begin{align*}
   d_h(x_k,y_{k,j})\leq \frac{1}{k},\ \psi(y_{k,j})-c_2(x_k,y_{k,j})\geq \varphi(x_k)-\frac{1}{j}.
\end{align*}
Now observe that, since $v_0$ is future-directed timelike by Proposition \ref{minimizer}, also $w_k$ is future-directed timelike for big $k$. Hence, $x_k\in I^+(x_0)\cap \Omega$ and $y_0\in I^+(x_k)$ for big $k$ and step 1 is applicable.
Since $y_0\in I^+(x_0)$ we know a superdifferential for the mapping ${c_2}_{|I^+}(\cdot,y_0)$ at $x=x_0$ is given by \eqref{super}. In particular, 
\begin{align}
c_2(x_k,y_0)-c_2(x_0,y_0)
\leq
-\frac{\partial L_2}{\partial v}(x_0,\dot \gamma(0))(w_k)+o(|w_k|_h), \label{polkj}
\end{align}
as $k\to \infty$.
A straight forward computation shows that
\begin{align*}
\frac{\partial L_2}{\partial v}(x_0,\dot \gamma(0))=2L_1(x_0,\dot \gamma(0))\cdot \(d_{x_0}\tau(\cdot)-\frac{-g_{x_0}(\dot \gamma(0),\cdot)}{|\dot \gamma(0)|_g}\).
\end{align*}
Clearly, $\frac{\partial L_2}{\partial v}(x_0,\dot \gamma(0))(|w_k|_h\dot \gamma(0))=2|w_k|_hL_2(x_0,\dot \gamma(0))$ and the latter is equal to $2|w_k|_h c_2(x_0,y_0)$ because $L_2(\gamma(t),\dot \gamma(t))$ is constant due to Proposition \ref{minimizer}. Therefore, using the definition of $v_0$,
\begin{align*}
\frac{\partial L_2}{\partial v}(x_0,\dot \gamma(0))(|w_k|_hv_0)
=
\frac{2|w_k|_h}{|\dot \gamma(0)|_h} c_2(x_0,y_0).
\end{align*}
 Since $|w_k-|w_k|_hv_0|_h\leq 2\sqrt{1/k} |w_k|_h$ thanks to Lemma \ref{cone} we have that
\begin{align*}
\bigg|\frac{\partial L_2}{\partial v}(x_0,\dot \gamma(0))(w_k)-\frac{2|w_k|_h}{|\dot \gamma(0)|_h}c_2(x_0,y_0)\bigg|_h\leq 2C \sqrt{1/k} |w_k|_h=o(|w_k|_h),
\end{align*}
as $k\to \infty$, where $C$ stands for the supremums norm of the linear map $\frac{\partial L_2}{\partial v}(x_0,\dot \gamma(0))$.
Therefore, using \eqref{polkj},
\begin{align}
c_2(x_k,y_0)-c_2(x_0,y_0)\leq -\frac{2|w_k|_h}{|\dot \gamma(0)|_h}c_2(x_0,y_0) +o(|w_k|_h)
\label{huiapfppfood}
\end{align}
as $k\to \infty$.
On the other hand, to obtain a contradiction from \eqref{doadaodjuaopidadioaadpao2} we still have to estimate $c_2(x_k,y_{k,j})-c_2(x_0,y_{k,j})$. With the definition of $c_2$ and the fact that $d(x_0,y_{k,j})\geq d(x_k,y_{k,j})$ we observe that
\begin{align*}
c_2(x_0,y_{k,j})=\(\tau(y_{k,j})-\tau(x_0)-d(x_0,y_{k,j})\)^2\leq\(\tau(y_{k,j})-\tau(x_0)-d(x_k,y_{k,j})\)^2.
\end{align*}
Then, 
\begin{align}
\nonumber
c_2(x_0,y_{k,j})-c_2(x_k,y_{k,j})&\leq
\(\tau(y_{k,j})-\tau(x_0)-d(x_k,y_{k,j})\)^2-\(\tau(y_{k,j})-\tau(x_k)-d(x_k,y_{k,j})\)^2 
\\[10pt]
&\nonumber
\leq
2(\tau(y_{k,j})-\tau(x_0)-d(x_k,y_{k,j}))(\tau(x_k)-\tau(x_0))
\\[10pt]
&\leq 2(\tau(y_{k,j})-\tau(x_0))(\tau(x_k)-\tau(x_0))
= 
o(|w_k|_h) \label{aaaaaaaaaaaa}
\end{align}
as $k\to \infty$,
where we made use of the fact that $\tau\circ \exp_{x_0}$ is a Lipschitz map on a small neighborhood of $0\in T_{x_0}M$.
Now we obtain from \eqref{huiapfppfood} and \eqref{aaaaaaaaaaaa} that, if $k$ is big enough,
\begin{align*}
c_2(x_k,y_0)-c_2(x_0,y_0) 
\leq c_2(x_k,y_{k,j})-c_2(x_0,y_{k,j})-\frac{|w_k|_h}{|\dot \gamma(0)|_h}c_2(x_0,y_0).
\end{align*}
Choosing $k$ big such that the above inequality holds and choosing $j$ such that $1/j<\frac{|w_k|_h}{|\dot \gamma(0)|_h}c_2(x_0,y_0)$, this gives a contradiction in view of in \eqref{doadaodjuaopidadioaadpao2} (with $\ep:=1/j$).  Hence, we proved the proposition.
\end{proof}

\begin{definition}\rm
Denote by $A$ the set of all $x_0\in  \Omega$ for which the above proposition holds true. We define the open set
\begin{align*}
\tilde \Omega_{0}:=\bigcup_{x_0\in A} \exp_{x_0}(\op{Cone}(v_0,\alpha)\cap B_r(0))\cap \Omega\subseteq \Omega
\end{align*}
Observe that, of course, $v_0=v_0(x_0)$, $\alpha=\alpha(x_0)$ and $r=r(x_0)$.
\end{definition}

\begin{corollary}\label{iuadhoiapaopao}
We have $\mu(\tilde \Omega_0)=1$.
\end{corollary}
\begin{proof}
We assume that $\mu(\tilde \Omega_0)<1$. Then the set $B:=A\backslash \tilde \Omega_0$ has positive $\mu$-measure. By the inner regularity of $\mu$ there exists a compact set $K\subseteq B$ with positive $\mu$-measure. Since $\mu$ is absolutely continuous w.r.t.\ the Lebesgue measure it follows that $K$ has positive Lebesgue measure.
Since $K$ is clearly a measurable set, we can find a Lebesgue point $x_0$ for $K$.
Since $x_0\in A$ there exist $r,\alpha>0$ and $v_0\in T_{x_0}M$ with $|v_0|_h=1$ such that
\begin{align*}
\exp_{x_0}(\op{Cone}(v_,\alpha)\cap B_r(0))\cap \Omega\subseteq \tilde \Omega_0.
\end{align*}
Obviously, this contradicts the fact that $x_0$ is a Lebesgue point for $K\subseteq \tilde \Omega_0^c$.
\end{proof}

\begin{remark}\rm
It is interesting to note that the above proof makes strong use of the fact that $\mu$ is absolutely continuous w.r.t.\ the Lebesgue measure on $M$ and that it does not work if for example $\mu$ is only absolutely continuous w.r.t.\ the Lebesgue measure on a smooth hypersurface.
\end{remark}

\begin{corollary}
We have $\tilde \Omega_0\subseteq \Omega_0$. 
\end{corollary}
\begin{proof}
This follows immediately from the definition of $\Omega_0$ and from Proposition \ref{dokhmbjuf}.
\end{proof}

\begin{proof}[Proof of Theorem \ref{gafuiahfauhfaohfoafjaoi}]
Obvious from the above corollary.
\end{proof}
\bigskip
\bigskip

{\center {\large{\textbf{Proof of Theorem \ref{gafuiahfauhfaohfoafjaoii}}}}}
\bigskip

We again start with a few lemmas. The first lemma is taken from \cite{Kell} and will be needed in one step in the proof of Theorem \ref{gafuiahfauhfaohfoafjaoii}.

\begin{lemma} \label{pkspadkpa}
Let $a,b\in \R$, $\ep>0$ and a Borel measurable set $B\subseteq [a,b]$ be given with ${\cal L}^1(B)\geq \ep (b-a)$. Then for all $k\in \N$ there exists $\{t_i\}_{1\leq i\leq k}\subseteq B$ with $t_1<... <t_k$ and $t_{i+1}-t_i\geq \frac{\ep}{2k}(b-a)$.
\end{lemma}

The next lemma introduces a smooth family of orthonormal frames on a convex set (see Definition \ref{convex} and Definition \ref{ONF}) which allows us to compare different tangent spaces. It is not surprising that the orthonormal frames will be constructed as the evaluation of parallel vector fields along geodesics:

\begin{lemma} \label{frame}
Let $U$ be a convex set.

Then there exist smooth maps $e_i:U\times U\to TU$ ($i=0,...,n)$ with the following properties:
\begin{enumerate}[(i)]
\item
For all $x,y\in U$, the set $\{e_0(x,y),...,e_n(x,y)\}$ is an orthonormal basis of the tangent space $T_yM$ such that $e_0(x,y)$ is timelike.
\item
For all $i=0,...,n$ and all $x,y\in U$ the tangent vector $e_i(x,y)$ arises as the parallel transport $V$ along the unique (up to reparametrization) geodesic inside $U$ between $x$ and $y$ with $V(0)=e_i(x,x)$.
\end{enumerate}
\end{lemma}
\begin{proof}
    Easy, see Lemma \ref{dghaidghadhaodfjuapoda}.
\end{proof}

\begin{corollary}\label{cor}
Let $U$ be as above. Let $x,y\in U$ and denote by $\gamma:[0,1]\to U$ the unique geodesic connecting $x$ and $y$.
Let $\exp_{x}^{-1}(y)=\dot \gamma(0)=:\sum_{i=0}^n \lambda_i e_i(x,x)$. Then it holds
\begin{align*}
\exp_y^{-1}(x)=-\dot \gamma(1)\overset{!}{=}\sum_{i=0}^n -\lambda_i e_i(x,y).
\end{align*}
\end{corollary}
\begin{proof}
The first claimed equality follows from the definition of the exponential map. For the second, observe that we can write
\begin{align*}
-\dot \gamma(1)=\sum_{i=0}^n \mu_i e_i(x,y) \text{ for } 
\begin{cases}
\mu_0=-g_y(-\dot \gamma(1),e_0(x,y)),\ &\text{and }
\\\\
\mu_i=g_y(-\dot \gamma(1),e_i(x,y)),\ &i\geq 1.
\end{cases}
\end{align*}
Let $i\geq 1$. By construction of $e_i$, $V(t):=e_i(x,\gamma(t))$ is a parallel vector field along $\gamma$. Hence, since the Levi-Civita connection is compatible with the metric,
\begin{align*}
\frac{d}{dt} g_{\gamma(t)}(-\dot \gamma(t),V(t))
=
g_{\gamma(t)}\(-\frac{\nabla \dot \gamma}{dt}(t),V(t)\)+g_{\gamma(t)}\(-\dot \gamma(t),\frac{\nabla V}{dt}(t)\)=0,
\end{align*}
since $V$ and $\dot \gamma$ are parallel.
This shows that
\begin{align*}
\mu_i=g_y(-\dot \gamma(1),e_i(x,y))=g_x(-\dot \gamma(0),e_i(x,x))=-\lambda_i.
\end{align*}
Analogously, $\mu_0=-\lambda_0$ and we conclude the proof of the corollary.
\end{proof}

\begin{definition}\rm\label{lkjh}
Let $U$ be as above.
We define the smooth projection
\begin{align*}
\pi:TU\to TU,\ \(x,\sum_{i=0}^n \lambda_i e_i(x,x)\)\mapsto \(x,\sum_{i=1}^n \lambda_i e_i(x,x)\).
\end{align*}
With some abuse of notation we will also write $\pi$ for the corresponding mapping on the tangent spaces $\pi:T_xM\to T_xM$, $x\in U$.
\end{definition}

The following technical lemma provides us with some uniform esimates which we will need in the proof of Theorem \ref{gafuiahfauhfaohfoafjaoii}. It makes sense to skip this lemma at first reading and return to it when needed in the proof of \ref{gafuiahfauhfaohfoafjaoii}. We prove the lemma here before the proof to obtain uniform constants which do not depend on the construction in the proof of \ref{gafuiahfauhfaohfoafjaoii}.

\begin{lemma} \label{lemma}
Let $x_0\in \Omega_0$, where $\Omega_0$ in as in Definition \ref{deff}. Then we can find an open neighborhood $U\subseteq \Omega_0$ with the following properties:
\begin{enumerate}[(i)]
\item
$U$ is convex and there exists an open convex set $V$ with $U\subseteq \subseteq V\subseteq \Omega_0$. 
\item
The future pointing causal geodesics which lie in $V$ are the unique (up to reparametrization) length maximizing curves.
\item
There exists a constant $C>0$ such that:
\begin{enumerate}[(1)]
\item
For all $x\in U$ we have
\begin{align}
\sup\{\psi(y)-c_2(x,y)\mid d_h(x,y)\leq C^{-1} \text{ or } d_h(x,y)\geq C\}<\varphi(x). \label{C_1}
\end{align}
\item
$d_h(U,\partial V)> C^{-1}$
\item 
For all $x\in U$, all $y\in B_{C^{-1}}(x)$ and all $j,l=0,...,n$ (denoting $\bar e_j:=e_j(x,y)\in T_yM$ and $e_j:=e_j(x,x)\in T_xM$)
\begin{align*}
|g(\bar e_j,\bar e_l)-
g(d_{x}\exp_{y}^{-1}(e_j), \bar e_l)|
\leq
\frac{1}{2(5n+1)}.
\end{align*}
Here, $B_{C^{-1}}(x)$ denotes the open ball of radius $C^{-1}$ and center $x$ w.r.t. the metric $h$.
\end{enumerate}
\item 
There exist constants $\bar C,\tilde C>0$ such that, for all $x\in U$ and $y\in V$ with $d_h(x,y)=C^{-1}$, we have
\begin{align*}
    |\exp_x^{-1}(y)|_h\leq \tilde C \text{ and } \sum_{j=0}^n \mu_j^2\geq \bar C,
\end{align*}
where
\begin{align*}
    \exp_x^{-1}(y)=\sum_{j=0}^n \mu_j e_j(x,x).
\end{align*}
\item Let
\begin{align*}
&k:dom(k)=V\times \{(x,v,w)\in T^2V\mid \exp_x(v)\in V\}\to \R,\ 
\\[10pt]
&(y,(x,v,w))\mapsto 
g_y(d_v(\exp_y^{-1}\circ \exp_x)(w),\exp_y^{-1}( x)).
\end{align*}
Then there is a small $\ep>0$ such that
\begin{align*}
\overline B_{C^{-1}}(U)\times \overline B_\ep\(\{(x,0,w)\in T^2V\mid (x,w)=\pi(x,\bar w) \text{ with } x\in \overline U \text{ and } |\bar w|_h=1\}\)
\subseteq dom(k)
\end{align*}
and the set is compact (here, $B_\ep$ denotes the $\ep$-ball in $T^2M$ w.r.t.\ $d_{T^2M}$, see subsection \ref{sasaki}).
In particular there is a modulus of continuity $\omega$ w.r.t.\ $d_h\times d_{T^2M}$ for $k$ restricted to this set.
\end{enumerate}
\end{lemma}
\begin{proof}
(i) and (ii) are clear (see also Proposition \ref{hiaudhgaiudhadadjaio}). Part (iii)(1) follows from the definition of $\Omega_0$ and Theorem \ref{Figalli}(iii). By enlarging $C$ if necessary we can also assume that (2) and (3) hold. Indeed, observe that the map
\begin{align*}
    f:V\times V\to \R,\ (x,y)\mapsto g_y(e_j(x,y),e_l(x,y))-g_y(d_x\exp_y^{-1}(e_j(x,x)),e_l(x,y)),
\end{align*}
is smooth and that $f(x,x)=0$. Thus, the claim follows from the uniform continuity of $f$ on compact subsets. Part (iv) follows immediately from the continuity of the map $V\times V\ni (x,y)\mapsto (x,\exp_x^{-1}(y))$ and from the fact that $B_{C^{-1}}(U)$ is relatively compact in $V$. For the proof of part (v) observe that the compactness (for small $\ep$) follows from the continuity of $\pi$ together with the compactness of the unit tangent bundle over $\overline U$ and from the fact that manifolds are locally compact.
\end{proof}

\begin{proof}[Proof of Theorem \ref{gafuiahfauhfaohfoafjaoii}]
Clearly, it suffices to prove the following: If $x_0\in \Omega_0$ then there exists a neighborhood $U\subseteq \Omega_0$ of $x_0$ such that the statement of the theorem holds for ${\cal L}$-a.e.\ $x\in U$. Thus, it suffices to prove the stated property for ${\cal L}$-a.e.\ $x\in U$ where $U\subseteq \Omega_0$ is as in the above lemma. For the proof, we fix an orthonormal frame $e:V\times V\to TV$ as in Lemma \ref{frame}, where $V$ is as in (i) of the above lemma. Let $\pi:TV\to TV$ be as in Definition \ref{lkjh}.

We argue by contradiction and assume that the set
\begin{align*}
B:=\{x\in U\mid \exists (y_k)_k\subseteq J^+(x), d(x,y_k)\to 0, \psi(y_k)-c_2(x,y_k)\to \varphi(x)\}
\end{align*}
is not a null set. 

\noindent \textbf{Step 1: The idea:} To obtain a contradiction we will use the $c_2$-convexity of $\varphi$. More precisely, we will prove that, for any $m\in \N$, there is a finite sequence of points $(x_i^m,y_i^m)_{1\leq i\leq m}\subseteq U\times M$ (denote by abuse of notation $(x_i,y_i)=(x_i^m,y_i^m)$) with $(x_i,y_i),(x_{i+1},y_i)\in J^+$ such that
\begin{align*}
 \psi(y_i)-c_2(x_i,y_i)\geq\varphi(x_i)- \frac{1}{m},\ \sum_{i=1}^{m-1} c_2(x_{i+1},y_i)-c_2(x_i,y_i)\xrightarrow{m\to \infty} -\infty.   
\end{align*}
On the other hand, as in \cite{Villani} (see page 74), the first inequality above and the $c_2$-convexity of $\varphi$ imply
\begin{align}
\varphi(x_1)\leq 1+\varphi(x_m)+\sum_{i=1}^{m-1} c_2(x_{i+1},y_i)-c_2(x_i,y_i) \label{dzuadszuaczasu11}
\end{align}
and since $\varphi$ is bounded on $U\subseteq \subseteq \Omega_0$ thanks to Theorem \ref{Figalli} we arrive at a contradiction.
\\

\noindent \textbf{Step 2: Construction of the sequences:} If $x\in B$ we can pick a sequence $(y_{x,k})_k$ as in the definition of $B$. From \eqref{C_1} we deduce that $y_{x,k}\neq x$ for $k$ big. For these $k\in \N$ pick a maximizing geodesic $\gamma_{x,k}:[0,1]\to M$ between $x$ and $y_k$. Set $v_{x,k}:=\frac{\dot \gamma_{x,k}(0)}{|\dot \gamma_{x,k}(0)|_h}.$

Now, since $U$ is relatively compact in $V$ by part (i) of the above lemma, the unit tangent bundle over $U$, $T^1U=\{(x,v)\in TU\mid |v|_h=1\}$, is relatively compact in $TV$ as well. Thus, $\op{Lip}(\pi_{|T^1U})<\infty$ and we can choose  $\delta>0$ with
\begin{align}
      \delta \cdot(\op{Lip}(\pi_{|T^1U})+3) \leq \ep \text{ and } \omega\big(\delta \cdot(\op{Lip}(\pi_{|T^1U})+3)\big)\leq \frac{\bar C}{4\tilde C}, \label{lkkjhgfdasrzu}
\end{align}
(where $\ep,\bar C,\tilde C$ and $\omega$ are as in the above lemma). Again by the precompactness of $T^1U$ we can cover this set with a finite number of open sets 
\begin{align*}
    U_i\subseteq TM, \ \op{diam}(U_i)\leq \delta,\ i=1,...,N.
\end{align*} 
Then it follows that
\begin{align*}
B\subseteq \bigcup_{i=1}^N D_i,\ \text{ where } D_i:=\{x\in B\mid (x,v_{x,k})\in U_i \text{ for infinitely many $k$} \}.
\end{align*}
Since $B$ is not a null set, there is some $i_0=1,....,N$ such that $D_{i_0}$ is not a null set. Consider the closed set $\overline D_{i_0}$. Since it has positive measure we can find a Lebesgue point $x_*\in \overline D_{i_0}\subseteq \overline U$ for $\overline D_{i_0}$. We will denote $D:=D_{i_0}$, so that $\overline D=\overline D_{i_0}$.

Choose $v_*\in T_{x_*}M$ with $|v_*|_h=1$ such that $(x_*,v_*)\in \overline U_{i_0}$ and define $w_*:=\pi(v_*)$. Now, since $x_*$ is a Lebesgue point, we 
make use of Fubini's theorem to find $0\neq u_*\in T_{x_*}M$ with $|u_*-w_*|_h\leq \delta$ and $\delta_*>0$ with $\delta_* |u_*|_h\leq \delta$ such that
\begin{align}
    \exp_{x_*}:B_{2\delta_*|u_*|_h}(0)\to V  \label{unasat}
\end{align}
is a diffeomorphism onto its image and such that
\begin{align*}
{\cal L}^1(r\in (0,\delta_*)\mid \exp_{x_*}(ru_*)\in \overline D\}\geq \frac{\delta_*}{2}
\end{align*}

 Now let $m\in \N$ and let us construct our $x_i$. Lemma \ref{pkspadkpa} tells us that there exist
\begin{align}
0\leq r_1\leq ...\leq r_m< \delta_* \text{ with }
r_{i+1}-r_i\geq \frac{\delta_*}{4m} \text{ and } \exp_{x_*}(r_iu_*)\in \overline D. \label{hjs}
\end{align}
Next, since the exponential map in \eqref{unasat} is a diffeomorphism, we can find, for each $1\leq i\leq m$, 
\begin{align}
u_i\in B_{\delta_*|u_*|_h}(0)\subseteq T_{x_*}M \text{ with } |u_i-(r_iu_*)|_h\leq \frac{\delta_* \delta}{8m} \text{ and } \exp_{x_*}(u_i)\in D. \label{qywdefetnhg}
\end{align}
 For each $i$ we define $x_i:=\exp_{x_*}(u_i)\in D$. 

We denote by $C_1$ the constant of part (iii) of Lemma \ref{lemma}. By definition of $D$ and by \eqref{C_1} there is $y_i\in M$ with
\begin{align}
C_1^{-1}\leq d_h(x_i,y_i)\leq C_1 \text{ such that } d(x_i,y_i)\leq \frac{1}{m} \text{ and } \psi(y_i)-c_2(x_i,y_i)\geq \varphi(x_i)-\frac{1}{m} \label{jidoadpoa}
\end{align}
and such that there exists a maximizing geodesic $\gamma_i:[0,1]\to M$ connecting $x_i$ with $y_i$ with
\begin{align}
\(x_i,\frac{\dot \gamma_i(0)}{|\dot \gamma_i(0)|_h}\)\in U_{i_0}. \label{ujnuzhbgtfc}
\end{align}
Now we have constructed our sequences $(x_i,y_i)_{1\leq i\leq m}$. Observe that the sequences depend on $m$.

\noindent \textbf{Step 3: Estimating the distances:}
We claim: \emph{If $m$ is large, then $x_{i+1}$ and $y_i$ are always causally related and it holds
\begin{align*}
d(x_{i+1},y_i)^2\geq d(x_i,y_i)^2+C_2(r_{i+1}-r_i)
\end{align*}
for some constant $C_2$ which is independent of $m$.}

We postpone the proof to the end of the whole proof since it is by far the most technical one. 
\\

\noindent \textbf{Step 4: Estimating the cost function:}
We claim: \emph{If $m$ is large, then
\begin{align*}
\sum_{i=1}^{m-1} &c_2(x_{i+1},y_i)-c_2(x_i,y_i)
\leq
C_3- \sum_{i=1}^{m-1} \frac{C_1^{-1}}{2} (d(x_{i+1},y_i)-d(x_i,y_i)), 
\end{align*}
where $C_3$ is independent of $m$.}
\\

\noindent \textbf{Proof of claim:}
First let $m$ be as large as needed for step 3. We observe that
\begin{align*}
    c_2(x_{i+1},y_i)&=((\tau(y_i)-\tau(x_i)-d(x_i,y_i))+(\tau(x_i)-\tau(x_{i+1}))-(d(x_{i+1},y_i)-d(x_i,y_i)))^2
    \\[10pt]
    &=:(a_i+b_i-c_i)^2.
\end{align*}
Thus,
\begin{align}
    c_2(x_{i+1},y_i)-c_2(x_i,y_i)=(a_i+b_i-c_i)^2-a_i^2=c_i(-2a_i+c_i)+b_i(2a_i+b_i-2c_i). \label{diauh1uzdgbauio}
\end{align}
Using the fact that $x_{i+1}\leq y_i$ it follows from \eqref{splitting} that
\begin{align*}
    c_i\leq \tau(y_i)-\tau(x_{i+1})-d(x_i,y_i)\leq (\tau(y_i)-\tau(x_i)-d(x_i,y_i))+(\tau(x_i)-\tau(x_{i+1}))=a_i+b_i.
\end{align*}
Inserting this inequality into the first term on the right side of \eqref{diauh1uzdgbauio} and using that, thanks to step 3, $c_i\geq 0$, this gives
\begin{align}
    c_2(x_{i+1},y_i)-c_2(x_i,y_i)\leq
    -a_ic_i +b_i(2a_i+b_i-c_i). \label{mndgsiospo2}
\end{align}
Using the fact that $d_h(x_i,y_i)\geq C_1^{-1}$, \eqref{splitting} gives $a_i\geq \frac{C_1^{-1}}{2}$ and
\begin{align}
    -a_ic_i\leq -\frac{C_1^{-1}}{2} (d(x_{i+1},y_i)-d(x_i,y_i)). \label{mndgsiospo}
\end{align}
Now we need to estimate the second part in \eqref{mndgsiospo2}. By the uniform continuity of $\tau$ (resp. $d$) on the relatively compact set $B_{C_1}(U)$ (resp. $B_{C_1}(U)\times B_{C_1}(U)$) we deduce that there exists a constant $C_{3,1}>0$ (which does not depend on $m$) such that
\begin{align}
    a_i,|b_i|,c_i\leq C_{3,1}. \label{dadgadaddaa}
\end{align}

Moreover, using the fact that the map $\tau\circ \exp_{x_*}:B_{\delta_*|u_*|_h}(0)\to \R$ is Lipschitz thanks to \eqref{unasat}, we get from the definition of $x_i$ and $u_i$ together with \eqref{hjs} and \eqref{qywdefetnhg} that
\begin{align}
    |b_i|\leq C_{3,2}(r_{i+1}-r_i), \label{mndgsiospo1}
\end{align}
where $C_{3,2}$ only depends on the Lipschitz constant, on $|u_*|_h$ and on $\delta$.
We now insert \eqref{mndgsiospo}, \eqref{dadgadaddaa} and \eqref{mndgsiospo1} into \eqref{mndgsiospo2} and we sum over $i=1,...,m-1$ to obtain 
\begin{align*}
    \sum_{i=1}^{m-1} c_2(x_{i+1},y_i)-c_2(x_i,y_i)
    \leq 
    C_3  -\sum_{i=1}^{m-1} \frac{C_1^{-1}}{2} (d(x_{i+1},y_i)-d(x_i,y_i))
\end{align*}
for $C_3:=4C_{3,1}C_{3,2}\delta_*$. In particular, $C_3$ is independent of $m$. 
\hfill \checkmark
\\

\noindent \textbf{Step 5: Conclusion:} 
We assume that $m$ is as large as needed in step 3 and 4.
 Since $r_{i+1}-r_i\geq \frac{\delta_*}{4m}$ it follows from step 3 that
\begin{align*}
d(x_{i+1}, y_i)^2
\geq
d(x_i,y_i)^2 + \frac{C_2 \delta_*}{4m}.
\end{align*}
Then we use \eqref{dzuadszuaczasu11}, step 4 and the above inequality to obtain
\begin{align*}
\varphi(x_1)&\leq 1+\varphi(x_m)+\sum_{i=1}^{m-1} c_2(x_{i+1},y_i)-c_2(x_i,y_i)
\\[10pt]
&\leq
1+\varphi(x_m)+C_3 -\frac{C_1^{-1}}{2} \sum_{i=1}^{m-1} (d(x_{i+1},y_i)-d(x_i, y_i))
\\[10pt]
&\leq
1+\varphi(x_m)+C_3-\frac{C_1^{-1}}{2} \sum_{i=1}^{m-1}\sqrt{d(x_i,y_i)^2+\frac{C_2 \delta_*}{4m}}-\sqrt{d(x_i,y_i)^2}.
\end{align*}
Since $d(x_i,y_i)\leq \frac{1}{m}$ by \eqref{jidoadpoa} one can easily check that
\begin{align*}
\varphi(x_1)
&\leq 1+\varphi(x_m)+ C_3-\frac{C_1^{-1}}{2} (m-1)\(\sqrt{\frac{1}{m^2}+\frac{C_2 \delta_*}{4m}}-\sqrt{\frac1{m^2}}\)
\\[10pt]
&\leq  1+\varphi(x_m)+ C_3-\frac{C_1^{-1}}4 \(\sqrt{1+\frac{C_2 \delta_*m}{4}}-1\).
\end{align*}
Since $\varphi$ is bounded on $U\subseteq\subseteq \Omega$ and $x_1,x_m\in U$, we see that the left hand side is bounded and the right hand side converges to $-\infty$ as $m\to \infty$. This gives the contradiction and, hence, proves the theorem. It remains to prove the claim of step 3 above.
\bigskip

\noindent \textbf{Proof of claim of step 3:}
Pick for any $i=1,...,m$ a maximizing geodesic $\gamma_i:[0,1]\to M$ such that \eqref{ujnuzhbgtfc} holds. As $d_h(x_i,y_i)\geq C_1^{-1}$, we can choose the first $t_i\in [0,1]$ with $d_h(x_i,\gamma_i(t_i))=C_1^{-1}$. We set $\bar y_i:=\gamma_i(t_i)\in V$.

Using the triangle inequality and the fact that $\gamma_i$ is maximizing, we obtain
\begin{align*}
d(x_i,y_i)=d(x_i,\bar y_i)+d(\bar y_i,y_i) \text{ and } d(x_{i+1},y_i)\geq d(x_{i+1},\bar y_i)+d(\bar y_i,y_i).
\end{align*}
Thus, if we can show that
\begin{align}
    d(x_{i+1},\bar y_i)^2\geq d(x_i,\bar y_i)^2+C_2(r_{i+1}-r_i) \label{frsbvajqihsj}
\end{align}
then we automatically have $d(x_{i+1},\bar y_i)\geq d(x_i,\bar y_i)$ and it follows
\begin{align*}
    d(x_{i+1},y_i)^2\geq d(x_{i+1},\bar y_i)^2+2d(x_{i+1},\bar y_i)d(\bar y_i,y_i)+d(\bar y_i,y_i)^2\geq d(x_i,\bar y_i)^2+C_2(r_{i+1}-r_i).
\end{align*}
In particular, $x_{i+1}$ and $y_i$ are causally related and the claim is proven. Thus, it remains to prove \eqref{frsbvajqihsj} for some constant $C_2$ that does not depend on $m$.
\\

Fix an arbitrary $i=1,...,m-1$.
First, observe that, as $x_i,\bar y_i,x_{i+1}\in V$ and $V$ is a convex set, all the expressions $\exp_{\bar y_i}^{-1}(x_{i+1})$ etc. are well-defined.

From Lemma \ref{lemma}(ii) we infer that if $\exp_{\bar y_i}^{-1}(x_{i+1})\in -\C_{\bar y_i}$ then $x_{i+1}$ and $\bar y_i$ are causally related and the Lorentzian distance between these points equals the Minkowski norm of the vector $\exp_{\bar y_i}^{-1}(x_{i+1})$.
Thus, the object we need to study is
\begin{align*}
f(x_{i+1},\bar y_i) \text{ for }
f:V\times V\to \R,\ f(x,y):=-g(\exp_y^{-1}(x),\exp_y^{-1}(x)).
\end{align*} 
Clearly, $f$ is a smooth map thanks to the convexity of $V$.
Using first order Taylor-expansion of the map $h:=f(\cdot,\bar y_i)\circ \exp_{x_*}$ at the point $u_i\in T_{x_*}M$ we obtain for some $u=(1-t)u_i+tu_{i+1}$ $(t\in (0,1))$\footnote{Observe that $((1-s)u_i+su_{i+1})\in \op{dom}(h)\  \forall s\in [0,1]$ thanks to \eqref{unasat} and \eqref{qywdefetnhg}, so that the mean value theorem used in the first line is applicable. Moreover, for the same reason, $(\bar y_i,(x_*,u,u_*+\ep_i))\in dom(k)$ in the last line.}
\begin{align}
&f(x_{i+1},\bar y_i) =h(u_{i+1})=h(u_i)+d_uh(u_{i+1}-u_i)\nonumber
\\[10pt]
&= \nonumber
f(x_i,\bar y_i)-2g\(d_{u}\(\exp_{\bar y_i}^{-1}\circ \exp_{x_*}\)\((r_{i+1}-r_i)u_*+R_{i+1}-R_i\),\exp_{\bar y_i}^{-1}(x_i)\) \nonumber
\\[10pt]
&=f(x_i,\bar y_i)-2(r_{i+1}-r_i) g\(d_{u}\(\exp_{\bar y_i}^{-1}\circ \exp_{x_*}\)\(u_*+\ep_i\),\exp_{\bar y_i}^{-1}(x_i)\),\nonumber
\\[10pt]
&=f(x_i,\bar y_i)-2(r_{i+1}-r_i)k(\bar y_i,(x_*,u,u_*+\ep_i)).\label{ikmjuzhgt}
\end{align}
Here, we have set $u_j=:r_ju_*+R_j$, $j=1,...,m$, $\ep_i:=\frac{R_{i+1}-R_i}{r_{i+1}-r_i}$ and $k$ denotes the map from Lemma \ref{lemma}(v).

To estimate the latter term recall that $\gamma_i:[0,1]\to M$ is a maximizing geodesic between $x_i$ and $y_i$. Let $\bar \gamma_i:[0,1]\to V$ be the geodesic reparametrization of the first part of $\gamma_i$ such that $\bar \gamma_i(1)=\bar y_i$. Then $\bar \gamma_i$ is the unique maximizing geodesic (up to reparametrization) which connects $x_i$ with $\bar y_i$. Denote $e_j:=e_j(x_i,x_i)$ so that
\begin{align*}
\bar v_i:=\dot {\bar \gamma}_i(0)=\exp_{x_i}^{-1}(\bar y_i)=\sum_{j=0}^n \mu_j e_j \text{ for some }\mu_j\in \R \text{ and }\ \bar w_i:=\pi(\bar v_i)=\sum_{j=1}^n \mu_j e_j.
\end{align*}
Writing $\bar e_j:=e_j(x_i,\bar y_i)$, we know from Corollary \ref{cor} that
\begin{align}
\exp_{\bar y_i}^{-1}(x_i)=\sum_{j=0}^n -\mu_j \bar e_j.\label{hiaudhaoidhadhsaiudhsiofosi}
\end{align}
We observe that
\begin{align}
k(\bar y_i,(x_i,0,|\bar v_i|_h^{-1}\bar w_i))=|\bar v_i|_h^{-1}g(d_{x_i}\exp_{\bar y_i}^{-1}(\bar w_i),\exp_{\bar y_i}^{-1}(x_i)) \label{oiuztr}
\end{align}
and, denoting by $\bar d$ the distance on $M\times T^2M$ induced by $d_h$ and $d_{T^2M}$, we also see
\begin{align*}
&\bar d\(\big(\bar y_i,(x_i,0,|\bar v_i|_h^{-1}\bar w_i)\big),\big(\bar y_i,(x_*,u,u_*+\ep_i)\big)\)
\\[8pt]
&\leq
d_{T^2M}\(\(x_i,0,|\bar v_i|_h^{-1}\bar w_i\),\(x_*,u,u_*+\ep_i\)\)
\\[8pt]
&\leq |u|_h +d_{TM}\(\(x_i,|\bar v_i|_h^{-1}\bar w_i\),\(x_*,u_*+\ep_i\)\) 
\\[10pt]
&\leq |u|_h+ d_{TM}\(\(x_i,\pi(|\dot {\bar \gamma}_i(0)|_h^{-1}\dot {\bar \gamma}_i(0))\),(x_*,w_*)\)+|\ep_i|_h+|u_*-w_*|_h 
\\[10pt]
&=|u|_h+|\ep_i|_h+d_{TM}\(\pi(x_i,|\dot \gamma_i(0)|_h^{-1}\dot \gamma_i(0)),\pi((x_*,v_*))\)+|u_*-w_*|_h. 
\end{align*}
From \eqref{qywdefetnhg} we get that $|u|_h\leq (1-t)|u_i|_h+t|u_{i+1}|_h<\delta_*|u_*|_h\leq \delta$ and from \eqref{hjs} and \eqref{qywdefetnhg} that $|\ep_i|_h\leq \delta$. Also, by definition of $u_*$, we have $|u_*-w_*|_h\leq \delta$. Moreover, using that $(x_i,|\dot \gamma_i(0)|_h^{-1}\dot \gamma_i(0)),(x_*,v_*)\in \overline U_{i_0}\subseteq \overline U$ by \eqref{ujnuzhbgtfc} and $\op{diam}(\overline U_{i_0})\leq  \delta$ we get that the third term is less or equal $\op{Lip}(\pi_{|T^1U})\delta$.
Thus,
\begin{align}
\bar d\(\big(\bar y_i,(x_i,0,|\bar v_i|_h^{-1}\bar w_i)\big),\big(\bar y_i,(x_*,u,u_*+\ep_i)\big)\) \label{oiuztr1}
\leq (\op{Lip}(\pi_{|T^1U})+3)\delta=:C_{2,1}\delta.
\end{align}
Since by definition a modulus of continuity is non-decreasing, we can use \eqref{oiuztr} and \eqref{oiuztr1} and $\omega$, the modulus of continuity introduced in Lemma \ref{lemma}\footnote{Note that the lemma is applicable since the computation shows $d_{T^2M}((x_i,0,|\bar v_i|_h^{-1}\bar w_i),(x_*,u,u_*+\ep_i))\leq C_{2,1}\delta \leq \ep$ thanks to the definition of $\delta$ in \eqref{lkkjhgfdasrzu}.}, to estimate \eqref{ikmjuzhgt} and to get
\begin{align*}
f(x_{i+1},\bar y_i)\geq \ &f(x_i,\bar y_i)-2(r_{i+1}-r_i)|\bar v_i|_h^{-1}
g(d_{x_i}\exp_{\bar y_i}^{-1}(\bar w_i),\exp_{\bar y_i}^{-1}(x_i))-2(r_{i+1}-r_i)\omega(C_{2,1}\delta)
\\[10pt]
\overset{\eqref{hiaudhaoidhadhsaiudhsiofosi}}{=}&f(x_i,\bar y_i)-2(r_{i+1}-r_i)|\bar v_i|_h^{-1}\sum_{\substack{j=1,\\l=0}}^n \mu_j (-\mu_l) g(d_{x_i}\exp_{\bar y_i}^{-1}(e_j),\bar e_l)-2(r_{i+1}-r_i)\omega(C_{2,1}\delta).
\end{align*}
Next we use (iii) of Lemma \ref{lemma} to deduce that 
\begin{align*}
    |g(d_{x_i}\exp_{\bar y_i}^{-1}(e_j),\bar e_l)-g(\bar e_j,\bar e_l)|\leq \frac1{2(5n+1)}.
\end{align*} 
Thus,
\begin{align*}
f(x_{i+1},\bar y_i)&\geq f(x_i,\bar y_i)
-2(r_{i+1}-r_i)|\bar v_i|_h^{-1} \(\sum_{\substack{j=1,\\l=0}}^n \mu_j (-\mu_l) g(\bar e_j,\bar e_l)\)
\\[10pt]
&-\frac{(r_{i+1}-r_i)|\bar v_i|_h^{-1}}{5n+1}\(\sum_{\substack{j=1,\\l=0}}^n |\mu_j||\mu_l|\)-2(r_{i+1}-r_i) \omega(C_{2,1}\delta),
\end{align*}
and using that $(\bar e_j)$ is an orthonormal basis in the tangent space $T_{\bar y_i}M$ we obtain
\begin{align*}
f(x_{i+1},\bar y_i)&\geq f(x_i,\bar y_i)
+2(r_{i+1}-r_i) |\bar v_i|_h^{-1}\(\sum_{j=1}^n \mu_j^2 \)
\\[10pt]
&-\frac{(r_{i+1}-r_i)|\bar v_i|_h^{-1}}{5n+1} \(\sum_{\substack{j=1,\\l=0}}^n |\mu_j||\mu_l|\)-2(r_{i+1}-r_i) \omega(C_{2,1}\delta).
\end{align*}
Now we denote by $C_{2,2}$ and $C_{2,3}$ the constants $\bar C$ and $1/\tilde C$ from Lemma \ref{lemma}(iv) which then gives 
\begin{align*}
    \mu_0^2-\sum_{j=1}^n \mu_j^2=|\exp_{x_i}^{-1}(\bar y_i)|_g^2=d(x_i,\bar y_i)^2\leq 1/m^2\leq \frac{C_{2,2}}{2}\leq \frac{1}{2}\sum_{j=0}^n \mu_j^2
\end{align*}
for $m$ big. This yields
\begin{align*}
    \sum_{\substack{j=1,\\l=0}}^n |\mu_j||\mu_l|
    \leq
    n\mu_0^2+(2n+1)\sum_{j=1}^n \mu_j^2 
    \leq
    (5n+1)\sum_{j=1}^n \mu_j^2.
\end{align*}
We use \ref{lemma}(iv) and the fact that $\bar v_i=\exp_{x_i}^{-1}(\bar y_i)$ to deduce
\begin{align*}
f(x_{i+1},\bar y_i)&\geq f(x_i,\bar y_i)+(r_{i+1}-r_i)|\bar v_i|_h^{-1}\( \sum_{j=1}^n \mu_j^2\) -2(r_{i+1}-r_i) \omega(C_{2,1}\delta)
\\[10pt]
&\geq f(x_i,\bar y_i)+(r_{i+1}-r_i) C_{2,2}C_{2,3}-2(r_{i+1}-r_i) \omega(C_{2,1}\delta)
\end{align*}
    By definition of $\delta$ we have $\omega(C_{2,1}\delta)\leq \frac{C_{2,2}C_{2,3}}{4}$. Then, with $C_2:=\frac{C_{2,2}C_{2,3}}2$ we have
\begin{align*}
f(x_{i+1},\bar y_i)\geq f(x_i,\bar y_i)
+C_2(r_{i+1}-r_i)> 0.
\end{align*}
Since $V$ is a convex set it follows from this that $x_{i+1}< \bar y_i$ or that $x_{i+1}> \bar y_i$. This almost proves the claim. However, we still need to argue why $x_{i+1}<\bar y_i$ and not $\bar y_i< x_{i+1}$. 

One can repeat the exact same argument for $x_{i+1}(t):=\exp_{x_*}((1-t)u_i+tu_{i+1})\in V$, $t\in (0,1]$, and prove that
\begin{align*}
-g_{\bar y_i}(\exp_{\bar y_i}^{-1}(x_{i+1}(t)),\exp_{\bar y_i}^{-1}(x_{i+1}(t)))> 0
\end{align*}
for all $t$. By reasons of continuity this shows that either $x_{i+1}(t)<\bar y_i$ for all $t\in (0,1]$ or that $\bar y_i< x_{i+1}(t)$ for all $t\in (0,1]$. Since $x_{i+1}(t)\to x_i< \bar y_i$ as $t\to 0$ we deduce that $x_{i+1}(t)>\bar y_i$ cannot be possible for small $t$. Thus we have proved that $x_{i+1}$ and $\bar y_i$ are causally related and that
\begin{align*}
d(x_{i+1},\bar y_i)^2\geq d(x_i,\bar y_i)^2+C_2(r_{i+1}-r_i).
\end{align*}
This finally proves the claim.
\hfill \checkmark
\end{proof}

{\center {\large{\textbf{Proof of Theorem \ref{hudaodiaoiii}}}}}
\bigskip

From Theorem \ref{gafuiahfauhfaohfoafjaoii} we deduce the following corollary.

\begin{corollary}\label{mainnn}
For $\mu$-a.e. $x_0\in \Omega_0$ there exists a neighborhood $U\subseteq \Omega_0$ of $x_0$ and some $\delta>0$ such that, for all $x\in U\cap I^+(x_0)$ it holds
\begin{align*}
\sup\{\psi(y)-c_2(x,y)\mid d(x,y)\leq \delta\}<\varphi(x).
\end{align*}
\end{corollary}
\begin{proof}
We consider the set of all $x_0\in \Omega_0$ for which Theorem \ref{gafuiahfauhfaohfoafjaoii} holds and such that $\partial_{c_2}\varphi(x_0)\cap I^+(x_0)\neq \emptyset$. Using that $\mu$ is absolutely continuous w.r.t. ${\cal L}$ we see that this set if of full $\mu$-measure by Theorem \ref{gafuiahfauhfaohfoafjaoi}, Theorem \ref{gafuiahfauhfaohfoafjaoii} and by Lemma \ref{hiuajdoaildkao}.

Let $x_0\in\Omega_0$ be in this set. 
Then we find $\delta>0$ and $r>0$ such that
\begin{align*}
\sup\{\psi(y)-c_2(x_0,y)\mid d(x,y)\leq \delta\}\leq\varphi(x_0)-r.
\end{align*}
Also let $y_0\in \partial_{c_2}\varphi(x_0)\cap I^+(x_0)$.

Let us assume that there is a sequence $(x_k)\subseteq I^+(x)$ with $x_k\to x_0$ and a sequence $(y_k)\subseteq J^+(x_k)$ such that $d(x_k,y_k)\leq \frac{1}{k}$ and $\psi(y_k)-c_2(x_k,y_k)\geq \varphi(x_k)-\frac{1}{k}$.
The sequence $(y_k)$ must be precompact by Theorem \ref{Figalli}. Then $d(x_0,y_k)\to 0$ by the uniform continuity of $d$ on compact sets.
Thus, for big $k$ we have $d(x_0,y_k)\leq \delta$ and, hence, 
\begin{align*}
\psi(y_k)-c_2(x_0,y_k)\leq \varphi(x_0)-r= \psi(y_0)-c_2(x_0,y_0)-r.
\end{align*}
Thus, if $x_k$ is close to $x_0$, it follows from the uniform continuity of $c_2$ on compact subsets of $J^+$ (observe that $y_k\in J^+(x_k)\subseteq J^+(x_0)$ and that $y_0\in J^+(x_k)$ for big $k$) that
\begin{align*}
\psi(y_k)-c_2(x_k,y_k)\leq \psi(y_0)-c_2(x_k,y_0)-\frac{r}{2}\leq \varphi(x_k)-\frac{r}{2}.
\end{align*}
This is a contradiction.
\end{proof}

\begin{definition}\rm
Denote by $A$ the set of all $x_0\in  \Omega_0$ for which the above corollary holds true with $U=U_{x_0}\subseteq \Omega_0$. We define the open set
\begin{align*}
\Omega_{1}:=\bigcup_{x_0\in A} U_{x_0}\cap I^+(x_0)\subseteq \Omega_0.
\end{align*}
\end{definition}

\begin{lemma} \label{hudoaidapjsoapd}
We have $\mu(\Omega_1)=1$.
\end{lemma}
\begin{proof}
The proof is easy and completely analogous to the proof of Corollary \ref{iuadhoiapaopao}.
\end{proof}

\begin{proof}[Proof of Theorem \ref{hudaodiaoiii}]
Let $x_0\in \Omega_1$. 
 By definition of $\Omega_1$ and by Corollary \ref{mainnn} and Theorem \ref{Figalli} we know that there exists an open neighborhood $U\subseteq \Omega_1$ of $x_0$ and constants $C,\delta>0$ such that, for any $x\in U$,
\begin{align*}
\varphi(x)=\sup\{\psi(y)-c_2(x,y)\mid d_h(x,y)\leq C,\ d(x,y)\geq \delta\}.
\end{align*}

Consider the open sets
\begin{align*}
&V_1:=\{y\in J^+(x_0)\mid d_h(x_0,y)< 2C,\ d(x_0,y)>\frac{\delta}{2}\} \text{ and } 
\\[10pt]
&V_2:=\{y\in J^+(x_0)\mid  d_h(x_0,y)< 4C,\ d(x_0,y)>\frac{\delta}{4}\}.
\end{align*}
By the continuity of $d$ and the completeness of $h$ we find an open neighborhood $W\subseteq U$ of $x_0$ such that $W\times V_2\subseteq I^+$. Using the continuity of $d$ and $d_h$ again, we can assume, by shrinking $W$ if necessary, that for all $x\in W$:
\begin{align*}
\varphi(x)=\sup\{\psi(y)-c(x,y)\mid y\in V_1\}.
\end{align*}
From Proposition \ref{injiuhgb} we know that $c_2$ is locally semiconcave on $W\times V_2$. Thus, the family of functions $(c(\cdot,y))_{y\in V_2}$ is locally uniformly locally semiconcave (\cite{Fathi/Figalli}, A15). Since $\overline V_1\subseteq V_2$ is compact it follows that $\varphi=\sup\{\psi(y)-c(\cdot,y)\mid y\in V_1,\ \psi(y)\in \R\}$ is locally semiconvex on $W$ as the finite supremum of a uniformly locally semiconvex family of functions (\cite{Fathi/Figalli}, A16).
\end{proof}

\begin{proof}[Proof of Corollary \ref{cd}]
With all the tools we have at hand by now, the proof follows in a standard manner.
We define the set
\begin{align*}
A:=\{x\in \Omega_1\mid  \partial_{c_2}\varphi(x)\neq \emptyset,\ \varphi \text{ is differentiable at } x\}.
\end{align*}
Then $A$ is of full $\mu$-measure by Lemma \ref{hiuajdoaildkao}, by Lemma \ref{hudoaidapjsoapd} and by the fact that locally semiconvex functions are differentiable ${\cal L}$-a.e.\ (see Theorem 10.8 in \cite{Villani}) and $\mu$ is absolutely continuous w.r.t.\ the Lebesgue measure.

Now, consider some $x\in A$. Let $y$ with $(x,y)\in \partial_{c_2}\varphi\cap J^+$. From Corollary \ref{mainnn} we deduce that $y\in I^+(x)$. Then, as $\varphi$ is differentiable at $x$ and $c_2(\cdot,y)$ is locally semiconcave in a neighborhood of $x$ it follows from $y\in \partial_{c_2}\varphi(x)$ that $c_2(\cdot,y)$ is differentiable at $x$ and that
\begin{align*}
-d_x\varphi=\frac{\partial c_2}{\partial x}c_2(x,y)
\end{align*}
(see \cite{Fathi/Figalli}, page 13).
Since $c_2$ satisfies the twist condition on $I^+$ (see Corollary \ref{twist}) it follows that $y$ is uniquely determined and given by
\begin{align}
y=\(\frac{\partial c_2}{\partial x}(x,\cdot)\)^{-1}(-d_x\varphi). \label{adjuha9ogf7sa}
\end{align}
Hence, for $\mu$-a.e. $x$ there exists only one $y\in M$ with $(x,y)\in\partial_{c_2}\varphi\cap J^+$. But since $\pi$ is concentrated on $\partial_{c_2}\varphi\cap J^+$ this means that $\pi$ is induced by a Borel map $T$ (see Lemma 2.20 in \cite{Ambrosio/Gigli}) and that $T$ is $\mu$-a.e. given by \eqref{adjuha9ogf7sa}. This proves the first part of the theorem.

Now let us assume that there exists another optimal coupling $\pi_1$ which is different from $\pi$. If this coupling does not admit a $\pi_1$-solution we are done. If it does, we have $\pi_1=(Id\times T_1)_\#\mu$ for a Borel map $T_1:M\to M$ by the preceding results. But $\frac{1}{2}(\pi+\pi_1)$ is, clearly, also an optimal coupling. If this coupling would also admit a $\frac{1}{2}(\pi+\pi_1)$-solution then it also has to be induced by a transport map. But this is only possible if $T=T_1$ $\mu$-a.e., hence if $\pi=\pi_1$. This is a contradiction and this concludes the proof of the corollary.
\end{proof}

\section*{Acknowledgements}
I would like to thank Stefan Suhr for proposing this interesting topic to me and for the helpful discussions on this subject. I would also like to thank Markus Kunze for his support and advices he gave me while writing this paper.

\section*{Statements and Declarations}

 \noindent \textbf{Conflict of interest:} None

\noindent \textbf{Data availability:} Apart from the references, this article does not use any external data.

\section{\Large{Appendix}}

\subsection{Lorentzian length functional and convex neighborhoods}\label{length}

In this subsection we will prove a lemma about convex neighborhoods which we needed in this paper.
By $(M,g)$ we will always denote a globally hyperbolic spacetime and $n+1:=\dim(M)$.

\begin{definition}\rm \label{convex}
We call an open set $U\subseteq M$ \emph{convex} if there exists an open set $\Omega\subseteq TM$ such that, for all $x\in U$ the set $\Omega_x:=\{v\in T_xM\mid (x,v)\in \Omega\}$ is star-shaped and $\exp_x:\Omega_x\to U$ is a diffeomorphism.
\end{definition}

\begin{remark}\rm
In a convex set there exists between each two points a unique geodesic (up to reparametrization) which lies in the convex set. It is well-known that every point admits an arbitrarily small convex neighborhood.
\end{remark}

\begin{proposition}\label{hiaudhgaiudhadadjaio}
Let $x_0\in M$. Then there exists a convex neighborhood $U$ of $x_0$ such that
all future pointing causal geodesics that lie in $U$ are length maximizing. 

Moreover, for any two points $x,y\in U$ it holds:
\begin{align*}
    g_x(\exp_x^{-1}(y),\exp_x^{-1}(y))< 0 \Rightarrow x< y \text{ or }x> y.
\end{align*}
\end{proposition}
\begin{proof}
    The second part is easily verified, just take the geodesic $\exp_x(t\exp_x^{-1}(y))$ in $U$ and observe that it is either future directed or past directed. For the first part, let $V$ be any convex neighborhood of $x_0$. By the global hyperbolicity of $M$ we can find a smaller convex neighborhood $U$ of $x_0$ such that any future pointing causal curve that starts and ends in $U$ lies entirely in $V$. Then, any future pointing causal geodesic that starts and ends in $U$ has to lie entirely in $V$ and thus, by  \cite{ONeill}, Proposition 34 in Chapter 5, has to be maximal.
\end{proof}

\begin{definition}\rm \label{ONF}
A \emph{smooth orthonormal frame} on an open set $U\subseteq M$ is a family of smooth vector fields $e_i\in \Gamma(TU)$ ($i=0,...,n$) such that
\begin{align*}
g(e_0,e_0)=-1,\ g(e_i,e_i)=1,\, i\geq 1,\, g(e_i,e_j)=0,\, i\neq j.
\end{align*}
\end{definition}

\begin{remark}\rm
\begin{enumerate}[(i)]
\item
Recall that, given a smooth curve $c:I\to M$ and a vector field $V:I\to TM$ along $c$, the covariant derivative of $V$ is given by
\begin{align}
\frac{\nabla V}{dt}(t)=\sum_{k=0}^n\(\dot v_k(t)+\sum_{i,j=0}^nv_i(t)\dot c_j(t) \Gamma_{ij}^k(c(t))\) \partial_k|_{c(t)}. \label{parallel}
\end{align}
Here, $v_i$ and $c_j$ are the coordinates of $V$ and $c$, $\Gamma_{ij}^k$ are the Christoffel symbols and $\partial_k$ is the $k$-th coordinate basis vector.
Thus, setting $\frac{\nabla V}{dt}(t)=0$, this is a linear ordinary differential equation which is uniquely determined by $V(0)\in T_{c(0)}M$.
\item Let $U$ be a convex set and $x_0\in U$ be fixed. Let $e_0,...,e_n$ be an orthonormal basis of $T_{x_0}M$. If $x\in U$ denote by $e_0(x),...,e_n(x)$ the orthonormal basis of $T_xM$ obtained by parallel transporting $e_0,...,e_n$ along the unique (up to reparametrization)  geodesic between $x_0$ and $x$ in $U$ (observe that this definition is independent of the parametrization of the geodesic, compare \eqref{parallel}).
Now, for $y\in U$ choose the unique (up to reparametrization) geodesic between $x$ and $y$ and denote the parallel transport of $e_i(x)$ by $e_i(x,y)\in T_yM$.
\end{enumerate}
\end{remark}

\begin{lemma} \label{dghaidghadhaodfjuapoda}
$e:=(e_0,...,e_n)$ is a smooth orthonormal frame.
\end{lemma}
\begin{proof}
Of course we only have to prove the smoothness, since parallel transport is an isometry as a map between the tangent spaces. Let $i=0,...,n$.
We first claim that the map
\begin{align*}
f:U\to TU,\ f(x):=(x,e_i(x)),
\end{align*}
is smooth. Indeed, let $c_x(t):=\exp_{x_0}(t\exp_{x_0}^{-1}(x))$ be the unique geodesic in $U$ between $x_0$ and $x$. Observe that $c_x(t)$ depends smoothly on $x$ and $t$.
In local coordinates we need to solve the differential equation
\begin{align*}
\dot v_k(t)+\sum_{i,j=0}^nv_i(t)\dot c_{x,j}(t) \Gamma_{ij}^k(c_x(t))=0,\ k=0,...,n.
\end{align*}
This is an ordinary differential equation which depends smoothly on its parameter $x$. Hence, also the solution depends smoothly on $x$ and on its initial value.
Hence, $f$ is smooth. Now we need to show that also $e_i:U\times U\to TU$ is smooth.

For, we define the map $g:TU\times U\to TU$ which parallel transports a vector $v\in T_{x}U$ along the unique geodesic fom $x$ to $y$ such that $g((x,v),y)=(y,w)\in TU$. 
Denote by $c_{x,y}:=\exp_{x}(t\exp_{x}^{-1}(y))$ the unique geodesic between $x$ and $y$. 
Then we need to solve the differential equation
\begin{align*}
\dot v_k(t)+\sum_{i,j=0}^nv_i(t)\dot c_{x,y,j}(t) \Gamma_{ij}^k(c_{x,y}(t))=0,\ k=0,...,n.
\end{align*}
This is a differential equation with parameters $x$ and $y$ and the equation depends smoothly on those. Hence, the solution also depends smoothly on $x$ and $y$ and on the initial value. This proves that $g$ is also smooth.

Finally observe that
\begin{align*}
e_i:U\times U\to TU,\ e_i(x,y)=g((x,e_i(x)),y)=g(f(x),y)
\end{align*}
is smooth as the composition of two smooth functions.
\end{proof}

\begin{example}\rm
In the simple case where $M=\R^{1+n}$ is the Minkowski space, the orthonormal basis $(e_0(x,y),...,e_n(x,y))$ is always equal to the standard basis in $\R^{1+n}$.
\end{example}

\subsection{Proof of Proposition \ref{flow}}
 \begin{proof}
 Let $\exp$ denote the exponential function w.r.t.\ the Lorentzian metric $g$. We define the set ${\cal D}:=\{(x,v,t)\in \overline \C\times \R\mid \exp_x(tv) \text{ is defined}\}$, which is (relatively) open in $\overline \C\times \R$. Next we consider the map
 \begin{align*}
     \psi:{\cal D}\to \R,\ (x,v,t)\mapsto \psi_{x,v}(t):=
     \begin{cases}
         \frac{1}{L_1(v)}\int_0^t L_1(\frac{d}{ds}\exp_x(sv))\, ds,\ &\text{ if } v\neq 0,
         \\\\
         t,\ &\text{ if } v=0.
     \end{cases}
 \end{align*}
\textbf{Claim 1: $\psi$ is continuous:}
Clearly, $\psi$ is continuous at any point $(x_0,v_0,t_0)\in {\cal D}$ with $v_0\neq 0$. To prove continuity at some point $(x_0,0,t_0)$, consider some $r>0$ such that $\exp_{x_0}(t_0v)$ is defined for all $|v|_h\leq r$. 
We find a (relatively) open neighborhood $U\subseteq \overline \C$ of the compact set
    $\{(x_0,v)\in \overline \C\mid |v|_h=r\}$ such that $L_1$ is uniformly continuous on $U$ and $\inf_U |L_1|=:a>0$. Now let $(x_k,v_k,t_k)\in {\cal D}$ with $v_k\neq 0$ and $(x_k,v_k,t_k)\to (x_0,0,t_0)$. Write $v_k=:\frac{|v_k|_h}{r} w_k=:\ep_k w_k$. Then by the positive $1$-homogeneity of $L_1$ as a function from the tangent space we have
    \begin{align}
        \psi_{x_k,v_k}(t_k)= \int_0^{t_k} \frac{ L_1(\frac{d}{ds}\exp_{x_k}(s v_k))}{L_1(v_k)}\, ds=
         \int_0^{t_k} \frac{ L_1(\frac{d}{d\tau}|_{\tau=s\ep_k}\exp_{x_k}(\tau w_k))}{L_1(w_k)}\, ds \label{ujsndftaz1}
    \end{align}
   Using the fact that $L_1$ is uniformly continuous on $U$ one easily checks that, by the smoothness of the exponential map, we have
    \begin{align}
        \sup_{s\in[0,t_k]} \bigg|L_1\(\frac{d}{d\tau}|_{\tau=s\ep_k}\exp_{x_k}(\tau w_k)\)-L_1\(w_k\)\bigg| \xrightarrow{k\to \infty}  0. \label{ujsndftaz}
    \end{align}
    Moreover, using the fact that $|L_1(w_k)|\geq a$ for $k$ big, we finally deduce from \eqref{ujsndftaz1} and \eqref{ujsndftaz} that
    \begin{align*}
        \psi_{x_k,v_k}(t_k)\to t_0.
    \end{align*}
This proves the claim.
\hfill \checkmark

 We define the set ${\cal D}_L$ as the image of the map
 \begin{align*}
     {\cal D}\to \overline \C\times \R,\ (x,v,t)\mapsto (x,v,\psi_{x,v}(t)).
 \end{align*}
Clearly, each slice $\{t\in \R\mid (x,v,t)\in {\cal D}_L\}$ is an open interval containing $0$. Using that $\psi_{x,v}$ is strictly increasing and that $\psi$ is continuous one readily checks that ${\cal D}_L$ is open as a subset of $\overline \C\times \R$. 
 \\
 
 \noindent \textbf{Claim 2: The map
 \begin{align*}
     {\cal D}_L\to \R,\ (x,v,t)\mapsto \psi_{x,v}^{-1}(t),
 \end{align*}
 is continuous and continuously differentiable w.r.t.\ $t$.}
 \\

Let $(x_0,v_0,t_0)\in {\cal D}_L$ and $(x_k,v_k,t_k)\in {\cal D}_L$ be a sequence converging to $(x_0,v_0,t_0)$.
Let $s_0\in \R$ with $\psi_{x_0,v_0}(s_0)=t_0$. Let $\ep>0$ such that $\psi_{x_0,v_0}$ is defined on $[s_0-\ep,s_0+\ep]$ and denote 
\begin{align*}
\delta:=\min\{|\psi_{x_0,v_0}(s_0\pm \ep)-t_0|\}.
\end{align*}
Since $\psi$ is continuous it follows that, if $k$ is large enough, we have
\begin{align*}
\min\{t_k-\psi_{x_k,v_k}(s_0-\ep),\psi_{x_k,v_k}(s_0+ \ep)-t_k\}\geq \frac{\delta}{2}. 
\end{align*}
Using the fact that $\psi$ is strictly increasing as a function of $s$ we deduce that, as $k$ is large, we must have
\begin{align*}
    |\psi_{x_k,v_k}(s)-t_k|\geq \frac{\delta}{2} \text{ for } s\leq s_0-\ep \text{ or } s\geq s_0+\ep.
\end{align*}
Thus,
\begin{align*}
|\psi_{x_k,v_k}^{-1}(t_k)-s_0| \leq \ep.
\end{align*}
This proves the continuity and the continuous differentiability w.r.t.\ $t$ follows from the fact that, easily verified,
\begin{align*}
   \dot \psi_{x,v}^{-1}(t)=\frac{L_1(v)}{L_1(\frac{d}{ds}|_{s=\psi_{x,v}^{-1}(t)}\exp_x(sv))},\ v \neq 0,\  \dot \psi_{x,0}^{-1}(t)=1,
\end{align*}
together with the same proof as in the first claim.
This proves the second claim. 
\hfill \checkmark
\bigskip
 
   We now define our map $\phi$ by
 \begin{align*}
     \phi:{\cal D}_L\to \overline \C,\ \phi_t(x,v):=(\exp_x(\psi_{x,v}^{-1}(t)v),\frac{d}{dt}(\exp_x(\psi_{x,v}^{-1}(t)v))).
 \end{align*}
 Observe that $\phi$ actually maps to $\overline \C$.
We are left to check all the claimed properties. We have already proved that, for any $(x,v)\in \overline \C$, the $t$-slice in an open interval containing $0$ and that ${\cal D}_L$ is relatively open. Moreover, from the above claim it follows that $\phi$ is continuous. Part (i) is trivial since geodesics are smooth and $\psi_{x,v}$ is smooth (for fixed $(x,v)$) and $\dot \psi_{x,v}(t)>0$, so that $\psi_{x,v}^{-1}$ is also smooth. Thus, it remains to prove the remaining properties of a local flow (see claim 3) and (ii) and (iii).

In the following, for a better overview, given $(x,v)\in \overline  \C$ we denote the maximal existence interval of $\exp_x(tv)$ by $I_{x,v}$ or by $I$ if it is clear about which $(x,v)$ we are speaking. We also denote $\psi:=\psi_{x,v}:I\to J:=\psi(I)$, so that $\psi^{-1}:J\to I$.
\\

\noindent \textbf{Claim 3:
 $\phi$ is a local flow}

Recall that we already proved that ${\cal D}_L\subseteq \overline \C\times \R$ is open, that each $t$-slice is an open interval containing $0$ and that $\phi$ is continuous. Thus, it suffices to show that, given $(x,v)\in \overline \C$ and $s\in J$ we have $J_{\phi_s(x,v)}\supseteq J-s$ and
\begin{align*}
    \phi_t(\phi_s(x,v))=\phi_{t+s}(x,v) \text{ for all } t\in J-s.
\end{align*}
Indeed, if we choose $t=-s$ it follows $\phi_{-s}(\phi_s(x,v))=(x,v)$ and by the arbitrariness of $(x,v)$ and $s$ we obtain that $J=J_{\phi_{-s}(\phi_s(x,v))}\supseteq J_{\phi_s(x,v))}-(-s)$. Hence, also $J_{\phi_s(x,v)}\subseteq J-s$, therefore $J_{\phi_s(x,v)}= J-s$.
\\

Since the case $v=0$ is trivial, we consider the case $v\neq 0$. Let $t\in J-s$ be arbitrary. We consider the geodesic $c(r):=\exp_x(rv)$, $r\in I$. We have by definition of $\phi$
\begin{align*}
    \phi_s(x,v)=(c(\psi^{-1}(s),(c\circ \psi^{-1})'(s))=:(y,w)\in TM,
\end{align*}
where $'$ also denotes the derivative.
Observe that, by definition of $\psi$ and the chain rule, 
\begin{align}
    w=\dot c(\psi^{-1}(s))\cdot \frac{L_1(v)}{L_1(\dot c(\psi^{-1}(s)))}\neq 0. \label{tdfaudai9o8u}
\end{align}
Thus, denoting $\bar c(r):=\exp_y(rw)$ and its maximal existence interval by $\bar I$, we have
\begin{align*}
    \bar I=\frac{L_1(\dot c(\psi^{-1}(s)))}{L_1(v)}(I-\psi^{-1}(s))\ni \frac{L_1(\dot c(\psi^{-1}(s)))}{L_1(v)}(\psi^{-1}(t+s)-\psi^{-1}(s))=:\bar t.
\end{align*}
Thus, $\psi_{\phi_s(x,v)}(\bar t)$ is well-defined and with the substitution $\tau=\frac{L_1(v)}{L_1(\dot c(\psi^{-1}(s)))}r+\psi^{-1}(s)$ in the following integral we get
\begin{align*}
\psi_{\phi_s(x,v)}(\bar t)&=\frac{1}{L_1(w)}\int_0^{\bar t} L_1(\dot {\bar c}(r))\, dr
\overset{\eqref{tdfaudai9o8u}}{=}\frac{1}{L_1(v)}\int_0^{\bar t} L_1(\dot {\bar c}(r))\, dr
\\[10pt]
&=\frac{1}{L_1(v)}\int_{\psi^{-1}(s)}^{\psi^{-1}(t+s)} L_1\(\dot {\bar c}\(\frac{L_1(\dot c(\psi^{-1}(s)))}{L_1(v)}(\tau-\psi^{-1}(s))\)\)\frac{L_1(\dot c(\psi^{-1}(s)))}{L_1(v)}\, d\tau 
\\[10pt]
&=\frac{1}{L_1(v)}\int_{\psi^{-1}(s)}^{\psi^{-1}(t+s)} L_1\(\frac{d}{d\tau} \({\bar c}\Big(\frac{L_1(\dot c(\psi^{-1}(s)))}{L_1(v)}(\tau-\psi^{-1}(s))\Big)\)\)\, d\tau
\\[10pt]
&\overset{\eqref{tdfaudai9o8u}}{=}\frac{1}{L_1(v)}\int_{\psi^{-1}(s)}^{\psi^{-1}(t+s)} L_1\(\frac{d}{d\tau} \exp_y\Big((\tau-\psi^{-1}(s)) \dot c(\psi^{-1}(s))\Big)\)\, d\tau
\\[10pt]
&=\frac{1}{L_1(v)}\int_{\psi^{-1}(s)}^{\psi^{-1}(t+s)} L_1\(\frac{d}{d\tau} \exp_x(\tau v)\)\, d\tau
\\[10pt]
&=\psi(\psi^{-1}(t+s))-\psi(\psi^{-1}(s))=t+s-s=t.
\end{align*}
By the arbitraryness of $t$ this proves $J_{\phi_s(x,v)}\supseteq J-s$ and it also shows that
\begin{align*}
   \bar c(\psi^{-1}_{\phi_s(x,v)}(t))=\exp_y(\bar tw)=\exp_y((\psi^{-1}(t+s)-\psi^{-1}(s))\cdot\dot c(\psi^{-1}(s)))
   =
   c\circ \psi^{-1}(t+s).
\end{align*}
Hence, 
\begin{align*}
    \phi_t(\phi_s(x,v))=(\bar c(\psi^{-1}_{\phi_s(x,v)}(t)),({\bar c}\circ \psi^{-1}_{\phi_s(x,v)})'(t))=(c\circ \psi^{-1}(t+s),(c\circ \psi^{-1})'(t+s))=\phi_{t+s}(x,v).
\end{align*}
This proves that $\phi$ is indeed a local flow.
\hfill \checkmark 
\bigskip

It remains to prove (ii) and (iii).

We start with (iii). Let $(x,v)\in \overline \C$ and $t>0$. Denote as usual by $J$ the domain of $\phi_\cdot(x,v)$. We need to prove that $\phi_{\cdot}(x,tv)$ is defined on $\frac{1}{t}J$ and that
\begin{align*}
    \pi\circ \phi_s(x,tv)=\pi \circ \phi_{ts}(x,v) \text{ for all } s\in \frac{1}{t}J.
\end{align*}
This is clear if $v=0$, so let us assume that $v\neq 0$.
Consider the geodesic $c(s)=\exp_x(sv)$ with its maximal existence interval $I$. The geodesic $\bar c(s)=\exp_x(s(tv))$ is defined on $\frac{1}{t}I$. Thus, $\psi_{x,tv}$ maps from $\frac{1}{t}I$ to $\R$ and we have for $s\in \frac{1}{t}I$, using the positive $1$-homogeneity of $L_1$ as a function from the tangent space, that
\begin{align*}
    \psi_{x,tv}(s)=\frac{1}{L_1(tv)}\int_0^s L_1(\dot {\bar c}(\tau))\, d\tau
    = \frac{1}{L_1(v)}\int_0^s  L_1(\dot c(t\tau))\, d\tau
    = \frac{1}{tL_1(v)}\int_0^{ts}  L_1(\dot c(\tau))\, d\tau
    =\frac{\psi_{x,v}(ts)}{t}.
\end{align*}
Thus $\psi_{x,tv}(\frac{1}{t}I)=\frac1t \psi_{x,v}(I)$ and, hence, $\phi_{\cdot}(x,tv)$ is defined on $\frac{1}{t}J$, as claimed. This computation shows moreover
\begin{align*}
    \psi^{-1}_{x,tv}(s)= \frac{1}{t}\psi_{x,v}^{-1}(ts) \text{ for } s\in \frac{1}{t}J.
\end{align*}
Thus,
\begin{align*}
    \pi\circ \phi_s(x,tv)=\bar c(\psi_{x,tv}^{-1}(s))=c(t \psi_{x,tv}^{-1}(s))=c(\psi_{x,v}^{-1}(ts))=\pi\circ \phi_{ts}(x,v)
\end{align*}
which proves the claim.

For (ii), assume $x<y$ because the case $x=y$ is trivial. Let $\gamma:[a,b]\to M$ be any minimizer for ${\cal A}_2$ connecting $x$ with $y$ and $s,t\in [a,b]$. 

In this case we saw in the proof of Lemma \ref{minimizer} that $\gamma(t)=\pi \circ \phi_{t-s}(\gamma(s),\dot \gamma(s))$.
 \end{proof}

\begin{corollary}\label{compact}
Let $K_0,K_1\subseteq M$ be compact and assume that $K_0\times K_1\subseteq I^+$. Then the set
\begin{align*}
\Gamma_{K_0,K_1}:=\{\gamma\in \Gamma\mid \gamma(0)\in K_0,\ \gamma(1)\in K_1\}
\end{align*}
is compact in $C^1([0,1],M)\subseteq C^0([0,1],TM)$ endowed with the topology of uniform convergence (w.r.t.\ one (hence all) metric(s) on $TM$).
\end{corollary}
\begin{proof}
Let $(\gamma_k)\subseteq \Gamma_{K_0,K_1}$ be any sequence. 
We need to show that there is $\gamma \in \Gamma_{K_0,K_1}$ such that, along a subsequence,
$\gamma_k\to \gamma$ as $k\to \infty$.

We already know that $\gamma_k(t)=\exp_L(x_k,tv_k)$ for some $x_k\in K_0$ and $v_k:=\dot \gamma_k(0)$. 
Set
\begin{align*}
C:=\sup\{c_1(x,y)\mid x\in K_0,\ y\in K_1\}.
\end{align*}
Then, using \eqref{splitting}, we see that, for each $k\in \N$ and each $t\in [0,1]$,
\begin{align*}
|\dot \gamma_k(t)|_h\leq d\tau(\dot \gamma_k(t))\leq 2 L_1(\dot \gamma_k(t))=2c_1(\gamma_k(0),\gamma_k(1))\leq 2C.
\end{align*}
From this we deduce two consequences:
Firstly, the family of curves $(\gamma_k)$ is equi-Lipschitz and the set $\{\gamma_k(0)\}\subseteq K_0$ is relatively compact. Thus, by the theorem of Arzelà-Ascoli the family of curves $(\gamma_k)$ is relatively compact in $C([0,1],M)$. Thus, without loss of generality, we can assume that $\gamma_k \to \gamma \in C([0,1],M)$.
Secondly, the sequence $(x_k,v_k)=(\gamma_k(0),\dot \gamma_k(0))$ is relatively compact in $TM$. Thus, for a subsequence that we do not relabel, $(x_k,v_k)\to (x,v)\in \overline \C$ as $k\to \infty$. By the continuity of the exponential map we clearly have
\begin{align*}
    \exp_L(x,tv)=\gamma(t) \text{ for all } t\in [0,1] \text{ such that both curves are defined.}
\end{align*}
Assume that $\exp_L(x,tv)$ is only defined up to some $t_0<1$. In particular, $v\neq 0$. Then $\exp_L(x,tv)\to \gamma(t_0)$ as $t\to t_0$ and, hence, $\exp_L(x,tv)$ is extendible. But this cannot be the case since $\exp_L(x,tv)$ is a reparametrization of a non-constant ($v\neq 0$) maximal geodesic as we have seen in the above proof.
Thus, $\exp_L(x,tv)$ is defined for all $t\in [0,1]$ and by the continuous differentiability of $(x,v,t)\to \exp_L(x,tv)$ w.r.t.\ $t$ (which holds since $\phi$ is continuous) we have that $\gamma_k\to \gamma$ in $C^1([0,1],M)$. Then we conclude the proof as follows:
By the continuity of $c_2$ on $J^+$,
\begin{align*}
    c_2(x,\exp_L(x,v))=\lim_{k\to \infty} c_2(x_k,\exp_L(x_k,v_k))= \lim_{k\to \infty} {\cal A}_2(\gamma_k)={\cal A}_2(\gamma).
\end{align*}
Thus, $\gamma=\exp_L(x,tv)$ minimizes ${\cal A}_2$ and, hence, belongs to $\Gamma_{K_0,K_1}$.
\end{proof}

\subsection{Local semiconvexity of $c_2$.}
In this subsection we want to prove that the cost function $c_2$ is locally semiconcave on $I^+$ (see Proposition \ref{injiuhgb}).

Let us mention that the local semiconcavity (with a linear modulus) of the negative Lorentzian distance function on the set $I^+$ was already proven by McCann in \cite{McCann2}, see also Remark \ref{hiahaufua8o}(b). From this we can deduce that the cost function $c_1$ (and, hence, also $c_2$) is locally semiconcave on $I^+$. 
However, for the sake of completeness, we will give a proof which is oriented towards the proof of Theorem B19 in \cite{Fathi/Figalli}.

In this subsection, we use the same notation as during the paper (i.e.\ $(M,g)$ will always denote a globally hyperbolic spacetime, $n+1:=\dim(M)$, $c_2$ is our cost function etc.).

We start with the definitions of superdifferentiability and local semiconcavity.

\begin{definition}\rm
    \begin{enumerate}[(a)]
\item
    A function $f:U\to \R$, where $U\subseteq \R^d$ is an open set, is \emph{superdifferentiable} at $x\in U$ with \emph{superdifferential} $p\in (\R^d)^*$ if
    \begin{align*}
    f(y)\leq f(x)+ \langle p,y-x\rangle+o(|y-x|).
    \end{align*}
    \item A function $f:U\to \R$ defined on an open subset $U\subseteq N$ of a smooth manifold $N$ is called \emph{superdifferentiable} at $x\in U$ with \emph{superdifferential} $p\in T_x^*N$ if for one (or any) chart $(\phi,V)$ of $U$ around $x$ the function $f\circ \phi^{-1}$ is superdifferentiable at $\phi(x)$ with superdifferential $p\circ d_x\phi^{-1}$.
    \end{enumerate}
\end{definition}

\begin{definition}\rm \label{opaudowzuafdddf}
\begin{enumerate}[(a)]
\item
    A function $f:U\to \R$, where $U\subseteq \R^d$ is an open set, is called \emph{locally semiconcave}, if for any $x_0\in U$ we find an open neighborhood $V\subseteq U$ around $x_0$ and a modulus of continuity $\omega$ such that:
    \begin{align*}
        \forall x\in V: \exists p\in (\R^d)^*: f(y)\leq f(x)+\langle p,y-x\rangle +|y-x|\omega(|y-x|) \ \forall y\in V.
    \end{align*}
    \item A function $f:U\to \R$ defined on an open subset $U\subseteq N$ of a smooth manifold $N$ is called \emph{locally semiconcave} if for any chart $(\phi,V)$ of $U$ the function $f\circ \phi^{-1}$ is locally semiconcave.
    \item  A function $f:U\to \R$ defined on an open subset $U\subseteq N$ of a manifold $N$  is called \emph{locally semiconvex} if the function $-f$ is locally semiconcave.
    \end{enumerate}
\end{definition}

\begin{remark}\rm\label{hiahaufua8o}
\begin{enumerate}[(a)]
\item It can be shown that the composition $f\circ \phi$ of a locally semiconcave function $f:\R^d\supseteq U\to \R$ with a $C^1$-function $\phi:V\to U$ is again locally semiconcave (see \cite{Fathi/Figalli}, Lemma A9). Thus, to prove local semiconcavity on a manifold it suffices to consider an atlas.
\item
At this point we should mention that other authors refer to condition (b) in the above definition as locally superdifferentiable and mean by locally semiconcave that the function is locally superdifferentiable with a linear modulus (that is, the modulus $\omega$ can be chosen to be linear). Being locally superdifferentiable (resp.\ with a linear modulus) is equivalent to say that the function is locally geodesically semiconcave (resp.\ with a linear modulus), see \cite{Villani}, Proposition 10.12. 
\item 
In this chapter we will use several properties of locally semiconvex (resp. semiconcave) functions. We refer the reader to \cite{Fathi/Figalli} for all the proofs.
\end{enumerate}
\end{remark}

\begin{notation}\rm
Recall the definition of $\Gamma_{x,y}$. More generally assume that we have given two sets $U,V\subseteq M$ such that $U\times V\subseteq I^+$. We define the set $\Gamma_{U,V}$ by
\begin{align*}
\Gamma_{U,V}:=\{\gamma:[0,1]\to M\mid \gamma(0)\in U,\,  \gamma(1)\in V,\, \gamma \text{ minimizes }{\cal A}_2\}.
\end{align*}
\end{notation}

\begin{proposition}\label{injiuhgb}
The function $c_2$ is locally semiconcave on $I^+$.

 Moreover, given $(x,y)\in I^+$, a superdifferential of $c_2$ at $(x,y)$ is given by
\begin{align}
T_xM\times T_yM\to \R,\ (v,w)\to \frac{\partial L_2}{\partial v}(\gamma(1),\dot \gamma(1))(w)-
\frac{\partial L_2}{\partial v}(\gamma(0),\dot \gamma(0))(v), \label{super}
\end{align}          
where $\gamma\in \Gamma_{x,y}$.
\end{proposition}
\begin{proof}
The proof is oriented towards \cite{Fathi/Figalli}, Theorem B19. We use a similar strategy and notation.

Let $(x_0,y_0)\in I^+$. We have to show that there exists a chart around $(x_0,y_0)\in M\times M$ such that $c_2$ is locally semiconcave when computed in local coordinates and that a superdifferential is given by \eqref{super}.
\medskip

First, since $c_2(x_0,y_0)>0$ and $c_2$ is continuous on $J^+$ we find open and bounded neighborhoods $U_1$ and $V_1$ around $x_0$ and $y_0$ such that 
\begin{align*}
\overline{U_1}\times \overline{V_1}\subseteq I^+ \text{ and } C_0:=\sup_{U_1\times V_1} c_2(x,y)<\infty. 
\end{align*}
Without loss of generality we can choose $U_1$ and $V_1$ in a way that we find two charts $\phi_x:U\to \R^{1+n}$ and $\phi_y:V\to \R^{1+n}$ such that $U_1=\phi_x^{-1}(B_1(0))$ and $V_1=\phi_y^{-1}(B_1(0))$. 
\medskip

\noindent\textbf{1. Claim:}
The set $\{(\gamma(t),\dot \gamma(t))\mid \gamma\in \Gamma_{U_1,V_1},\ t\in[0,1]\}$ is relatively compact in $TM$.
\medskip

\noindent \textbf{Proof of claim:}
Let $(x,y)\in U_1\times V_1$ be arbitrary and let $\gamma\in \Gamma_{x,y}$.
As we have seen in Lemma \ref{minimizer}, $L_1(\dot \gamma(t))$ needs to be constant and, hence, equal to $c_1(x,y)$. 
Hence, using \eqref{splitting}, we get for all $t\in [0,1]$ that
\begin{align}
|\dot \gamma(t)|_h \leq d\tau(\dot \gamma(t))\leq 2L_1(\dot \gamma(t))\leq 2C_0.\label{dsjidoskdsps}
\end{align}
Thus, the set in the statement of the claim is contained in $\{(x,v)\in TM\mid x\in B_{2C_0}(U_1)\ |v|_h\leq 2C_0\}$. The latter set is clearly relative compact in $TM$ by the boundedness of $U_1$ and the completeness of $M$. This proves the claim.
\hfill \checkmark
\medskip

We return to the proof. From the claim (or, more precisely, from \eqref{dsjidoskdsps}) we deduce that there exists $\ep>0$ such that for all $\gamma\in \Gamma_{U_1,V_1}$ it holds
\begin{align}
 \gamma([0,\ep])\subseteq  \phi_x^{-1}(B_2(0)) \text{ and } \gamma([1-\ep,1])\subseteq \phi_y^{-1}(B_2(0)). \label{dgaidwuasasasafioas}
\end{align}

\noindent \textbf{2. Claim:}
We claim that there exists $1>\delta>0$ such that the sets
\begin{align*}
    \{(\bar \gamma(t),\dot {\bar \gamma}(t))\mid \gamma\in \Gamma_{U_1,V_1},\ t\in[0,\ep]\},\ \{(\bar \gamma(t),\dot {\bar \gamma}(t))\mid \gamma\in \Gamma_{U_1,V_1},\ t\in[1-\ep,1]\}
\end{align*}
are relatively compact in $TU\cap \op{int}(\C)$ and $TV\cap \op{int}(\C)$, where for any $\gamma \in \Gamma_{U_1,V_1}$ and $h_x,h_y\in \R^{1+n}$ with $|h_x|,|h_y|\leq \delta$ we define the curve $\bar \gamma$ by
\begin{align}
\bar \gamma:[0,1]\to M,\ \bar \gamma(t):=
\begin{cases}
 \phi_x^{-1}\(\frac{\ep-t}{\ep}h_x+\phi_x(\gamma(t))\),\ &t\leq \ep,
\\[10pt]
\gamma(t),\ &t\in [\ep,1-\ep],
\\[10pt]
\phi_y^{-1}\(\frac{t-(1-\ep)}{\ep}h_y+\phi_y(\gamma(t))\),\ &t\geq 1-\ep.
\end{cases} \label{hdaiujoaiopd}
\end{align}
\medskip

\noindent \textbf{Proof of claim:} We only check the statement for the first set since we can deal with the second analogously. Thanks to the first claim and \eqref{dgaidwuasasasafioas} the set $T\phi_x((\gamma(t),\dot \gamma(t)))_{\gamma \in \Gamma_{U_1,V_1},t\leq \ep}$ is relatively compact in $T\phi_x(TU)=\R^{1+n}\times \R^{1+n}$. Moreover, the map
\begin{align*}
e: \Gamma_{\overline U_1,\overline V_1}\times [0,1]\to \C,\ (\gamma,t) \mapsto (\gamma(t),\dot \gamma(t)),
\end{align*}
is continuous and $\Gamma_{\overline U_1,\overline V_1}\times [0,1]$ is compact as a subset of $C^1([0,1],M)\times [0,1]$. Both properties follow from Lemma \ref{compact}. 
Since every $\gamma \in\Gamma_{\overline U_1,\overline V_1}$ is a maximizer for the Lorentzian length functional and $\overline U_1\times \overline V_1\subseteq I^+$ it follows that each $\gamma$ is timelike, so that $e(\Gamma_{\overline U_1,\overline V_1}\times [0,1])$ is a compact set inside $\op{int}(\C)$. 

Thus, $T\phi_x((\gamma(t),\dot \gamma(t)))_{\gamma \in \Gamma_{U_1,V_1},t\leq \ep}$ is relatively compact in $T\phi_x(\op{int}(\C)\cap TU)$. Then the statement of the claim is obvious.
\hfill \checkmark
\\

With this claim at hand we can now prove the lemma in the same manner as in \cite{Fathi/Figalli}.

Given $(x_1,y_1),(x_2,y_2)\in B_{\frac{\delta}{2}}(0)\times B_{\frac{\delta}{2}}(0)$ set $h_x:=x_2-x_1$ and $h_y:=y_2-y_1$. Let $\gamma\in \Gamma_{\phi_x^{-1}(x_1),\phi_y^{-1}(y_1)}$ and define $\bar \gamma$ by \eqref{hdaiujoaiopd}. We can then estimate 
\begin{align*}
&c_2(\phi_x^{-1}(x_2),\phi_y^{-1}(y_2))-c_2(\phi_x^{-1}(x_1),\phi_y^{-1}(y_1))
\\[10pt]
\leq \ &\int_0^1 L_2(\dot {\bar \gamma}(t))\, dt-\int_0^1 L_2(\dot \gamma(t))\, dt
\\[10pt]
=\ &\int_0^\ep L_2(\dot {\bar \gamma}(t))-L_2(\dot \gamma(t))\, dt+\int_{1-\ep}^{1} L_2(\dot {\bar \gamma}(t))-L_2(\dot \gamma(t))\, dt.
\end{align*}
We deal with the first integral since the second can be treated analogously. As in \cite{Fathi/Figalli}, we define the new Lagrangian
\begin{align*}
\tilde L_{2,x}:\R^{1+n}\times \R^{1+n}\to \R,\ \tilde L_{2,x}(x,v):=L_2(\phi_x^{-1}(x),d_x\phi_x^{-1}(v))
\end{align*}
and the new curves
\begin{align*}
    \alpha_x:=\phi_x\circ \gamma \text{ and } \bar\alpha_x=\phi_x \circ \bar \gamma,
\end{align*}
so that
\begin{align*}
    \int_0^\ep L_2(\dot {\bar \gamma}(t))-L_2(\dot \gamma(t))\, dt=\int_0^\ep \tilde L_{2,x}(\dot {\bar \alpha}_x(t))- \tilde L_{2,x}(\dot \alpha_x(t))\, dt.
\end{align*}
 The Lagrangian $\tilde L_{2,x}$ is smooth on $T\phi_x(\op{int}(\C)\cap TU)$, where $T\phi_x$ again denotes the tangent bundle chart induced from $\phi_x$. 
Thanks to claim 2 we can find a compact set $K\subseteq T\phi_x(\op{int}(\C)\cap TU)$ (which is independent of the choice of $(x_1,y_1),(x_2,y_2)$) such that $(1-s)(\alpha(t),\dot \alpha(t))+s(\bar \alpha(t),\dot {\bar \alpha}(t))\in K$ for all $s\in [0,1]$ and $t\in [0,\ep]$. Since $\tilde L_{2,x}$ is smooth on a neighborhood of $K$, there exists a modulus of continuity $\omega_x$ for the differential $D\tilde L_{2,x}$ restricted to $K$ (observe that also this modulus is independent of the chosen points $(x_1,y_1),(x_2,y_2)$). Using the mean value theorem and then the modulus $\omega_x$ we finally deduce
\begin{align*}
\int_0^\ep \tilde L_{2,x}(\dot {\bar  \alpha}_x(t))-\tilde L_{2,x}(\dot \alpha_x(t))\, dt
=
\int_0^\ep 
D\tilde L_{2,x}(\alpha_x(t),\dot { \alpha}_x(t))\Big[\frac{\ep-t}{\ep}h_x,-\frac{h_x}{\ep}\Big]\, dt +\omega_x\Big(\frac{h_x}{\ep}\Big) |h_x|.
\end{align*}
The same computation for the second integral shows that we have
\begin{align*}
&c_2(\phi_x^{-1}(x_2),\phi_y^{-1}(y_2))-c_2(\phi_x^{-1}(x_1),\phi_y^{-1}(y_1))
\\[10pt]
\leq
&\int_0^\ep 
D\tilde L_{2,x}(\alpha_x(t),\dot { \alpha}_x(t))\Big[\frac{\ep-t}{\ep}h_x,-\frac{h_x}{\ep}\Big]\, dt +
\int_{1-\ep}^1 
D\tilde L_{2,y}(\alpha_y(t),\dot { \alpha}_y(t))\Big[\frac{t-(1-\ep)}{\ep}h_y,\frac{h_y}{\ep}\Big]\, dt
\\[10pt]
+&\omega_x\Big(\frac{h_x}{\ep}\Big) |h_x|+\omega_y\Big(\frac{h_y}{\ep}\Big) |h_y|.
\end{align*}
This proves the first part of the lemma. 
Moreover, using the Euler-Lagrange equation for timelike minimizers (recall that $L_2$ is smooth on $\op{int}(\C)$) and integrating the above integrals by parts, this proof also shows that a superdifferential of $c_2$ at some $(x,y)\in U_1\times V_1$ is given by
\begin{align*}
(w_0,w_1)\to \frac{\partial L_2}{\partial v}(\gamma(1),\dot \gamma(1))(w_1)-
\frac{\partial L_2}{\partial v}(\gamma(0),\dot \gamma(0))(w_0),
\end{align*}          
where $\gamma\in \Gamma_{x,y}$. See also Corollary B20 in \cite{Fathi/Figalli}.
\end{proof}

\begin{remark}\rm
Let us mention that the same proof also shows that $c_2$ is actually locally semiconcave with a linear modulus.
We just have to use a Lipschitz constant for $D{\tilde L}_{2,x}$ on the compact set $K$ instead of the modulus of continuity.
\end{remark}

\begin{corollary}\label{twist}
    The function $c_2$ satisfies the twist condition on $I^+$, that is,
    for each $x\in M$ the map
    \begin{align*}
        I^+(x)\ni y\mapsto \frac{\partial c_2}{\partial x}(x,y)\in {\cal L}(T_xM,\R)
    \end{align*}
    is injective on its domain of definition.
\end{corollary}
\begin{proof}
    If $x\in M$ and $y\in I^+(x)$ are given, we know that a superdifferential for the map $c_2(\cdot,y)$ at $x$ is given by
    \begin{align*}
       - \frac {\partial L_2}{\partial v}(x,v),
    \end{align*}
    where $v=\dot \gamma(0)\in \op{int}(\C_x)$ and $\gamma \in \Gamma_{x,y}$.
    Thus, if $c_2(\cdot,y)$ is differentiable at $x$ then $\frac{\partial c_2}{\partial x}(x,y)=-\frac {\partial L_2}{\partial v}(x,v)$. So, to prove the claim it suffices to prove that 
    \begin{align*}
        \op{int}(\C_x)\to \R,\ v\mapsto L_2(x,v),
    \end{align*}
    is strictly convex.
   We prove this in the two subsequent lemmas.
\end{proof}

\begin{lemma}
    Consider the Minkowski space $(\R^{1+n},\langle \langle \cdot,\cdot\rangle \rangle)$. Denote $f(v):=-\sqrt{|\langle \langle v,v\rangle \rangle|}$. Then, for $v\in \R^{1+n}$ timelike and $w\in \R^{1+n}\backslash \op{span}(v)$ we have
    \begin{align*}
       D^2f(v)[w,w]>0.
    \end{align*}
\end{lemma}
\begin{proof}
    An easy computation shows that
    \begin{align*}
        D^2f(v)[w,w]=\frac{-1}{f(v)^3} (-\langle \langle v,v\rangle \rangle \langle \langle w,w\rangle \rangle +\langle \langle v,w\rangle \rangle^2).
    \end{align*}
    Since $(-1)/f(v)^3>0$ by the timelikeness of $v$ it suffices to prove that the second part in the above equation is strictly positive. By applying a Lorentz transformation if necessary we can assume that $v=ae_0$ for some $a\neq 0$. Denoting $w=(w_0,\bar w)$ and denoting the Euclidean norm of $\bar w$ with $|\bar w|$ it then follows
    \begin{align*}
        (-\langle \langle v,v\rangle \rangle \langle \langle w,w\rangle \rangle +\langle \langle v,w\rangle \rangle^2)
        =a^2 (-w_0^2+|\bar w|^2)+(aw_0)^2
        >0.
    \end{align*}
    In the last step we have used that $\bar w\neq 0$ which follows from our assumption $w\notin \op{span}(v)$.
\end{proof}

\begin{corollary} \label{pokjnbvcxdertzujik}
   Given $(x,v)\in \op{int}(\C)$, the bilinear form
    \begin{align*}
        \frac{\partial^2 L_2}{\partial v^2}:T_xM\times T_xM\to \R
    \end{align*}
    is positive definite. In particular, $L_2(x,\cdot)$ is strictly convex on $\op{int}(\C_x)$.
\end{corollary}
\begin{proof}
Since $x$ is fixed we can define the map
\begin{align*}
    f:T_xM\to \R,\ f(w):=-\sqrt{|g_x(w,w)|}.
\end{align*}
Observe that, in appropriate coordinates (namely choosing on orthonormal basis), this function is precisely the function $f$ from the above lemma. Then we have for $w\in T_xM$ a timelike vector $L_2(x,w)=(d\tau(w)+f(w))^2$. Taking derivatives we obtain
\begin{align*}
   \frac{\partial L_2}{\partial v}(w)=2 (d\tau(v)+f(v)) (d\tau(w)+Df(v)[w])  
\end{align*}
and, thus,
\begin{align*}
   \frac{\partial^2 L_2}{\partial v^2}(w)=2  (d\tau(w)+Df(v)[w])^2+2(d\tau(v)+f(v))D^2f(v)[w,w].
\end{align*}
Since $v$ is timelike we have $d\tau(v)+f(v)>0$ by \eqref{splitting}. Moreover, since $f$ is well-known to be convex on $\op{int}(\C_x)$, we also have $D^2f(v)[w,w]\geq 0$ (which also follows from the proof of the above lemma). Thus, we obtain that both terms on the right hand side of the above equation are non-negative. If $w\notin \op{span}(v)$ it follows from the preceding lemma that the second term is strictly positive. If, on the other hand, $w=\lambda v$ for some $\lambda\neq 0$ the first term is equal to $2\lambda^2 (d\tau(v)+f(v))^2$, which is strictly positive by \eqref{splitting}. This proves the lemma.
\end{proof}

\subsection{Countable rectifiability}

This section is based on section 10 of \cite{Villani}. We modify the proof of Theorem 10.48 to get a slightly stronger result which we need in our case.
We first recall the basic terminology and give the definitions of tangent cone and countable rectifiability as given in \cite{Villani}.

\begin{definition}\rm\label{tangentcone}
Let $S\subseteq \R^n$ and $x\in \bar S$. The \emph{tangent cone} to $S$ at $x$ is defined as
\begin{align*}
T_xS:=\bigg\{\lim_{k\to \infty} \frac{x_k-x}{t_k}\mid x_k\in S,\ x_k\to x,\ t_k\to 0,\ t_k>0\bigg\}.
\end{align*}
\end{definition}

\begin{definition}\rm
A subset $S\subseteq \R^n$ is said to be \emph{countably $(n-1)$-rectifiable} if there exist measurable sets $D_k\subseteq \R^{n-1}$ and Lipschitz continuous functions $f_k:D_k\to \R^n$, $k\in \N$, such that
\begin{align*}
S\subseteq \bigcup_{k=1}^\infty f_k(D_k).
\end{align*}
\end{definition}

\begin{theorem}\label{rect}
Let $S\subseteq \R^n$ be a set and assume that there exists $\alpha>0$ such that, for any $x\in S$, it holds 
$T_xS\subseteq \R^n\backslash\op{Cone}(v_x,\alpha)$ where $v_x\in S^{n-1}$. Then, $S$ is countably $(n-1)$-rectifiable.
\end{theorem}
\begin{proof}
We modify the proof of Theorem 10.48(ii) of \cite{Villani}.

We have
\begin{align}
\langle \zeta, v_x\rangle \leq \alpha |\zeta| \text{ for all } \zeta\in T_xS. \label{dkdalpas}
\end{align}

Let $F$ be a finite set in $S^{n-1}$ such that the balls $(B_{\frac{1-\alpha}{4}}(\nu))_{\nu \in F}$ cover $S^{n-1}$. We claim that:
\begin{align*}
\forall x\in S\ \ \exists r>0,\nu\in F:\forall y\in  S\cap B_r(x): \langle y-x,\nu\rangle \leq \frac{1+\alpha}{2}|y-x|.
\end{align*}
If the claim was false we find some $x\in S$ and a sequence $y_k\in  S\cap B_{1/k}(x)$ such that for all $\nu\in F$
\begin{align*}
\langle y_k-x,\nu\rangle > \frac{1+\alpha}{2} |y_k-x|.
\end{align*}
Now choose $\nu\in F$ such that $v_x\in B_{(1-\alpha)/4}(\nu)$. Then
\begin{align*}                                             
\langle y_k-x,v_x\rangle
=
\langle y_k-x,\nu\rangle +
\langle y_k-x,v_x-\nu\rangle
&>
\frac{1+\alpha}{2} |y_k-x|-|y_k-x||v_x-\nu|
\\[10pt]
&\geq 
\frac{1+\alpha}{2} |y_k-x|-\frac{1-\alpha}{4}|y_k-x|
\\[10pt]
&=\frac{1+3\alpha}{4}|y_k-x|.
\end{align*}
Without loss of generality we can assume that $\frac{y_k-x}{|y_k-x|}$ converges to some $\zeta \in T_xS\cap S^{n-1}$. Taking limits in the previous computation we obtain
\begin{align*}
\langle \zeta, v_x\rangle \geq \frac{1+3\alpha}{4}.
\end{align*}
This is now a contradiction to \eqref{dkdalpas} and thus we proved the claim.

The second part of the proof consists of even more obvious modifications of the proof in \cite{Villani}: From the claim we deduce that 
\begin{align*}
S=\bigcup_{k\in \N,\ \nu\in F} \bigg\{x\in S\mid \forall y\in S\cap B_{\frac{1}{k}}(x): \langle y-x,\nu\rangle\leq \frac{1+\alpha}{2}|y-x|\bigg\}
\end{align*}
and thus it suffices to show that each set of the union is countably $(n-1)$-rectifiable. Then one proves, as in \cite{Villani}, that for each such set the orthogonal projection onto $\nu^\perp$ is locally injective and the inverse is Lipschitz continuous with Lipschitz constant $\frac2{1-\alpha}$. We leave the obvious details and modifications of the proof for the reader.
\end{proof}

\begin{corollary}\label{fdafafafwefwfew2}
    Let $S\subseteq \R^n$ and assume that there exists $r>0$ such that, for each $x\in S$ there exists $t_x>0$ and $v_x\in \R^n$ such that
    \begin{align*}
        x+B_{tr}(tv_x)\subseteq S^c \text{ for all } 0< t\leq t_x.
    \end{align*}
    Then $S$ is countably $(n-1)$-rectifiable.
\end{corollary}

\begin{proof}
For $x\in S$ we clearly have $v_x\neq 0$. Moreover, by setting $S_m:=\{x\in S\mid |v_x|\leq m\}$ and taking the union over $m\in \N$ we can assume that $0\neq |v_x|\leq m$ for all $x\in S$ and some $m\in \N$.

   Let $x\in S$ and let $(y_k)\subseteq S$ be any sequence converging to $x$. Then, as soon as $|y_k-x|\leq t_x$ we clearly have
   \begin{align*}
   \bigg|(y_k-x)-\frac{|y_k-x|}{|v_x|}v_x\bigg|\geq \frac{r|y_k-x|}{|v_x|}\geq  \frac{r|y_k-x|}{m}.
   \end{align*}
   But Lemma \ref{cone} gives that, for $w\in \op{Cone}(\frac{v_x}{|v_x|},\alpha),$ we have
   \begin{align*}
       \bigg|w-|w| \frac{v_x}{|v_x|}\bigg|\leq 2\sqrt{1-\alpha} |w|.
   \end{align*}
   Thus, $(y_k-x)\notin \op{Cone}(\frac{v_x}{|v_x|},\alpha)$ for $\alpha<1$ with $2\sqrt{1-\alpha}< \frac{r}{m}$. But then also
   \begin{align*}
       T_xS\subseteq \R^n\backslash \op{Cone}(\frac{v_x}{|v_x|},\alpha)
       \end{align*}
       and this holds for any $x\in S$. Therefore, using the above theorem, we conclude the corollary.
   \end{proof}

\begin{definition}\rm \label{omns}
Let $M$ be a $n$-dimensional manifold and $S\subseteq M$. We say that $S$ is \emph{countably $(n-1)$-rectifibale} if for any chart $(\phi,U)$ of $M$ the set $\phi(U\cap S)$ is countably $(n-1)$-rectifiable.
\end{definition}

\begin{lemma}\label{mnhas}
    Let $M$ be a $n$-dimensional manifold and $S\subseteq M$. If we can find, around any $x\in S$, a chart $(U,\phi)$ such that $\phi(U\cap S)$ is countably $(n-1)$-rectifiable, then $S$ is countably $(n-1)$-rectifiable.
\end{lemma}
\begin{proof}
    Let $(\phi,U)$ be any chart. Since we are assuming manifolds to be second-countable we find a countable family of chart $(\phi_i,U_i)$ with $S\subseteq \bigcup_{i=1}^\infty U_i$ and such that $\phi_i(U_i\cap S)$ is countably $(n-1)$-rectifiable. Then we have
    \begin{align*}
        \phi(U\cap S)=\bigcup_{i=1}^\infty \phi(U_i\cap U\cap S)
        =\bigcup_{i=1}^\infty (\phi\circ \phi_i^{-1})(\phi_i(U_i\cap U\cap S)).
    \end{align*}
    Since $\phi_i(U_i\cap U\cap S)$ is countably $(n-1)$-rectifiable, contained in the open set $\phi_i(U_i)$ and $\phi\circ \phi_i^{-1}$ is locally Lipschitz on $\phi_i(U_i)$ one readily checks that $(\phi\circ \phi_i^{-1})(\phi_i(U_i\cap U\cap S))$ is countably $(n-1)$-rectifiable too. Thus, also the union is countably $(n-1)$-rectifiable, but this is just $\phi(U\cap S)$.
\end{proof}

\subsection{Riemannian metric on $T^kM$}\label{sasaki}

\begin{definition}\rm
Let $(M,g)$ be a smooth Riemannian manifold. Then the \emph{Whitney sum} of $k$ tangent bundles $TM$ is the vector bundle
\begin{align*}
T^kM:=\bigcup_{x\in M} (T_xM)^k
\end{align*}
with its canonical differentiable structure. 

Let $(x,v):=(x,v_1,...,v_k)\in T^kM$ and let $\alpha:=(c,w_1,...,w_k),\beta:=(\gamma,u_1,...,u_k):I\to T^kM$ be smooth curves with $\alpha(0)=\beta(0)=(x,v)$.
Then we define
\begin{align*}
\langle \dot \alpha(0),\dot \beta(0)\rangle_{(x,v)} :=
\langle \dot c(0),\dot \gamma(0)\rangle_x +\sum_{i=1}^k \Big\langle \frac{Dw_i}{dt}(0),\frac{Du_i}{dt}(0)\Big\rangle_x,
\end{align*}
where $\frac{Dw_i}{dt}$ resp. $\frac{Du_i}{dt}$ denotes the covariant derivative of $w_i$ along $c$ resp. $u_i$ along $\gamma$.
\end{definition}

\begin{remark}\rm
    Using local coordinates it is easy to check that the above expression actually defines a Riemannian metric. In the case $k=1$ this metric is also known as \emph{Sasaki metric}.
\end{remark}

\bibliography{OTaR_090625}
\end{document}